\documentclass[a4paper,reqno,12pt]{amsart}
\usepackage{amsmath,amsfonts,amssymb,amsthm}
\usepackage[utf8]{inputenc}    
\usepackage[T1]{fontenc}       
\usepackage{enumerate}

\usepackage[margin = 1in]{geometry}
\usepackage{mathtools}
\usepackage{mathabx}
\usepackage{bbm}

\usepackage[colorlinks, allcolors = black, breaklinks = true]{hyperref}

\renewcommand{\d}{\ensuremath{\mathrm{d}}}
\newcommand{\R}{\ensuremath{\mathbb{R}}}
\newcommand{\Z}{\ensuremath{\mathbb{Z}}}

\newcommand{\N}{\ensuremath{\mathbb{N}}}

\newcommand{\II}{\ensuremath{\mathcal I}}
\newcommand{\JJ}{\ensuremath{\mathcal J}}

\newcommand{\TT}{\ensuremath{\mathcal T}}
\newcommand{\NN}{\ensuremath{\mathcal N}}
\newcommand{\MM}{\ensuremath{\mathcal M}}
\newcommand{\BB}{\ensuremath{\mathcal B}}
\newcommand{\SSS}{\ensuremath{\mathcal S}}
\newcommand{\AAA}{\ensuremath{\mathcal A}}
\newcommand{\F}{\ensuremath{\mathcal F}}
\newcommand{\eps}{\ensuremath{\varepsilon}}
\newcommand{\verti}[1]{\ensuremath{\left\lvert #1 \right\rvert}}
\newcommand{\vertii}[1]{\ensuremath{\left\lVert #1 \right\rVert}}
\newcommand{\vertiii}[1]{{\left\lvert\kern-0.25ex\left\lvert\kern-0.25ex\left\lvert #1
    \right\rvert\kern-0.25ex\right\rvert\kern-0.25ex\right\rvert}}

\renewcommand{\d}{\ensuremath{{\rm d}}}

\newcommand{\ind}{\ensuremath{\mathbbm{1}}}
\newcommand{\supp}{\ensuremath{\mathrm{supp} \,}}

\newcommand{\floor}[1]{\ensuremath{\left\lfloor #1 \right\rfloor}}

\newcommand{\ignore}[1]{}

\newtheorem{theorem}{Theorem}[section]
\newtheorem{definition}[theorem]{Definition}
\newtheorem{proposition}[theorem]{Proposition}
\newtheorem{remark}[theorem]{Remark}
\newtheorem{lemma}[theorem]{Lemma}
\newtheorem{corollary}[theorem]{Corollary}

\numberwithin{equation}{section}
\begin{document}
%
\title[Stability of receding traveling waves for a free boundary problem]{Stability of receding traveling waves for a fourth order degenerate parabolic free boundary problem}
\thanks{MVG appreciates discussions with Mircea Petrache on a related problem. The authors acknowledge funding from and the kind hospitality of the Fields Institute for Research in Mathematical Sciences in Toronto and the New York University in Abu Dhabi. MVG and SI are grateful to the Courant Institute of the New York University in New York City for hosting them. MVG and NM wish to thank the University of Victoria, BC, for its kind hospitality. MVG also received funding from the University of Michigan at Ann Arbor, the National Science Foundation under Grant No.~NSF DMS-1054115, the Max Planck Institute for Mathematics in the Sciences in Leipzig, and the Deutsche Forschungsgemeinschaft (grant \# 334362478). SI is partially supported by NSERC Discovery grant \# 371637-2014, and NM  is in part supported by NSF grant DMS-1211806. The authors thank the anonymous reviewer for several suggestions that have lead to an improved presentation of this revised version.}
\date{\today}
\author{Manuel~V.~Gnann}
\address[Manuel~V.~Gnann]{Center for Mathematics\newline
Technical University of Munich\newline
Boltzmannstr.~3, 85747 Garching near Munich, Germany}
\author{Slim~Ibrahim}
\address[Slim~Ibrahim]{Department of Mathematics and Statistics\newline
University of Victoria\newline
3800 Finnerty Road, Victoria, B.C., Canada V8P 5C2}
\author{Nader~Masmoudi}
\address[Nader~Masmoudi]{Department of Mathematics\newline
Courant Institute for Research in Mathematical Sciences\newline
New York University\newline
251 Mercer St., New York, NY 10012, U.S.A.}
\begin{abstract}
Consider the thin-film equation $h_t + \left(h h_{yyy}\right)_y = 0$ with a zero contact angle at the free boundary, that is, at the triple junction where liquid, gas, and solid meet. Previous results on stability and well-posedness of this equation have focused on perturbations of equilibrium-stationary or self-similar profiles, the latter eventually wetting the whole surface. These solutions have their counterparts for the second-order porous-medium equation $h_t - (h^m)_{yy} = 0$, where $m > 1$ is a free parameter. Both porous-medium and thin-film equation degenerate as $h \searrow 0$, but the porous-medium equation additionally fulfills a comparison principle while the thin-film equation does not.

\medskip

In this note, we consider traveling waves $h = \frac V 6 x^3 + \nu x^2$ for $x \ge 0$, where $x = y-V t$ and $V, \nu \ge 0$ are free parameters. These traveling waves are receding and therefore describe de-wetting, a phenomenon genuinely linked to the fourth-order nature of the thin-film equation and not encountered in the porous-medium case as it violates the comparison principle. The linear stability analysis leads to a linear fourth-order degenerate-parabolic operator for which we prove maximal-regularity estimates to arbitrary orders of the expansion in $x$ in a right-neighborhood of the contact line $x = 0$. This leads to a well-posedness and stability result for the corresponding nonlinear equation. As the linearized evolution has different scaling as $x \searrow 0$ and $x \to \infty$, the analysis is more intricate than in related previous works. We anticipate that our approach is a natural step towards investigating other situations in which the comparison principle is violated, such as droplet rupture.
\end{abstract}
\maketitle
\tableofcontents
%

\section{Introduction and statement of results\label{sec:intro}}
\subsection{The thin-film equation formulated as a classical free boundary problem}
Consider the free boundary problem for the thin-film equation
\begin{subequations}\label{tfe_free}
\begin{align}
h_t + \left(h h_{yyy}\right)_y &= 0 \quad \mbox{for} \quad t > 0 \quad \mbox{and} \quad y > Y_0(t), \label{tfe}\\
h = h_y &= 0 \quad \mbox{for} \quad t > 0 \quad \mbox{and} \quad y = Y_0(t), \label{bc_1}\\
h_{yyy} &= \frac{\d Y_0}{\d t}(t) \quad \mbox{for} \quad t > 0 \quad \mbox{and} \quad y = Y_0(t). \label{bc_2}
\end{align}
\end{subequations}
The dependent variable $h = h(t,y)$ models the height of a two-dimensional thin viscous film on a one-dimensional flat substrate as a function of the independent variables time $t$ and base point $y$ on the substrate (cf.~\cite{d.1985,odb.1997,beimr.2009}). In the particular case \eqref{tfe_free} one may view the problem as the lubrication approximation of Darcy's flow in the Hele-Shaw cell: We refer to rigorous results in \cite{go.2003,km.2013,km.2015}. Here we assume a droplet with support $\left(Y_0(t),\infty\right)$, where $Y_0(t)$ denotes the free boundary. The boundary conditions \eqref{bc_1} determine the position of the free boundary $Y_0(t)$ (also known as \emph{contact line} or \emph{triple junction}, since liquid, gas, and solid border here) and the slope of the film at this position. We assume that the slope vanishes (\emph{zero contact angle}), which is commonly referred to as \emph{complete (perfect) wetting regime}. Finally, the third condition~\eqref{bc_2} determines how the free boundary evolves. Since \eqref{tfe_free} is in divergence form, one can read of the transport velocity $h_{yyy}$ (in lubrication theory, this is in fact the vertically-averaged horizontal velocity), which by compatibility has to be the same as the velocity of the free boundary $\frac{\d Y_0}{\d t}(t)$. As a consequence, the mass of the droplet is conserved.

\subsection{Special solutions to the thin-film and porous-medium equation}
\subsubsection{Stationary and source-type self-similar solutions to the thin-film equation}
The simplest generic solutions to \eqref{tfe_free} are equilibrium-stationary solutions. Here one may assume a \emph{time-independent} profile $h(t,y) = H(y)$ fulfilling $\frac{\d^3 H}{\d y^3} = 0$ subject to \eqref{bc_1} (the condition \eqref{bc_2} is trivially fulfilled). This leads to the solution $H(y) = C (y - Y_0)^2$ for $y \ge 0$, where $Y_0$ fixes the position of the contact line and $C > 0$ is an arbitrary constant. By shift and scaling one may without loss of generality assume
\begin{equation}\label{equi_stat}
H(y) = y^2 \quad \mbox{for} \quad y \ge 0.
\end{equation}
This is indeed the profile around which the perturbative earlier results in \cite{gko.2008,bgko.2016,j.2015} have been obtained: There unique solutions exist if the initial data is close to $H$ given through \eqref{equi_stat} in sufficiently strong norms. While the analysis in \cite{gko.2008,bgko.2016} is limited to $1+1$-dimensions and a Besov-type norm is employed, the assumptions in \cite{j.2015} are weaker and only say that $\partial_y \sqrt{h_{|t = 0}}$ has to be uniformly close to $\partial_y \sqrt{H(y)} \stackrel{\eqref{equi_stat}}{=} 1$ for $y \ge Y_0$ and the result also applies in higher dimensions.

\medskip

In the case of compactly supported droplets, generic solutions have source-type self-similar shape, i.e., they are of the form
\begin{equation}\label{sourcetype}
h(t,y) = (t+1)^{-\frac 1 5} H(x) \quad \mbox{with} \quad x = (t+1)^{-\frac 1 5} y,
\end{equation}
where the scaled variable $x$ can be read off from \eqref{tfe} under the assumption of conservation of mass. These solutions converge to $M \delta_0$ distributionally as $t \searrow -1$, where $M > 0$ denotes the mass of the droplet. Inserting \eqref{sourcetype} into \eqref{tfe}, we arrive at a fourth-order linear ordinary differential equation (ODE) with constant coefficients, giving the solution
\begin{equation}\label{smythhill}
H(x) = \frac{1}{120} (x^2-X^2)^2 \quad \mbox{for} \quad \verti{x} \le X, \quad \mbox{where} \quad X := \left(\frac{225 M}{2}\right)^{\frac 1 5},
\end{equation}
also known as \emph{Smyth-Hill solution} (cf.~\cite{sh.1988}, by scaling it is possible to assume $X = 1$). A linear stability analysis of \eqref{smythhill}, also discussing higher asymptotics, is contained in \cite{bw.2002,ms.2015} while convergence of weak solutions of \eqref{tfe} to \eqref{smythhill} is proved for instance in \cite{ct.2002,cu.2007,mms.2009,cu.2014} mainly using entropy-dissipation arguments\footnote{In \cite{ct.2002} it is stated that ``strong'' solutions to \eqref{tfe} are considered, although only two spatial derivatives are controlled and uniqueness is not known.}. In \cite{g.2015} existence and uniqueness of solutions to \eqref{tfe_free} for small perturbations of \eqref{smythhill} have been proved including refined asymptotic results.

\subsubsection{Special solutions to the porous-medium equation}
The common feature of \eqref{equi_stat} and \eqref{sourcetype} and the representative analyses presented there is that they have their counterpart in analogous special solutions and analyses for the second-order \emph{porous-medium equation}
\begin{equation}\label{pme}
\partial_t h - \partial_y^2 \left(h^m\right) = 0 \quad \mbox{in} \quad \{h > 0\}, \quad \mbox{where} \quad m > 1.
\end{equation}
Equations~\eqref{tfe} and \eqref{pme} both are degenerate-parabolic, but the second-order porous-medium equation \eqref{pme} additionally fulfills a \emph{comparison principle} while the thin-film equation \eqref{tfe} does not. Equilibrium-stationary solutions for \eqref{pme} are given by profiles of the form $H(y) = C (y-Y_0)^{\frac 1 m}$ for $y \ge Y_0$. Again, by scaling without loss of generality $C = 1$ and $Y_0 = 0$, so that
\begin{equation}\label{equi_pme}
H(y) = y^{\frac 1 m} \quad \mbox{for} \quad y \ge 0.
\end{equation}
Traveling waves to \eqref{pme} can be written as $h(t,y) = H(x)$ where $x := y - V t$ and $V \in \R$ is the velocity of the wave. They can be computed as $H(x) = \left(- V (m-1)/m\right)^{\frac{1}{m-1}} x^{\frac{1}{m-1}}$ where we have assumed $Y_{0|t = 0} = 0$ and necessarily $V < 0$ holds true. By scaling without loss of generality 
\begin{equation}\label{tw_pme}
H(x) = x^{\frac{1}{m-1}} \quad \mbox{for} \quad x \ge 0,
\end{equation}
which is a traveling-wave front. Finally, source-type self-similar solutions have the shape
\begin{subequations}\label{source_pme}
\begin{equation}
h(t,y) = (t+1)^{-\frac{1}{m+1}} H(x) \quad \mbox{with} \quad x = (t+1)^{-\frac{1}{m+1}} y,
\end{equation}
where
\begin{equation}
\begin{aligned}
H(x) &=  \left(\frac{(m-1) \left(X^2 - x^2\right)}{2 m (m+1)}\right)^{\frac{1}{m-1}} \quad \mbox{for} \quad \verti{x} \le X, \\
X &= \left(\frac{2 m (m+1)}{m-1}\right)^{\frac{1}{m+1}} \left(\frac{\Gamma\left(\frac{3m-2}{2m-2}\right)}{\sqrt{\pi} \Gamma\left(\frac{m}{m-1}\right)}\right)^{\frac{m-1}{m+1}} M^{\frac{m-1}{m+1}},
\end{aligned}
\end{equation}
\end{subequations}
and $M > 0$ denotes the mass (by scaling without loss of generality $X = 1$ can be assumed). The solutions in \eqref{source_pme} are the famous \emph{Barenblatt-Pattle} profiles (cf.~\cite{b.1952,p.1959}). Note that a well-posedness theory of perturbations of \eqref{tw_pme} or \eqref{source_pme} has been developed in \cite{k.2016,k.1999}, respectively, also covering higher dimensions, while a linear stability analysis (including higher asymptotics) is presented in \cite{bw.1998,s.2014}. Convergence of general solutions of \eqref{tfe} to \eqref{source_pme} using entropic arguments was studied in \cite{ct.2000}. In fact, the entropy studied in \cite{ct.2002} is a special case of the entropy considered in \cite{ct.2000}. The spatial part of the linear operator in \cite{j.2015} is simply the square of a special case of the corresponding linear operator considered in \cite{k.2016}, while it was observed in \cite{bw.2002,ms.2015,g.2015} that the spatial part of the linear operator is given by $\mathcal P (\mathcal P + 2)$, where $\mathcal P$ is the spatial part of a linearized porous-medium operator. Then the governing idea of the previously cited thin-film papers is to transfer as much knowledge as possible from the porous-medium case to the thin-film case. On the other hand, there are known features of \eqref{tfe} such as de-wetting phenomena or rupture of droplets which violate the comparison principle and therefore require different techniques compared to those applied to \eqref{pme}.

\subsubsection{Traveling-wave solutions to the thin-film equation}
In the present note our goal is to slightly change this perspective and to concentrate on qualitative behavior that has no counterpart in the second-order case \eqref{pme}.  To this end, it is convenient to use a traveling-wave ansatz, i.e., we assume that $h(t,y) = H(x)$, where $x := y - V t$ (the contact point is fixed to $y  = 0$ at time $t = 0$) and $V \in \R$ denotes the velocity of the wave. Using this in \eqref{tfe}, we obtain\footnote{Throughout the paper, we drop unnecessary parentheses so that differential operators act on everything on their right-hand side, whereas derivatives denoted by indices only act on the particular function.}
\[
- V \frac{\d H}{\d x} + \frac{\d}{\d x} H \frac{\d^3 H}{\d x^3} = 0 \quad \mbox{for} \quad x > 0,
\]
which we can integrate once using conditions~\eqref{bc_1}~and~\eqref{bc_2}, so that
\[
\frac{\d^3 H}{\d x^3} = V \quad \mbox{for} \quad x > 0 \quad \mbox{subject to} \quad H = \frac{\d H}{\d x} = 0 \quad \mbox{at} \quad x = 0.
\]
After three integrations we arrive at $H(x) = \frac V 6 x^3 + \nu x^2$ for $x \ge 0$, where $\nu \ge 0$ is a free integration parameter. As we need to have a non-negative and moving profile, we necessarily have $V > 0$, that is, the traveling wave is receding and the thin fluid film de-wets the surface. From now on, we assume the generic situation $\nu > 0$. Using the scaling transformations
\begin{subequations}\label{complete_rescale}
\begin{align}
h &= \frac{36 \nu^3}{V^2} \breve h \quad \mbox{and} \quad H = \frac{36 \nu^3}{V^2} \breve H \quad \mbox{with} \quad \breve H = \breve x^3 + \breve x^2, \\ 
t &= \frac{36 \nu}{V^2} \breve t, \\
y &= \frac{6 \nu}{V} \breve y \quad \mbox{and} \quad Y_0 = \frac{6 \nu}{V} \breve Y_0, \\
x &= \frac{6 \nu}{V} \breve x,
\end{align}
\end{subequations}
we see that for the rescaled quantities $\breve h$, $\breve H$, $\breve t$, $\breve Y$, $\breve Y_0$, and $\check x$ the dependence on $V$ and $\nu$ disappears. Hence, we can assume without loss of generality $V = 6$ and $\nu = 1$, so that
\begin{equation}\label{tw}
H(x) = \begin{cases} x^3 + x^2 & \mbox{for} \quad x \ge 0,\\ 0 & \mbox{for} \quad x < 0. \end{cases}
\end{equation}
The profile \eqref{tw} is one of the simplest examples of a special solution to \eqref{tfe_free} violating the comparison principle: compare to the stationary solution \eqref{equi_stat} at time $t = 0$ and for an arbitrary time $t > 0$. It also has the remarkable feature that the scaling of $H$ is different in the limiting cases $x \searrow 0$ and $x \to \infty$. This sets \eqref{tw} apart from all special solutions presented above: The equilibrium-stationary solution \eqref{equi_stat} has distinct scaling and can be expressed as a power of the corresponding equilibrium-stationary solution in the porous-medium case \eqref{equi_pme}. Obviously, a solution with distinct scaling in $x$ cannot be compactly supported, but nevertheless the Smyth-Hill profile \eqref{smythhill} is simply a power of the Barenblatt-Pattle solution \eqref{source_pme}, having the same scaling at both boundaries $x = \pm X$. The traveling wave \eqref{tw} cannot be expressed as a power of a known special solution to \eqref{pme}.

\medskip

A natural question is to ask whether the traveling wave \eqref{tw} is stable with respect to small perturbations which will be the subject that we are pursuing in what follows. We believe that this is also relevant for another striking phenomenon, that is, thin-film rupture (cf.~\cite{bp.1998,ck.2007} for examples in the case of thin films under the action of van der Waals forces). A situation in which this can happen may be modelled by two traveling waves of the form \eqref{tw} with contact lines at $y = 0$ as $t = 0$, but moving away from one another as time passes. Small perturbations of these  traveling waves should allow for a positive profile $h_{|t = 0}$, but one should expect a topological change of the support as time evolves.

\subsection{Perturbations of traveling waves and linearization}
For studying perturbations of \eqref{tfe_free} around the traveling-wave profile $H$ given in \eqref{tw}, the \emph{von Mises transform}
\begin{equation}\label{vonmises}
h(t,Y(t,x)) = x^3 + x^2 \quad \mbox{for} \quad t, x > 0
\end{equation}
is convenient. Similar transformations have been used in \cite{k.1999} in the context of the porous-medium equation, and in \cite{j.2015,ggko.2014,g.2015,g.2016,gp.2018} in the thin-film case. This transformation automatically fixes the free boundary to the point $x = 0$ in the new coordinates. Furthermore, the traveling wave is given by the (simpler) linear function $Y_{\rm TW}(t,x) = x + 6 t$. The point $Y_0 = Y_{|x =0}$ determines the position of the contact line, so that $\frac{\d Y_0}{\d t} = \partial_t Y_{|x = 0}$ is its velocity. We briefly outline how from \eqref{vonmises} a structurally simpler equation can be derived:

\medskip

First, we differentiate \eqref{vonmises} with respect to time $t$, which yields
\[
h_t + h_y Y_t = 0 \quad \mbox{for} \quad t, x > 0
\]
and upgrades to
\begin{equation}\label{diff_hod2}
h_y Y_t =  (h h_{yyy})_y \quad \mbox{for} \quad t, x > 0
\end{equation}
by making use of the evolution equation \eqref{tfe}. Derivatives transform according to $\partial_y = Y_x^{-1} \partial_x$, so that employing \eqref{vonmises} once more in \eqref{diff_hod2}, we arrive at the nonlinear equation
\begin{equation}\label{nonlinear}
(3 x^2 + 2 x) \partial_t Y = \partial_x (x^3 + x^2) \left(Y_x^{-1} \partial_x\right)^2 Y_x^{-1} (3 x^2 + 2 x) \quad \mbox{for} \quad t,x > 0.
\end{equation}
Note that problem~\eqref{nonlinear} for $Y$ subject to initial conditions $Y_{|t = 0} = Y^{(0)}$ is formally well-posed, that is, no further boundary conditions are necessary. The boundary conditions \eqref{bc_1} and \eqref{bc_2} are automatically fulfilled through the von Mises transform \eqref{vonmises}.

\medskip

For deriving a linear problem associated to \eqref{nonlinear}, we set
\begin{equation}\label{def_v}
Y =: x + 6 t + v,
\end{equation}
where $v$ denotes the perturbation of the traveling wave. Note that $v$ does \emph{not} fulfill any boundary condition as the velocity of the wave might be perturbed as well. The linearization of the left-hand side of \eqref{nonlinear} reads
\begin{equation}\label{lhs}
6 (3 x^2 + 2 x) + (3 x^2 + 2 x) \partial_t v
\end{equation}
and the right-hand side is given by
\begin{equation}\label{rhs}
6 (3 x^2 + 2 x) - \partial_x (x^3 + x^2) \left(v_x \partial_x^2 + \partial_x v_x \partial_x + \partial_x^2 v_x\right) (3 x^2 + 2 x).
\end{equation}
Apparently the first summands in expressions~\eqref{lhs}~and~\eqref{rhs} cancel. Furthermore, we can use  the operator identity
\[
v_x \partial_x^2 + \partial_x v_x \partial_x + \partial_x^2 v_x = \partial_x^3 v - v \partial_x^3,
\]
and the fact that $\partial_x^3 (3 x^3 + 2 x) \equiv 0$. Thus we arrive at the linearized problem for $v$:
\begin{equation}\label{linear_v}
(3 x^2 + 2 x) \partial_t v + \partial_x (x^3 + x^2) \partial_x^3 v (3 x^2 + 2 x) = f \quad \mbox{for} \quad t, x > 0,
\end{equation}
where $f$ denotes a general right-hand side. It is apparent that by setting
\begin{equation}\label{def_u}
u := (3 x^2 + 2 x) v
\end{equation}
and using the fact that $u_{|x = 0} = 0$ if $v$ is bounded at $x = 0$, we can study instead of \eqref{linear_v}
\begin{subequations}\label{linear_u}
\begin{align}
\partial_t u + \AAA u &= f \quad \mbox{for} \quad t, x > 0,\label{lin_pde}\\
u &= 0 \quad \mbox{for} \quad t > 0, \; x = 0,\label{bc_lin}
\end{align}
\end{subequations}
where we may introduce the linear operator
\begin{equation}\label{defa}
\AAA := \partial_x (x^3 + x^2) \partial_x^3 = x^{-1} p(D) + x^{-2} q(D),
\end{equation}
with $D := x \partial_x$ the scaling-invariant (logarithmic\footnote{Note that $D = \partial_s$ if we set $s := \ln x$.}) derivative and where $p(\zeta)$ and $q(\zeta)$ are fourth-order polynomials. As $p(D) = x \partial_x x^3 \partial_x^3$ vanishes on $\left\{x^0, x^0 \ln x, x^1, x^2\right\}$ and $q(D) = x^2 \partial_x x^2 \partial_x^3$ vanishes on $\left\{x^0, x^1, x^1 \ln x, x^2\right\}$, we infer that
\begin{equation}\label{def_pq}
p(\zeta) = \zeta^2 (\zeta-1) (\zeta-2) \qquad \mbox{and} \qquad q(\zeta) = \zeta (\zeta-1)^2 (\zeta-2).
\end{equation}
Notice that differentiating \eqref{bc_lin} with respect to $t$, we get $\partial_t u_{|x = 0} = 0$ and that for sufficiently smooth $u$ with $u_{|x = 0} = 0$ we also have $\AAA u_{|x = 0} = 0$. This makes the condition
\begin{equation}\label{bc_rhs}
f = 0 \quad \mbox{for} \quad t > 0, \quad x = 0
\end{equation}
for the right-hand side $f$ necessary.

\medskip

In the first part of this paper (cf.~\S\ref{sec:lin}), we will concentrate on the study of the linear problem \eqref{linear_u} for given right-hand sides $f : \, (0,\infty)^2 \to \R$ fulfilling \eqref{bc_rhs} and initial data $u_{|t = 0} = u^{(0)} : \, (0,\infty) \to \R$ meeting \eqref{bc_lin}. More precisely, we are going to prove existence and uniqueness of solutions $u$ for every $f$ and $u^{(0)}$ in suitable function spaces. Moreover, $u$ fulfills maximal-regularity estimates in terms of $f$ and the initial data $u^{(0)}$.

\medskip

Note that in the afore-mentioned previous approaches in weighted Hilbert-Sobolev spaces (cf.~\cite{gko.2008,bgko.2016}) and (weighted) $L^p$-spaces (cf.~\cite{j.2015}) the linearization around the equilibrium stationary profile \eqref{equi_stat} (having distinct scaling in $x$) leads to a scaling invariant spatial part of the linear operator. The fact that in \eqref{tw} the term $x^2$ is dominant as $x \searrow 0$ and $x^3$ is dominant as $x \to \infty$ is reflected by the feature that the corresponding linear operators $\partial_x x^3 \partial_x^3 = x^{-1} p(D)$ and $\partial_x x^2 \partial_x^3 = x^{-2} q(D)$ (cf.~\eqref{defa} and \eqref{def_pq}) dominate in the respective limits. For the contribution $x^{-2} q(D)$, originating from the addend $x^2$ in the traveling wave \eqref{tw}, we find the special structure
\[
x^{-2} q(D) = x^2 \partial_x^4 + 2 x \partial_x^3 = (x \partial_x^2)^2,
\]
where $x \partial_x^2$ can be interpreted as a linearized porous-medium operator. This is not surprising, as the addend $x^2$ corresponds exactly to the equilibrium profile \eqref{equi_stat} for which this structure is known from \cite{gko.2008,bgko.2016,j.2015}. However, the contribution $x^{-1} p(D)$, even after multiplication with an arbitrary power of $x$, cannot be written as the square of a degenerate second-order operator of the form $x^{-\alpha} \partial_x x^{\alpha+\beta} \partial x$ with real constants $\alpha$ and $\beta$, so that the strong analogy to the porous-medium equation \eqref{pme} is lost. The core of our linear analysis lies in balancing contributions coming from these two operators.

\subsection{Comparison to works on the thin-film equation with general mobility}
We remark that a more general version of the thin-film equation \eqref{tfe} exists that reads
\begin{equation}\label{tfe_general}
\partial_t h + \partial_y \left(h^n \partial_y^3 h\right) = 0 \quad \mbox{in} \quad \{h > 0\},
\end{equation}
where $n \in (0,3)$ is the exponent of the now nonlinear mobility $h^n$. The restrictions on $n$ come from the fact that for $n \le 0$ solutions to \eqref{tfe_general} can be non-positive and the speed of propagation is infinite, while for $n = 3$, coming from a no-slip condition at the liquid-solid interface, a singularity of $h$ at the free boundary occurs which cannot move unless the dissipation is allowed to be infinite (\emph{no-slip paradox}, cf.~\cite{m.1964,hs.1971,dd.1974}). Of particular interest is the case $n = 2$ (quadratic mobility), which can be derived by means of formal asymptotic expansions from the Navier-Stokes system with Navier-slip at the substrate and film heights $h$ that are small compared to the slip length (cf.~\cite{d.1985,odb.1997,beimr.2009}). A stability analysis for perturbations of traveling waves has been carried out in \cite{ggko.2014,g.2016,gp.2018}, while the linear stationary profile with \emph{partial-wetting} boundary conditions (nonzero equilibrium contact angle) was considered in \cite{k.2011} and for mobility exponents $n \in \left(0,\frac{14}{5}\right) \setminus \left\{2,\frac 5 2, \frac 8 3, \frac{11}{4}\right\}$ in \cite{km.2013,km.2015,k.2015,k.2017}. The regularity of the source-type self-similar solution was discussed in \cite{bpw.1992,ggo.2013,bgk.2016} covering all possible mobility exponents and zero as well as nonzero dynamic contact angles. While the analysis in these situations is more delicate than in the case of linear mobility in \cite{gko.2008,bgko.2016,gk.2010,j.2015,g.2015} due to singular terms appearing for the solution at the contact line, the traveling wave for quadratic mobility is $H(x) = x^{\frac 3 2}$, where the velocity has been normalized to $V = - \frac 3 8$. Although comparison with \eqref{equi_stat} demonstrates that the comparison principle for $H(x) = x^{\frac 3 2}$ is violated as well, this traveling wave still has distinct scaling in $x$ and is wetting the whole surface (as do the source-type self-similar solutions discussed in \cite{bpw.1992,ggo.2013,bgk.2016}). This leads to a linear operator having the same scaling as $x \searrow 0$ and $x \to \infty$ and thus the coercivity and elliptic regularity estimates are simpler compared to the present work.

\medskip

Finally, we anticipate that the techniques developed in this paper are not only applicable to the situation at hand, but that a similar reasoning may be applied to the thin-film equation
\[
\partial_t h + \partial_y \left(\left(h^3 + \lambda h^2\right) \partial_y^3 h\right) = 0 \quad \mbox{in} \quad \{h > 0\}
\]
with partial-wetting boundary conditions (nonzero equilibrium contact angle), where $\lambda > 0$ is the slip length. The latter equation can be derived from the Navier-Stokes system with Navier slip at the liquid-solid interface (cf.~\cite{beimr.2009,d.1985,odb.1997}). In view of \cite{k.2011}, where a quadratic mobility and partial-wetting boundary conditions have been employed, we expect, however, logarithmic corrections to occur, which are not present in the setting considered here.

\subsection{The nonlinear equation\label{sec:nonlin_intro}}
Let us now formulate the nonlinear problem associated to \eqref{linear_u}. We start by writing the nonlinear equation \eqref{nonlinear} in terms of $v = Y - x - 6 t$ (cf.~\eqref{def_v}):
\begin{equation}\label{nonlinear_v}
(3 x^2 + 2 x) \partial_t v + 6 (3 x^2 + 2 x) = \partial_x (x^3 + x^2) \left((1+v_x)^{-1} \partial_x\right)^2 (1+v_x)^{-1} (3 x^2 + 2 x)
\end{equation}
for $t,x > 0$. Using $u = (3 x^2 + 2 x) v$ (cf.~\eqref{def_u}) and equations~\eqref{lin_pde} and \eqref{defa}, we separate into linear and nonlinear parts:
\begin{subequations}\label{nonlinear_u}
\begin{align}
\partial_t u + \AAA u &= \NN(u) \quad \mbox{for} \quad t, x > 0, \label{nonlinear_pde}\\
u &= 0 \quad \mbox{for} \quad t > 0, \, x = 0, \label{bc_nonlin}
\end{align}
\end{subequations}
where the nonlinearity $\NN(u)$ is given by\footnote{Again note that derivatives $\partial_x$ act on \emph{everything} on their right-hand sides.}
\begin{equation}\label{def_nu}
\NN(u) := \partial_x (x^3 + x^2) \left(\left((1+v_x)^{-1} \partial_x\right)^2 (1+v_x)^{-1} (3 x^2 + 2 x) - 6 + \partial_x^3 \left(3 x^2 + 2 x\right) v\right).
\end{equation}
From \eqref{def_nu} we infer that the nonlinearity $\NN(u)$ is a linear combination (with constant coefficients) of terms of the form\footnote{Here, the derivatives $\partial_x$ only act on the factors separated by $\times$.}
\begin{subequations}\label{term_cond_non}
\begin{equation}\label{term_non}
(1+v_x)^{-3-s^\prime} \times \partial_x^{s_0^\prime+1} \left(x^3 + x^2\right) \times \partial_x^{s_0} \left(x^3 + x^2\right) \times \bigtimes_{j = 1}^n \partial_x^{s_j+1} v,
\end{equation}
with
\begin{equation}\label{cond_coeff_non}
s_0^\prime + s_0 + s_1 + \ldots + s_n = 3, \quad s_0 \le 1, \quad n \in \{2,\ldots,6\}, \quad s^\prime := \# \{s_j : \, j \ge 1 \mbox{ and } s_j \ge 1\},
\end{equation}
\end{subequations}
where $\# M$ denotes the cardinality of a finite set $M$ and where $u$ and $v$ are related through \eqref{def_u}. Indeed, in \eqref{def_nu} three derivatives $\partial_x$ need to be distributed on the individual factors, leading to $s_0^\prime + s_0 + s_1 + \ldots + s_n = 3$, at most one can act on the first factor $x^3+x^2$, so that $s_0 \le 1$, and at least one acts on the other factor $x^3 + x^2$ to give $3 x^2 + 2 x$, leading to the expression $\partial_x^{s_0^\prime+1} \left(x^3 + x^2\right)$. Since $\NN(u)$ has no contribution that is constant or linear in $v$, we have $n \ge 2$. The constant $s^\prime$ simply counts the number of times a derivative $\partial_x$ acts on a factor $(1+v_x)^{-j}$, where $j \ge 3$, since only then a new factor $v_{xx}$ is generated and the exponent $-j$ decreases by $1$. Because we have $3 + s^\prime \le 6$, we obtain the upper bound $n \le 6$ after subtracting the contributions that are constant and linear in $v$. 

\medskip

The nonlinear problem \eqref{nonlinear_u} and in particular estimates on the nonlinearity \eqref{def_nu} will be the subject of \S\ref{sec:nonlin}.
\subsection{Outline and statement of results\label{sec:out}}
The paper is structured as follows:

\medskip

The linear theory is contained in \S\ref{sec:lin}. In \S\ref{sec:coerc} we first discuss the coercivity properties of the linear operator $\AAA$. This requires joint coercivity of its summands in suitably chosen weighted inner products
\begin{subequations}\label{def_in_norm}
\begin{equation}\label{def_base_in}
(u,v)_\alpha := \int_0^\infty x^{-2 \alpha} u v \frac{\d x}{x} \quad \mbox{and} \quad (u,v)_{k,\alpha} := \sum_{j = 0}^k \left(D^j u, D^j v\right)_\alpha,
\end{equation}
with corresponding norms
\begin{equation}\label{def_base_norm}
\verti{u}_\alpha := \sqrt{(u,u)_\alpha} \quad  \mbox{and} \quad |u|_{k,\alpha} := \sqrt{(u,u)_{k,\alpha}},
\end{equation}
\end{subequations}
where $k \in \N_0$ and $\alpha \in \R$. Note that increasing the number $k$ in \eqref{def_in_norm} does not lead to more regularity of $u$ at $x = 0$ since every additional derivative $\partial_x$ is multiplied with a factor $x$. However, $\verti{u}_{1,\alpha} < \infty$ implies $u = o\left(x^\alpha\right)$ as $x \searrow 0$, that is, the larger $\alpha$, the stronger the decay of $u$ as $x \searrow 0$. Loosely speaking, the constant $k$ determines the regularity in the interior, while $\alpha$ measures the regularity at the boundary $x = 0$.

\medskip

In fact, we will not show coercivity estimates directly for $\AAA$ but for differential operators $\tilde \AAA$ and $\check \AAA$ fulfilling the commutation relations
\begin{equation}\label{commute_a_op}
(D-1) \AAA = \tilde \AAA (D-1) \quad \mbox{and} \quad (D-2) \tilde \AAA = \check \AAA (D-2).
\end{equation}
Indeed, for right-hand sides $f$ that are smooth in $x$ for positive times, the solution $u$ to \eqref{linear_u} turns out to be a smooth function in $x$ for positive times, that is,
\begin{equation}\label{formal_exp_u}
u(t,x) = u_1(t) x + u_2(t) x^2 + u_3(t) x^3 + \ldots \quad \mbox{as} \quad x \searrow 0.
\end{equation}
The formal expansion \eqref{formal_exp_u} suggests $\verti{u}_\alpha < \infty$ for $\alpha < 1$ which automatically restricts the set of admissible exponents $\alpha$ in order to have coercivity of $\AAA$, the \emph{coercivity range} of $\AAA$, to be empty.  Applying the operators $(D-1)$ and $(D-2) (D-1)$ to \eqref{lin_pde}, respectively, we obtain
\begin{subequations}\label{eq_ta_ca}
\begin{align}
\partial_t \tilde u + \tilde \AAA \tilde u &= \tilde f \quad \mbox{for} \quad t, x > 0, \label{eq_ta}\\
\partial_t \check u + \check \AAA \check u &= \check f \quad \mbox{for} \quad t, x > 0, \label{eq_ca}
\end{align}
\end{subequations}
where we have set
\begin{equation}\label{def_tc}
\tilde w := (D-1) w, \quad \check w := (D-2) \tilde w = (D-2) (D-1) w
\end{equation}
for a locally integrable function $w : (0,\infty) \to \R$. Thus, in view of \eqref{formal_exp_u} we have
\begin{subequations}\label{formal_exp_tu_cu}
\begin{align}
\tilde u(t,x) &= u_2(t) x^2 + 2 u_3(t) x^3 + 3 u_4(t) x^4 + \ldots \quad \mbox{as} \quad x \searrow 0, \label{formal_exp_tu}\\
\check u(t,x) &= 2 u_3(t) x^3 + 6 u_4(t) x^4 + 12 u_5(t) x^5 + \ldots \quad \mbox{as} \quad x \searrow 0, \label{formal_exp_cu}
\end{align}
\end{subequations}
that is, norms $\verti{\tilde u}_{\tilde \alpha}$ with $\tilde \alpha < 2$ and $\verti{\check u}_{\check \alpha}$ with $\check \alpha < 3$ are finite, so that not surprisingly coercivity estimates for larger weights for the operators $\tilde \AAA$ and $\check \AAA$ hold true (cf.~Lemma~\ref{lem:coerc_ta} and Lemma~\ref{lem:coerc_ca}).

\medskip

By accessible arguments, which mainly use the additive structure of $\AAA$ given by \eqref{defa}, coercivity of $\tilde \AAA$ or $\check \AAA$ implies parabolic maximal regularity of \eqref{eq_ta_ca}, respectively. This is discussed in \S\ref{sec:maxreg} (cf.~Proposition~\ref{prop:maxreg1}). In \S\ref{sec:elliptic} we additionally prove that corresponding estimates for $\tilde u$ and $\check u$ imply estimates for $u$ or $u - u_1 x$ or $u - u_1 x - u_2 x^2$, respectively. This is a consequence of Hardy's inequality (cf.~Proposition~\ref{prop:ell_reg}). The resulting estimates control the expansion of the solution $u$ to \eqref{linear_u} in the sense of
\[
u(t,x) = u_1(t) x + u_2(t) x^2 + o(x^2) \quad \mbox{as} \quad x \searrow 0,
\]
where $u_1\in BC^0([0,\infty))$, and $u_2 \in L^2((0,\infty))$. Due to \eqref{def_u} the corresponding parabolic estimates only imply control of the norm $\sup_{t,x > 0} \verti{v(t,x)}$. However, such an estimate appears \emph{not} to be sufficient in order to treat the corresponding nonlinear problem for $v$ as the validity of transformations~\eqref{vonmises} and \eqref{def_v} can only be ensured for functions $v$ with small Lipschitz norm. This can be also seen by taking the structure of the nonlinearity as in \eqref{term_non} into account since factors $(1+v_x)^{-1}$ cannot be controlled otherwise. Hence, we require in particular to control the expansion \eqref{formal_exp_u} with $u_2 \in BC^0([0,\infty))$ which corresponds to $u_3 \in L^2((0,\infty))$. While one may expect to achieve this by applying the operator $D-3$ to \eqref{eq_ca}, corresponding manipulations as in \eqref{commute_a_op} and \eqref{def_tc} do \emph{not} yield a coercive operator. Observe that unlike $u_1 x$ and $u_2 x^2$, the monomial $u_3 x^3$ is not contained in the kernel of $\AAA$ (cf.~\eqref{defa} and \eqref{def_pq}) and has to be controlled using a different reasoning. Notice that applying the operator $\tilde \AAA$ to \eqref{eq_ta}, we get
\[
\partial_t \left(\tilde \AAA \tilde u\right) + \tilde \AAA \left(\tilde \AAA \tilde u\right) = \left(\tilde \AAA \tilde f\right) \quad \mbox{for} \quad t, x > 0,
\]
that is, the tuple $\left(\tilde \AAA \tilde u, \tilde \AAA \tilde f\right)$ fulfills the same equation as $\left(\tilde u, \tilde f\right)$ and since $\tilde u_{x = 0} = 0$, also $\tilde \AAA \tilde u_{|x = 0} = 0$ in view of \eqref{defa} and \eqref{commute_a_op}. Therefore, corresponding estimates for $\left(\tilde \AAA \tilde u, \tilde \AAA \tilde f\right)$ hold true. However, this again requires to pass from norms in $\tilde \AAA \tilde u$ to estimates in $u$, i.e., we require knowledge about the elliptic regularity of $\tilde \AAA$. Note that because of \eqref{defa} and \eqref{def_base_norm} the operator $\tilde\AAA$ scales like $x^{-2}$ when $x \searrow 0$, so that by this method indeed stronger control of the solution $u$ to \eqref{linear_u} at $x = 0$ can be expected. We conclude this part by discussing the resulting maximal-regularity estimates of \eqref{linear_u} that appear to be sufficient for the treatment of the full problem for $v$ in \S\ref{sec:maxreg2} (cf.~Proposition~\ref{prop:mr_full}). Applying elliptic regularity here also requires arguments relying on the polynomial equation originating from inserting a power series of $u$ and $f$ in form of \eqref{formal_exp_u} into the linear equation \eqref{lin_pde}. In \S\ref{sec:norms}--\S\ref{sec:lin_rig} we finally demonstrate how the presented arguments can be made rigorous using the resolvent equation (cf.~Proposition~\ref{prop:resolvent}) and a time-discretization argument (cf.~Proposition~\ref{prop:linear_existence} and Proposition~\ref{prop:linear_uniqueness}). We remark that \S\ref{sec:lin_res}--\S\ref{sec:lin_rig} do not contain new methods and are included in the paper for the sake of completeness. Furthermore, we also remark that a semi-group approach in \S\ref{sec:lin_rig} is equally-well possible.

\medskip

In \S\ref{sec:nonlin} we present the nonlinear estimates connected to \eqref{nonlinear_u}. In \S~\ref{sec:nonlin_main} we give an overview of the main results, in particular the main nonlinear estimates leading to a well-posedness result for \eqref{nonlinear_u}. This is based on structural observations on the nonlinearity (cf.~\S\ref{sec:nonlin_struct}) and control of $v$ in $C^0$-based norms (cf.~\S\ref{sec:lip_control}, Lemma~\ref{lem:lipschitz} and Lemma~\ref{lem:sup_control}). These arguments are put together in \S\ref{sec:nonlin_main_est} thus concluding the paper and leading to our main result, Theorem~\ref{th:main}, for which we state a simplified version of it already at this stage:
\begin{theorem}\label{th:main_simple}
For any $0 < \delta < \frac 1 2$ there exists $\eps > 0$ such that for all locally integrable initial data $u^{(0)} : \, (0,\infty) \to \R$ with
\begin{equation}\label{norm_init_main}
\vertiii{u^{(0)}}_{\mathrm{init}}^2 := \verti{u^{(0)}}_{8,\delta}^2 + \verti{u^{(0)} - u^{(0)}_1 x}_{8,1+\delta}^2 + \verti{u^{(0)} - u^{(0)}_1 x - u^{(0)}_2 x^2}_{8,2+\delta}^2 \le \eps^2,
\end{equation}
the nonlinear problem \eqref{nonlinear_u} with initial condition $u_{|t = 0} = u^{(0)}$ has exactly one locally integrable solution in a suitably chosen norm (cf.~\eqref{norm_sol} with $N = 1$ and $k = 3$). This solution fulfills the a-priori estimate $\sup_{t \ge 0} \vertiii{u(t)}_{\mathrm{init}} \lesssim_\delta \vertiii{u^{(0)}}_{\mathrm{init}}$ and meets the expansion
\[
u(t,x) = u_1(t) x + u_2(t) x^2 + R(t,x) x^3 \quad \mbox{as} \quad x \searrow 0.
\]
This also implies
\begin{enumerate}[(a)]
\item $u_1 \in BC^1\left([0,\infty)\right)$ with $\sup_{t \ge 0} \verti{\frac{\d^\ell u_1}{\d t^\ell}} \lesssim_\delta \vertiii{u^{(0)}}_{\mathrm{init}}$ for $\ell = 0, 1$,
\item $u_2 \in BC^0\left([0,\infty)\right)$ with $\sup_{t \ge 0} \verti{u_2} \lesssim_\delta \vertiii{u^{(0)}}_{\mathrm{init}}$,
\item $R = R(t,x) \in L^2\left((0,\infty);BC^0([0,\infty))\right)$ with $\int_0^\infty \sup_{x \ge 0} \verti{R(t,x)}^2 \d t\lesssim_\delta \vertiii{u^{(0)}}_{\mathrm{init}}^2$.
\end{enumerate}
Furthermore, we have $\vertiii{u(t)}_{\mathrm{init}} \to 0$ as $t \to \infty$, that is, the traveling wave \eqref{tw} is asymptotically stable.
\end{theorem}
Note that the norm \eqref{norm_init_main} is quasi-minimal in the sense that taking weights up to $2+\delta$ just gives control of the expansion of $u$ up to $u_2$ in $BC^0$ in \eqref{formal_exp_u} which is the critical scaling such that Lipschitz control of $v$ (cf.~\eqref{def_v}) can be obtained.

\medskip

Higher-regularity results and estimates on the coefficients $u_j$ will be presented later on as well (cf.~Theorem~\ref{th:main} and Corollary~\ref{cor:main}). We also note that by parabolic regularity theory, the function $u$ is smooth in the interior $\{(t,x): \, t > 0 \mbox{ and } x > 0\}$. Furthermore, smallness of $\sup_{t \ge 0} \vertiii{u(t)}_{\mathrm{init}}$ implies by virtue of \eqref{vonmises}, \eqref{def_v}, and \eqref{def_u} that such defined $h$ solves \eqref{tfe_free} and is smooth in $\{h > 0\}$. Additionally, we obtain from \eqref{vonmises}, \eqref{def_v}, and \eqref{def_u} that almost everywhere in time $t > 0$
\begin{align*}
& x + \frac 1 2 \frac{u_1 + u_2 x + R x^2}{1 + \frac 3 2 x} = y - 6 t \quad \mbox{as} \quad x \searrow 0\\
&\Rightarrow \quad x = \frac{y - 6 t - \frac 1 2 u_1}{1 + \frac 1 2 u_2 - \frac 3 4 u_1} \left(1 - \frac{\frac 1 2 \bar R - \frac 3 4 u_2 + \frac 9 8 u_1}{\left(1 + \frac 1 2 u_2 - \frac 3 4 u_1\right)^2} \left(y - 6 t - \frac 1 2 u_1\right)\right) \quad \mbox{as} \quad y \searrow 6 t + \frac 1 2 u_1,
\end{align*}
where $\bar R = \bar R\left(t,y -6 t - \frac 1 2 u_1\right)$ and $\bar R = \bar R\left(t,\bar y\right) \in L^2\left((0,\infty);BC^0([0,\infty))\right)$ fulfills an a-priori estimate in terms of $\vertiii{u^{(0)}}_{\mathrm{init}}$. Inserting this into \eqref{vonmises}, we infer that $h$ meets the expansion
\[
h = \left(\frac{y - 6 t - \frac 1 2 u_1}{1 + \frac 1 2 u_2 - \frac 3 4 u_1}\right)^2 + \frac{1 - \bar R + 2 u_2 - 4 u_1}{1 + \frac 1 2 u_2 - \frac 3 4 u_1} \left(\frac{y - 6 t - \frac 1 2 u_1}{1 + \frac 1 2 u_2 - \frac 3 4 u_1}\right)^3 \quad \mbox{as} \quad y \searrow 6 t + \frac 1 2 u_1
\]
almost everywhere in $t > 0$.

\begin{remark}
Scaling back $V$ and $\nu$ according to \eqref{complete_rescale}, we see that because of \eqref{def_v} and \eqref{def_u} we have
\begin{equation}\label{rescale_v_u}
\breve v = \frac{V}{6 \nu} v \quad \mbox{and} \quad \breve u = \frac{V^2}{36 \nu^3} u, \quad \mbox{where} \quad u := \left(3 \frac V 6 x^2 + 2 \nu x\right) v.
\end{equation}
As Theorem~\ref{th:main_simple} is formulated for the rescaled quantities $\breve u$ and $\breve x$, the condition \eqref{norm_init_main} on the initial data reads in the original (unrescaled) variables
\begin{equation}\label{cond_init_un}
\left(\frac{V}{6 \nu}\right)^4 \verti{u^{(0)}}_{8,\delta}^2 + \left(\frac{V}{6 \nu}\right)^2 \verti{u^{(0)} - u_1^{(0)} x}_{8,1+\delta}^2 + \verti{u^{(0)} - u_1^{(0)} x - u_2^{(0)} x^2}_{8,2+\delta}^2 \le \nu^2 \left(\frac{V}{6 \nu}\right)^{2\delta} \eps^2.
\end{equation}
%
\begin{enumerate}[(a)]
\item\label{num:q1} The limit $V \searrow 0$ corresponds to the equilibrium-stationary solution \eqref{equi_stat}, which was treated already in \cite{gko.2008,j.2015,bgko.2016}. From \eqref{cond_init_un} we see that the initial data has to be of order $O\left(V^{2 \delta}\right)$ as $V \searrow 0$, where $\delta > 0$ can be chosen arbitrarily small. Note, however, that $\eps$ is a function of $\delta$ as well and that $\eps \searrow 0$ as $\delta \searrow 0$, because coercivity of the linear operator ceases to hold for $\delta = 0$ (cf.~\eqref{coercivity_a}). An in $V$ (quasi-)uniform result would therefore require to explicitly characterize the dependence of $\eps$ on $\delta$.
\item\label{num:q2} The limit $\nu \searrow 0$ corresponds to a non-generic situation in which instead of \eqref{tw} the traveling wave
\begin{equation}\label{tw_alt}
\breve H = \begin{cases} \breve x^3 & \mbox{ for } \breve x \ge 0,\\ 0 & \mbox{ for } \breve x < 0 \end{cases}
\end{equation}
is considered. Using instead of \eqref{vonmises} the transform
\[
\breve h\left(\breve t,\breve Y\left(\breve t,\breve x\right)\right) = \breve x^3 \quad \mbox{for} \quad \breve t, \breve x > 0
\]
and defining $\breve v$ as in \eqref{def_v}, we need to set $\breve u := 3 \breve x^2 \breve v$ instead of using the definition \eqref{def_u}. This gives corresponding linear and nonlinear equations \eqref{lin_pde} and \eqref{nonlinear_pde}, where the operator $\breve \AAA$ in \eqref{defa} reduces to $\breve \AAA = \breve x^{-1} p\left(\breve x \partial_{\breve x}\right)$ and the nonlinearity \eqref{def_nu} attains the simpler form
\[
\breve \NN\left(\breve u\right) := \partial_{\breve x} \breve x^3 \left(\left(\left(1+\breve v_{\breve x}\right)^{-1} \partial_{\breve x}\right)^2 \left(1+\breve v_{\breve x}\right)^{-1} 3 \breve x^2 - 6 + \partial_{\breve x}^3 3 \breve x^2 \breve v\right).
\]
Because of $\breve u = 3 \breve x^2 \breve v$, the boundary conditions \eqref{bc_lin} and \eqref{bc_nonlin} need to be replaced by
\[
\breve u = \partial_{\breve x} \breve u = 0 \quad \mbox{for} \quad \breve t > 0, \, \breve x = 0.
\]
It is an open problem whether this problem (including the additional boundary condition) is well-posed. However, we notice that the right-hand side of \eqref{cond_init_un} vanishes as $\nu \searrow 0$. This is not surprising, as because of $\breve u = 3 x^2 \breve v$ boundedness of $\vertii{\breve u}_{\mathrm{init}}$ only implies control of $\vertii{\breve v}_{L^\infty_{\breve x}}$ but not $\vertii{\breve v_{\breve x}}_{L^\infty_{\breve x}}$. And while Theorem~\ref{th:main} would allow for larger weights that remove this issue, we would still need $\verti{u^{(0)}}_{8,\delta} = O\left(\nu^{3-\delta}\right)$, $\verti{u^{(0)} - u_1^{(0)} x}_{8,1+\delta} = O\left(\nu^{2-\delta}\right)$, and $\verti{u^{(0)} - u_1^{(0)} x - u_2^{(0)} x^2}_{8,2+\delta} = O\left(\nu^{1-\delta}\right)$ as $\nu \searrow 0$, making the choice of the norm in \eqref{norm_init_main} or \eqref{norm_init} unnatural for this situation. In fact our analysis relies in part on a gain of regularity due to the operator $\breve x^{-2} q\left(\breve x \partial_{\breve x}\right)$, which is precisely the addend in $\breve \AAA$ (cf.~\eqref{defa}) that is not present anymore for $\nu = 0$. 
\end{enumerate}
We believe that the questions raised in \eqref{num:q1} and \eqref{num:q2} are interesting directions to pursue in future research. In particular, we expect that the methods developed in this work can be adapted to address these scaling limits by employing a suitable interpolation norm built from the norms used here. This is in fact also a motivation why in the present work we do not rely on control coming from the operator $x^{-2} q(D)$ only and treating $x^{-1} p(D)$ as a perturbation, which would be more in line with the approaches in \cite{gko.2008,j.2015,bgko.2016}.
\end{remark}

\subsection{Notation}
For $f, g \ge 0$ and a finite set $A$, we write $f \lesssim_A g$ or  equivalently $g \gtrsim_A f$, whenever a constant $C = C(A)$, only depending on $A$, exists such that $f \le C g$. In this case, we say that $f$ can be estimated (or bounded) by $g$. If $f \lesssim_A g$ and $g \lesssim_A f$, we write $f \sim_A g$. For a finite set $A$ and a real variable $x$, we say that a property $P(x)$ is true for $x \gg_A 1$ (or $x \ll_A 1$) if where exists a constant $C = C(A) > 0$ such that $P(x)$ is true for $x \ge C$ (or $x \le C^{-1}$). In this case we say that $P(x)$ is true for sufficiently large (small) $x$. Furthermore, if the constant $C$ is universal or its dependence is specified in the context, we simply write $f \lesssim g$ etc. For any function $w\in C^N([0,\infty))$, we write
\[
w(x)=w_0+w_1x+w_2x^2+\ldots+w_Nx^N+o\left(x^N\right) \quad \mbox{as} \quad x \searrow 0 \quad \Leftrightarrow \quad w_j = \frac{1}{j!} \frac{\d^j w}{\d x^j}(0),
\]
where $j = 0,\ldots,N$. As noted in the context of \eqref{term_cond_non}, we write $E_1 \times \ldots \times E_M$ for a product, where the $E_j$ are expressions of the form $E_j = \prod_{\ell = 1}^{m} D^{s_{j,\ell}} w_{j,\ell}$ with sufficiently regular $w_{j,\ell} = w_{j,\ell}(x)$ and where differential operators $D$ act on everything to their right-hand side within $E_j$.
\section{Linear theory\label{sec:lin}}
\subsection{Coercivity\label{sec:coerc}}
We begin by introducing some terminology and making preparatory observations in \S\ref{sec:coerc}. Coercivity estimates to be used at later stages are the subject of \S\ref{sec:coerc_comm}.
\subsubsection{Definition of norms and preliminary remarks}
In this section, we repeat some of the notions introduced in \cite[\S4 and \S5]{ggko.2014}. First, we introduce a scale of inner products and norms given by
\begin{eqnarray}\nonumber
\verti{u}_{k,\alpha}^2 := (u,u)_{k,\alpha}, \quad \mbox{where} \quad (u,v)_{k,\alpha} &:=& \sum_{j = 0}^k \int_0^\infty x^{-2 \alpha} (D^j u) (D^j v) \frac{\d x}{x} \\
&\stackrel{\eqref{def_base_in}}{=}& \sum_{j = 0}^k \left(D^j u, D^j v\right)_\alpha, \label{def_prod}
\end{eqnarray}
where $u, v \colon (0,\infty) \to \R$ are locally integrable such that the expressions in \eqref{def_prod} are finite. We remark that using the transformation $s:=\ln x$, we have
\[
(u,v)_{k,\alpha} = \sum_{j = 0}^k \int_\R e^{-2 \alpha s} (\partial_s^j u) (\partial_s^j v) \, \d s \sim \left(\breve u, \breve v\right)_{W^{k,2}(\R)},
\]
where $\breve u := e^{-\alpha s} u$, $\breve v := e^{-\alpha s} v$ and $W^{k,2}(\R)$ is the standard unweighted Hilbert-Sobolev space on $\R$.

\medskip

We are interested in the coercivity of the operator $\AAA$, that is, we would like to know if
\begin{equation}\label{coercivity}
\left<u, \AAA u\right>_1 \gtrsim \left<u,u\right>_2 \quad \mbox{for all locally integrable } \, u \, \mbox{ with } \, \left<u,u\right>_2 < \infty,
\end{equation}
where $\langle\cdot,\cdot\rangle_j$ are scalar products to be specified in what follows. Therefore, we may use the structure of $\AAA$ given by its definition in \eqref{defa} as a sum of two operators $x^{-1} p(D)$ and $x^{-2} q(D)$ (again, note that these operators have a distinct scaling in $x$, $\sim x^{-1}$ or $\sim x^{-2}$, respectively). For $u$ locally integrable such that $\verti{u}_{2,\alpha} < \infty$ we have
\begin{equation}\label{sum_coerc}
\left(u,\AAA u\right)_\alpha = \left(u, p(D) u\right)_{\alpha+\frac 1 2} +  \left(u, q(D) u\right)_{\alpha+1}.
\end{equation}
Equation~\eqref{sum_coerc} shows that the study of coercivity of the linear operator $\AAA$ can be reduced to the study of coercivity of an operator $P(D)$, where $P(\zeta)$ is a fourth-order polynomial
\begin{equation}\label{def_p_op}
P(\zeta) = \prod_{j = 1}^4 (\zeta-\gamma_j) \quad \mbox{with} \quad \gamma_1 \le \gamma_2 \le \gamma_3 \le \gamma_4.
\end{equation}
Observe that by employing $\breve u := e^{-\alpha s} u$ and the Fourier transform $\F \breve u$ of $\breve u$, we obtain by Plancherel's theorem
\begin{equation}\label{fourier}
\left(u, P(D) u\right)_\alpha = \int_\R e^{-2\alpha s} u P(\partial_s) u \, \d s = \int_\R \breve u P(\partial_s+\alpha) \breve u \, \d s = \int_\R \Re P(i \xi+\alpha) \verti{\F \breve u}^2 \, \d \xi.
\end{equation}
In the last equality we have used that the integral $(u, P(D) u)_\alpha$ is real. Now the operator $P(D)$ is coercive in the sense of
\begin{equation}\label{coerc_p}
\left(u,P(D) u\right)_\alpha \gtrsim_\alpha \verti{u}_{2,\alpha}^2 \quad \mbox{for all locally integrable } \, u \, \mbox{ with } \, \verti{u}_{2,\alpha} < \infty
\end{equation}
if and only if we can bound the fourth-order polynomial $\Re P(i \xi+\alpha)$ by a positive constant from below. The range of $\alpha$ for which \eqref{coerc_p} is satisfied will be called the \emph{coercivity range} of $P(D)$. In \cite[Prop.~5.3]{ggko.2014} by explicit computation a criterion for coercivity was derived:
\begin{lemma}\label{lem:coer_p}
The differential operator $P(D)$ given in \eqref{def_p_op} is coercive with respect to $(\cdot,\cdot)_\alpha$ in the sense of \eqref{coerc_p} if the weight $\alpha$ fulfills
\begin{subequations}\label{coer_cond}
\begin{align}
\alpha & \in (-\infty,\gamma_1) \cup (\gamma_2, \gamma_3) \cup (\gamma_4, \infty), \\
\alpha & \in \left(m(\gamma)-\frac{\sigma(\gamma)}{\sqrt 3}, m(\gamma) + \frac{\sigma(\gamma)}{\sqrt 3}\right),
\end{align}
\end{subequations}
where
\[
m(\gamma) := \frac 1 4 \sum_{j = 1}^4 \gamma_j \quad \mbox{and} \quad \sigma(\gamma) = \sqrt{\frac 1 4 \sum_{j = 1}^4 (\gamma_j - m(\gamma))^2}
\]
are mean and square root of the variance of the zeros $\gamma_j$, respectively.
\end{lemma}
We conclude from \eqref{def_pq} and \eqref{coer_cond} that the operator $p(D)$ is coercive for all $\alpha_1$ satisfying
\begin{equation}\label{alpha1}
\alpha_1 \in \left(-\frac 1 4 \sqrt{\frac{11}{3}} + \frac 3 4, + \frac 3 4 + \frac 1 4 \sqrt{\frac{11}{3}}\right) \cap (0,1) = \left(\frac 3 4 - \frac 1 4 \sqrt{\frac{11}{3}}, 1\right).
\end{equation}
The criterion \eqref{coer_cond} on the other hand does not yield coercivity for $q(D)$. However, we obtain that $q(D)$ is a non-negative operator with respect to $\left(\cdot,\cdot\right)_{\alpha_2}$ provided that we choose $\alpha_2 := 1$. We can even quantify non-negativity in this case:
\begin{equation}\label{q_nonneg}
\begin{aligned}
\left(u, q(D) u\right)_1 &\stackrel{\eqref{fourier}}{=} \int_\R q(i\xi+1) \verti{\F \breve u}^2 \d\xi = \int_\R (\xi^2 + \xi^4) \verti{\F \breve u}^2 \d \xi = \int_\R \left((\partial_s \breve u)^2 + (\partial_s^2 \breve u)^2\right) \d s \\
&\; = \verti{(D-1) u}_1^2 + \verti{(D-1)^2 u}_1^2.
\end{aligned}
\end{equation}
Due to \eqref{sum_coerc} and \eqref{q_nonneg} we are forced to take $\alpha = 0$, so that $\alpha_1 = \frac 1 2$ (one may verify by using \eqref{alpha1} that indeed $\alpha_1 = \frac 1 2$ is admissible) and \eqref{sum_coerc} yields
\begin{equation}\label{coercivity_a}
\begin{aligned}
(u,\AAA u)_0 &\; = (u,p(D)u)_{\frac 1 2} +  (u,q(D) u)_1 \\
&\stackrel{\eqref{coercivity}}{\sim} 
\underbrace{\verti{u}_{\frac 1 2}^2 + \verti{D u}_{\frac 1 2}^2 + \verti{D^2 u}_{\frac 1 2}^2}_{= \verti{u}_{2,\frac 1 2}^2} +  \left(\verti{(D-1) u}_1^2 + \verti{(D-1)^2 u}_1^2\right).
\end{aligned}
\end{equation}
The constant in estimate~\eqref{coercivity_a} is universal. Note that estimate~\eqref{coercivity_a} is inconvenient for all subsequent arguments. To see this, observe
\[
\verti{(D-1) u}_1^2 = \int_0^\infty x^{-2} ((D-1) u)^2 \frac{\d x}{x} = \int_0^\infty x^2 \left(\partial_x (x^{-1} u)\right)^2 \frac{\d x}{x}.
\]
While it is true that $\partial_x (x^{-1} u) \stackrel{\eqref{formal_exp_u}}{=} \partial_x \left(x^{-1} u - u_1\right) = u_2 x + u_3 x^2 + \cdots$ as $x \searrow 0$ and the integrals $\int_0^1 \left(x^{-1} u - u_1\right)^2 \frac{\d x}{x}$ and $\int_1^\infty \left(x^{-1} u\right)^2 \frac{\d x}{x}$ are finite, the estimates
\begin{subequations}\label{hardy_crit}
\begin{align}
\int_0^\infty x^2 \left(\partial_x (x^{-1} u - u_1)\right)^2 \frac{\d x}{x} &\gtrsim \int_0^\infty \left(x^{-1} u - u_1\right)^2 \frac{\d x}{x} = \verti{u - u_1 x}_1^2, \label{hardy_crit_1}\\
\int_0^\infty x^2 \left(\partial_x (x^{-1} u - u_1)\right)^2 \frac{\d x}{x} &\gtrsim \int_0^\infty \left(x^{-1} u\right)^2 \frac{\d x}{x} = \verti{u}_1^2 \label{hardy_crit_2}
\end{align}
\end{subequations}
correspond to the critical-case Hardy inequalities and are known to fail: A detailed study of the resolvent equation (cf.~\cite[\S6]{ggko.2014}) shows that $u = u(x)$ can be assumed to be smooth on $[0,\infty)$ and rapidly decaying as $x \to \infty$ and therefore a logarithmic divergence of the right-hand side in \eqref{hardy_crit_1} occurs at $x = \infty$ and of the right-hand side in \eqref{hardy_crit_2} at $x = 0$. This is a quite general feature making weight exponents $\alpha \notin \Z$ for the norms $\verti{\cdot}_{k,\alpha}$ and inner products $(\cdot,\cdot)_{k,\alpha}$ more convenient. Since taking the time trace later on will shift the weight $\alpha$ by $-\frac 1 2$, we assume $\alpha \notin \frac 1 2 \Z$ in what follows.

\medskip

We further note that unlike in \cite{gko.2008,bgko.2016} it is not convenient to study derivatives $\partial_x^k u$ with $k \ge  1$. In order to recognize this, observe that if $u$ solves \eqref{lin_pde}, then $\partial_x^k u$ solves
\begin{equation}\label{der_lin}
\partial_t (\partial_x^k u) + \AAA_k (\partial_x^k u) = \partial_x^k f \quad \mbox{for } \, t, x > 0,
\end{equation}
where the operator $\AAA_k$ is defined by the identity $\partial_x^k \AAA = \AAA_k \partial_x^k$. Using \eqref{defa} and \eqref{def_pq} we have
\begin{align*}
\AAA &= x^{-1} p(D) +  x^{-2} q(D) \\
&= x^{-1} D^2 (D-1) (D-2) +  x^{-1} D x^{-1} D (D-1) (D-2) \\
&= \partial_x D (D-1) (D-2) +  \partial_x^2 (D-1) (D-2).
\end{align*}
Hence, an easy induction shows that we have
\begin{equation}\label{ak0}
\AAA_k = \partial_x (D+k) (D-(1-k)) (D-(2-k)) +  \partial_x^2 (D-(1-k)) (D-(2-k)).
\end{equation}
From \eqref{ak0} we infer that
\begin{equation}\label{ak}
\AAA_k = x^{-1} p_k(D) +  x^{-2} q_k(D),
\end{equation}
where
\begin{equation}\label{def_pkqk}
\begin{aligned}
p_k(D) &= D (D+k) (D-(1-k)) (D-(2-k)),\\
q_k(D) &= D (D-1) (D-(1-k)) (D-(2-k)).
\end{aligned}
\end{equation}
The equality \eqref{sum_coerc} can be generalized to
\begin{equation}\label{sum_pkqk}
\left(v,\AAA_k v\right)_\alpha = \left(v, p_k(D) v\right)_{\alpha+\frac 1 2} +  \left(v,q_k(D)v\right)_{\alpha+1},
\end{equation}
where $v = \partial_x^k u$. Now we can study the coercivity of $p_k(D)$ and $q_k(D)$ separately using the coercivity result \eqref{coer_cond} and obtain that $p_k(D)$ ($q_k(D)$) is coercive with respect to $(\cdot,\cdot)_{\alpha_1}$ ($(\cdot,\cdot)_{\alpha_2}$) if
\begin{subequations}\label{coer_pkqk}
\begin{align}
\alpha_1 &= 0 \quad \mbox{for } \, k = 1 \, \mbox{ (only non-negative)}, \\
\begin{split}
\alpha_1 &\in \left(\frac{3 (1-k) - \sqrt{k^2 - 2k + \frac{11}{3}}}{4}, \frac{3 (1-k) + \sqrt{k^2 - 2k + \frac{11}{3}}}{4}\right) \\
& \qquad \cap \left(1-k,2-k\right) \quad \mbox{for } \, k \ge 2,
\end{split}
\\
\alpha_2 &\in \left(0,\frac{1}{\sqrt 6}\right) \quad \mbox{for } \, k = 1, \\
\alpha_2 &= 0 \quad \mbox{for } \, k = 2 \, \mbox{ (only non-negative)}, \\
\begin{split}
\alpha_2 &\in \left(\frac{2-k}2 - \frac{\sqrt{k^2-2k+3}}{2 \sqrt 3},\frac{2-k}2 + \frac{\sqrt{k^2-2k+3}}{2 \sqrt 3}\right) \\
& \qquad \cap \left(2-k,0\right) \quad \mbox{for } \, k \ge 2.
\end{split}
\end{align}
\end{subequations}
In view of \eqref{sum_pkqk}, \eqref{coer_pkqk} leads to restrictions on $\alpha$ (that is, coercivity constraints on $\AAA_k$), so that a non-negative operator $\AAA_k$ is obtained only for $k = 2$ and coercitivity does not hold for any $k \ge 0$.

\subsubsection{Coercivity estimates for commutated linear operators\label{sec:coerc_comm}}
In the previous section we have recognized that coercivity of the operator $\AAA$ in our scale of weighted inner products and norms \eqref{def_in_norm} requires joint coercivity (non-negativity) of the operators $p(D)$ and $q(D)$ with respective weights. Thus the coercivity range shrinks to a single value for which the operator $q(D)$ is only non-negative. This is due to the double ``middle'' root of $q(D)$ (cf.~\eqref{def_pq} and \eqref{coer_cond}). Here we show how we can shift this root so that the intersection of the coercivity ranges of the appropriately shifted operators is indeed a non-empty interval. In fact, having the possibility to slightly shift the weight is convenient for applying Hardy's inequality (see below). 

\medskip

As outlined in \S\ref{sec:out}, we can obtain better control on $u$ and allow for a larger range of weights if we apply $(D-1)$ to equation~\eqref{lin_pde}:
\[
\partial_t \tilde u + \tilde \AAA \tilde u = \tilde f \quad \mbox{for } \, t, x > 0,
\]
where $\tilde u = (D-1) u$ and $\tilde f = (D-1) f$ have been defined in \eqref{def_tc}. Note that this operation in particular preserves the boundary conditions \eqref{bc_lin} and \eqref{bc_rhs}. In view of \eqref{defa} and \eqref{def_pq}, the commutation of $(D-1)$ with the operator $\AAA$ yields
\begin{subequations}\label{comm_D1A}
\begin{equation}
x^{-1} (D-2) p(D) +  x^{-2} (D-3) q(D) \stackrel{\eqref{defa}}{=} (D-1) \AAA \stackrel{\eqref{commute_a_op}}{=} \tilde \AAA (D-1)
\end{equation}
with
\begin{equation}\label{def_ta}
\tilde \AAA = x^{-1} \tilde p(D) +  x^{-2} \tilde q(D) \quad \mbox{and } \, \begin{cases} \tilde p(\zeta) \stackrel{\eqref{def_pq}}{=} \zeta^2 (\zeta-2)^2, \\ \tilde q(\zeta) \stackrel{\eqref{def_pq}}{=} \zeta (\zeta-1) (\zeta-2) (\zeta-3).\end{cases}
\end{equation}
\end{subequations}
We obtain the following coercivity result:
\begin{lemma}\label{lem:coerc_ta}
The operator $\tilde \AAA$ (cf.~\eqref{def_ta}) fulfills coercivity in the sense of
\begin{equation}\label{coercivity_ta}
\left(\tilde \AAA \tilde u, \tilde u\right)_{\tilde \alpha} \gtrsim_{\tilde\alpha} \verti{\tilde u}_{2, \tilde \alpha + \frac 1 2}^2 +  \verti{\tilde u}_{2,\tilde \alpha+1}^2
\end{equation}
for all locally integrable $\tilde u$ with $\verti{\tilde u}_{2,\tilde \alpha+\frac 1 2} < \infty$ and $\verti{\tilde u}_{2,\tilde \alpha+1} < \infty$ if
\begin{equation}\label{coer_ta}
\tilde \alpha \in \left(0,1\right).
\end{equation}
\end{lemma}

\begin{proof}
Using the criterion \eqref{coer_cond} of Lemma~\ref{lem:coer_p}, we can explicitly calculate that $\tilde p(D)$ and $\tilde q(D)$ are coercive in the sense of \eqref{coerc_p} with respect to $(\cdot,\cdot)_{\alpha_1}$ or $(\cdot,\cdot)_{\alpha_2}$, respectively, if
\begin{equation}\label{coer_pq}
\alpha_1 \in \left(1 - \frac{1}{ \sqrt 3}, 1 + \frac{1}{ \sqrt 3}\right) \quad \mbox{and} \quad \alpha_2 \in \left(1,2\right).
\end{equation}
Since by testing we have
\[
\left(\tilde \AAA \tilde u, \tilde u\right)_{\tilde \alpha} = \left(\tilde p(D) \tilde u, \tilde u\right)_{\tilde \alpha + \frac 1 2} +  \left(\tilde q(D) \tilde u, \tilde u\right)_{\tilde \alpha + 1},
\]
we obtain $\tilde \alpha + 1/2 = \alpha_1$ and $\tilde \alpha + 1 = \alpha_2$ and due to \eqref{coer_pq} estimate~\eqref{coercivity_ta} under the constraint \eqref{coer_ta} follows.
\end{proof}
Due to \eqref{coercivity_ta} and since we can choose $\tilde \alpha > 0$, we expect control of the solution $u$ to \eqref{linear_u} in form of $u(t,x) = u_1(t) x + o(x)$ as $x \searrow 0$. It appears to be convenient to take in what follows $\tilde \alpha = \delta$, where $0 < \delta < \frac 1 2$, as then $\tilde \alpha$ is subcritical with respect to $u_1 x$ and $\tilde\alpha+1$ is supercritical with respect to $u_1 x$.

\medskip

For obtaining control up to $u_2 x^2$, we may consider the function $\check u := (D-1) (D-2) u$ (cf.~\eqref{def_tc}). Indeed, applying $(D-1) (D-2)$ cancels the expansion of $u$ (cf.~\eqref{formal_exp_cu}) so that we have $\check u = O(x^3)$ as $x \searrow 0$ and in particular the boundary conditions \eqref{bc_lin} and \eqref{bc_rhs} are preserved. Hence, we can expect higher-regularity estimates for $\check u$ that include weights larger than $2$. Using \eqref{commute_a_op} and \eqref{def_ta}, we can compute
\begin{equation}\label{check_a}
\check \AAA = x^{-1} \check p(D) +  x^{-2} \check q(D) \quad \mbox{with } \, \begin{cases} \check p(\zeta) \stackrel{\eqref{def_pq}}{=} \zeta^2 (\zeta-2) (\zeta-3), \\ \check q(\zeta) \stackrel{\eqref{def_pq}}{=} \zeta (\zeta-1) (\zeta-3) (\zeta-4). \end{cases}
\end{equation}
By the same reasoning a coercivity result analogous to Lemma~\ref{lem:coerc_ta} holds:
\begin{lemma}\label{lem:coerc_ca}
The operator $\check \AAA$ (cf.~\eqref{check_a}) fulfills coercivity in the sense of
\begin{equation}\label{coercivity_ca}
\left(\check \AAA \check u, \check u\right)_{\check \alpha} \gtrsim \verti{\check u}_{2, \check \alpha + \frac 1 2}^2 +  \verti{\check u}_{2,\check \alpha+1}^2
\end{equation}
for all locally integrable $\check u$ with $\verti{\check u}_{2,\check \alpha+\frac 1 2} < \infty$ and $\verti{\check u}_{2,\check \alpha+1} < \infty$ if
\begin{equation}\label{coer_ca}
\check \alpha \in \left(1-\sqrt{\frac 5 6}, \frac 3 2\right).
\end{equation}
\end{lemma}
In particular $\check \alpha = 1 + \delta$ for $0 < \delta < \frac 1 2$ is an admissible weight exponent, so that $\check\alpha$ is supercritical with respect to the addend $u_1 x$ and $\check \alpha + 1$ is supercritical with respect to the addend $u_2 x^2$ in the respective cases.

\subsection{Parabolic maximal regularity I: heuristics\label{sec:maxreg}}
In this section we make use of the coercivity estimates \eqref{coercivity_ta} and \eqref{coercivity_ca} of Lemmas~\ref{lem:coerc_ta} and \ref{lem:coerc_ca}, respectively, in order to obtain maximal regularity for the operators $\tilde \AAA$ and $\check \AAA$ (cf.~\eqref{def_ta} and \eqref{check_a}), respectively. A rigorous justification of the subsequent arguments is not difficult and can be carried out by a time-discretization argument contained in \S\ref{sec:lin_res}--\S\ref{sec:lin_rig}. However, since the arguments are relatively technical and hide the simplicity of the reasoning, we stick to the time-continuous formulation for the time being and concentrate on deriving estimates assuming existence of sufficiently regular solutions from the outset.

\medskip

Throughout the paper, we use the notation
\begin{equation}\label{decomp_w}
\underline w := \frac{w}{x+1}
\end{equation}
for a function $w: (0,\infty) \to \R$. Note that for $w: [0,\infty) \to \R$ smooth, we have by power series expansion for all $N \in \N$
\[
\underline w = \sum_{j = 0}^\infty (-1)^j x^j \left(\sum_{k = 0}^N w_k x^k + O\left(x^{N+1}\right)\right) = \sum_{s = 0}^N \underline w_s x^s + O\left(x^{N+1}\right),
\]
where for $s \ge 0$
\begin{equation}\label{coeff_uw}
\underline w_s = \sum_{k = 0}^s (-1)^{s-k} w_k.
\end{equation}
From \eqref{coeff_uw} we see that indeed
\begin{equation}\label{sum_w1w2}
\underline w_s + \underline{x w}_s = w_s.
\end{equation}
We begin by testing equation~\eqref{eq_ta} with $\tilde u$ in the inner product $\left(\cdot,\cdot\right)_{\tilde \alpha}$ and obtain
\begin{equation}\label{test_ta}
\left(\partial_t \tilde u, \tilde u\right)_{\tilde \alpha} + \left(\tilde \AAA \tilde u, \tilde u\right)_{\tilde \alpha} = \left(\tilde f, \tilde u\right)_{\tilde \alpha}.
\end{equation}
Observe that $\left(\partial_t \tilde u, \tilde u\right)_{\tilde \alpha} = \frac 1 2 \frac{\d}{\d t} \verti{\tilde u}_{\tilde \alpha}^2$ and that by Young's inequality\footnote{Note that the operations $\widetilde{(\ldots)}$ and $\underline{{(\ldots)}}$ do not commute. We first apply $\underline{{(\ldots)}}$ and afterwards $\widetilde{{(\ldots)}}$.},
\[
\left(\tilde f, \tilde u\right)_{\tilde \alpha} = \left(\widetilde{x \underline f}, \tilde u\right)_{\tilde \alpha} + \left(\widetilde{\underline f},  \tilde u\right)_{\tilde \alpha} \le \frac{1}{2 c} \left(\verti{\widetilde{x \underline f}}_{\tilde \alpha - \frac 1 2}^2 +  \verti{\widetilde{\underline f}}_{\tilde \alpha -1}^2\right) + \frac c 2 \left(\verti{\tilde u}_{\tilde \alpha + \frac 1 2}^2 +  \verti{\tilde u}_{\tilde \alpha + 1}^2\right)
\]
for any $c > 0$. Assuming $\tilde \alpha$ as in \eqref{coer_ta} and employing coercivity in form of \eqref{coercivity_ta}, after adjusting $c$, the tested equation \eqref{test_ta} upgrades to
\begin{equation}\label{basic_ta}
\frac{\d}{\d t} \verti{\tilde u}_{\tilde \alpha}^2 + \verti{\tilde u}_{2,\tilde \alpha + \frac 1 2}^2 +  \verti{\tilde u}_{2,\tilde \alpha + 1}^2 \lesssim_{\tilde\alpha}  \left(\verti{\widetilde{x \underline f}}_{\tilde \alpha - \frac 1 2}^2 +  \verti{\widetilde{\underline{f}}}_{\tilde \alpha -1}^2\right).
\end{equation}
This is a weak estimate since only two spatial derivatives $D$ are controlled although the operator $\tilde \AAA$ is of order four. It can be upgraded to a higher-regularity estimate by applying $D^k$ with $k \ge 2$ to \eqref{eq_ta}, that is,
\[
\partial_t D^k \tilde u + D^k \tilde \AAA \tilde u = D^k \tilde f \quad \mbox{for} \quad t, x > 0.
\]
Testing this equation with $D^k \tilde u$ in the inner product $\left(\cdot,\cdot\right)_{\tilde \alpha}$, we arrive after some elementary manipulations (using \eqref{def_ta}) at
\begin{equation}\label{test_lin2}
\begin{aligned}
& \left(\partial_t D^k \tilde u, D^k \tilde u\right)_{\tilde \alpha} + \left(\tilde p(D) (D-1)^k \tilde u, D^k \tilde u\right)_{\tilde \alpha + \frac 1 2} +  \left(\tilde q(D) (D-2)^k \tilde u, D^k \tilde u\right)_{\tilde \alpha + 1} \\
& \quad = \left(D^k \tilde f, D^k \tilde u\right)_{\tilde \alpha}.
\end{aligned}
\end{equation}
Apparently $\left(\partial_t D^k \tilde u, D^k \tilde u\right)_{\tilde\alpha} = \frac 1 2 \frac{\d}{\d t} \verti{D^k \tilde u}_{\tilde\alpha}^2$ for the first term in \eqref{test_lin2}. Furthermore, since $\tilde p(D)$ and $\tilde q(D)$ are fourth-order operators, integration by parts and a standard interpolation estimate show that
\begin{align*}
\left(\tilde p(D) (D-1)^k \tilde u, D^k \tilde u\right)_{\tilde \alpha + \frac 1 2} &\ge \frac 1 2 \verti{\tilde u}_{k+2,\tilde \alpha + \frac 1 2}^2 - \tilde C \verti{\tilde u}_{\tilde \alpha + \frac 1 2}^2, \\
\left(\tilde q(D) (D-1)^k \tilde u, D^k \tilde u\right)_{\tilde \alpha + 1} &\ge \frac 1 2 \verti{\tilde u}_{k+2,\tilde \alpha + 1}^2 - \tilde C \verti{\tilde u}_{\tilde \alpha + 1}^2,
\end{align*}
where $\tilde C > 0$ is chosen sufficiently large. Finally, skew-symmetry of $D$  with respect to $\left(\cdot,\cdot\right)_0$ in conjunction with $x^{-\tilde\alpha} D = \left(D + \tilde\alpha\right) x^{-\tilde\alpha}$ yields
\begin{align*}
\left(D^k \tilde f, D^k \tilde u\right)_{\tilde\alpha} &= \left(D^{k-2} \tilde f, \left(D - 2 \tilde \alpha\right)^2 D^k \tilde u\right)_{\tilde\alpha} \\
&\lesssim_{\tilde\alpha} c^{-1} \left(\verti{\widetilde{x \underline f}}_{k-2,\tilde \alpha - \frac 1 2}^2 +  \verti{\widetilde{\underline f}}_{k-2,\tilde \alpha - 1}^2\right) + c \left(\verti{\tilde u}_{k+2,\tilde \alpha + \frac 1 2}^2 +  \verti{\tilde u}_{k+2,\tilde \alpha + 1}^2\right)
\end{align*}
for any $c > 0$, so that by enlarging $\tilde C$, estimate~\eqref{test_lin2} turns into
\begin{equation}\label{higher_ta}
\begin{aligned}
& \frac{\d}{\d t} \verti{D^k \tilde u}_{\tilde \alpha}^2 + \verti{D^{k+2} \tilde u}_{\tilde \alpha + \frac 1 2}^2 - \tilde C \verti{\tilde u}_{\tilde \alpha + \frac 1 2}^2 + \verti{D^{k+2} \tilde u}_{\tilde \alpha + 1}^2 - \tilde C \verti{\tilde u}_{\tilde \alpha + 1}^2 \\
& \quad \lesssim_{k,\tilde\alpha} \verti{\widetilde{x \underline f}}_{k-2,\tilde \alpha - \frac 1 2}^2 +  \verti{\widetilde{\underline{f}}}_{k-2,\tilde \alpha -1}^2.
\end{aligned}
\end{equation}
The combination of \eqref{basic_ta} with \eqref{higher_ta} yields
\begin{equation}\label{mr_ta_0}
\frac{\d}{\d t} \left(\verti{\tilde u}_{\tilde \alpha}^2 + \tilde C \verti{D^k \tilde u}_{\tilde \alpha}^2\right) + \verti{\tilde u}_{k+2,\tilde \alpha + \frac 1 2}^2 +  \verti{\tilde u}_{k+2,\tilde \alpha + 1}^2 \lesssim_{k,\tilde\alpha} \verti{\widetilde{x \underline f}}_{k-2,\tilde \alpha - \frac 1 2}^2 +  \verti{\widetilde{\underline f}}_{k-2,\tilde \alpha -1}^2.
\end{equation}
In order to obtain control on $\partial_t u$ as well, observe that with help of \eqref{decomp_w} we have
\begin{equation}\label{time_control}
\partial_t \widetilde{x \underline u} \stackrel{\eqref{lin_pde}}{=} \widetilde{x \underline f} - \widetilde{x \underline{\AAA u}}.
\end{equation}
In conjunction with the commutator $D x^{-1} = x^{-1} (D-1)$ and $\AAA \stackrel{\eqref{defa}}{=} x^{-1} p(D) + x^{-2} q(D)$, where both $p(D)$ and $q(D)$ have the linear factor $D-1$ (cf.~\eqref{def_pq}), we arrive at
\begin{equation}\label{time_tu1}
\verti{\partial_t \widetilde{x \underline u}}_{k-2,\tilde\alpha-\frac 1 2} \lesssim_k \verti{\widetilde{x \underline f}}_{k-2,\tilde\alpha-\frac 1 2} + \verti{\tilde u}_{k+2,\tilde\alpha+\frac 12}.
\end{equation}
Indeed, from \eqref{time_control} the first term on the right-hand side of \eqref{time_tu1} is trivial and we only need to treat the term
\[
\verti{\widetilde{x \underline{ \AAA u}}}_{k-2,\tilde\alpha-\frac 1 2} \stackrel{\eqref{defa}}{\le} \verti{(D-1) (x+1)^{-1} p(D) u}_{k-2,\tilde\alpha-\frac 1 2} + \verti{(D-1) x^{-1} (x+1)^{-1} q(D) u}_{k-2,\tilde\alpha-\frac 1 2}.
\]
Now observe that
\[
(D-1) (x+1)^{-1} p(D) u = (D-1) (x+1)^{-1} D^2 (D-2) \tilde u.
\]
Using
\begin{enumerate}[(a)]
\item $D (x+1)^{-\gamma} = -\gamma x (x+1)^{-\gamma-1}$ for $\gamma \in \R$,
\item the bound $(x+1)^{-\gamma} \le \min\{1,x\}$ for any $\gamma \ge 0$,
\end{enumerate}
we infer
\[
\verti{(D-1) (x+1)^{-1} p(D) u}_{k-2,\tilde\alpha-\frac 1 2} \lesssim_k \verti{\tilde u}_{k+2,\tilde\alpha + \frac 1 2}.
\]
We skip the details for the term $\verti{(D-1) x^{-1} (x+1)^{-1} q(D) u}_{k-2,\tilde\alpha-\frac 1 2}$ as there are no material differences to the one just treated.

\medskip

With an analogous reasoning also
\begin{equation}\label{time_tu2}
\verti{\partial_t \widetilde{\underline u}}_{k-2,\tilde\alpha-1} \lesssim_{k,\tilde\alpha}  \verti{\widetilde{\underline f}}_{k-2,\tilde\alpha-1} + \verti{\tilde u}_{k+2,\tilde\alpha+1}.
\end{equation}
Using \eqref{time_tu1} and \eqref{time_tu2} in \eqref{higher_ta} leads to an upgraded version of estimate~\eqref{mr_ta_0} that reads
\begin{equation}\label{mr_ta}
\begin{aligned}
&\frac{\d}{\d t} \left(\verti{\tilde u}_{\tilde \alpha}^2 + \tilde C \verti{D^k \tilde u}_{\tilde \alpha}^2\right) + \verti{\partial_t \widetilde{x \underline u}}_{k-2,\tilde\alpha-\frac 1 2}^2 +  \verti{\partial_t \widetilde{\underline u}}_{k-2,\tilde\alpha-1}^2 + \verti{\tilde u}_{k+2,\tilde \alpha + \frac 1 2}^2 +  \verti{\tilde u}_{k+2,\tilde \alpha + 1}^2 \\
& \quad \lesssim_{k,\tilde\alpha} \verti{\widetilde{x \underline f}}_{k-2,\tilde \alpha - \frac 1 2}^2 +  \verti{\widetilde{\underline f}}_{k-2,\tilde \alpha -1}^2.
\end{aligned}
\end{equation}
In integrated form we obtain
\begin{equation}\label{maxreg_ta}
\begin{aligned}
&\sup_{t \ge 0} \verti{\tilde u}_{k,\tilde \alpha}^2 + \int_0^\infty \left(\verti{\partial_t \widetilde{x \underline u}}_{k-2,\tilde\alpha-\frac 1 2}^2 +  \verti{\partial_t \widetilde{\underline u}}_{k-2,\tilde\alpha-1}^2 + \verti{\tilde u}_{k+2,\tilde \alpha + \frac 1 2}^2 +  \verti{\tilde u}_{k+2,\tilde \alpha + 1}^2\right) \d t  \\
& \quad \lesssim_{k,\tilde\alpha} \verti{\tilde u_{|t = 0}}_{k, \tilde \alpha}^2 + \int_0^\infty \left(\verti{\widetilde{x \underline f}}_{k-2,\tilde \alpha - \frac 1 2}^2 +  \verti{\widetilde{\underline f}}_{k-2,\tilde \alpha -1}^2\right) \d t.
\end{aligned}
\end{equation}
Indeed, on the left-hand side, four spatial derivatives and one time derivative more on $\tilde u$ (with appropriate weights) are controlled than for $\tilde f$ on the right-hand side, which is the maximal regularity gain possible.

\medskip

Since equation~\eqref{eq_ca} is structurally the same as \eqref{eq_ta} and the above reasoning only required coercivity of $\tilde \AAA$ in form of \eqref{coercivity_ta} and some extra effort to derive \eqref{time_tu1} and \eqref{time_tu2} which can be easily adapted, under the coercivity assumption \eqref{coer_ca} for $\check \alpha$, we also have
\begin{equation}\label{maxreg_ca}
\begin{aligned}
&\sup_{t \ge 0} \verti{\check u}_{k,\check \alpha}^2 + \int_0^\infty \left(\verti{\partial_t \widecheck{x \underline u}}_{k-2,\check\alpha-\frac 1 2}^2 +  \verti{\partial_t \widecheck{\underline u}}_{k-2,\check\alpha-1}^2 + \verti{\check u}_{k+2,\check \alpha + \frac 1 2}^2 +  \verti{\check u}_{k+2,\check \alpha + 1}^2\right) \d t \\
&\quad \lesssim_{k,\check\alpha}  \verti{\check u_{|t = 0}}_{k, \check \alpha}^2 + \int_0^\infty \left(\verti{\widecheck{x \underline f}}_{k-2,\check \alpha - \frac 1 2}^2 +  \verti{\widecheck{\underline f}}_{k-2,\check \alpha -1}^2\right) \d t.
\end{aligned}
\end{equation}
We additionally notice that by applying $\ell$ time derivatives and by applying the operator $\tilde \AAA$ or $\check \AAA$ $m$-times to \eqref{eq_ta} or \eqref{eq_ca}, respectively, we obtain
\begin{subequations}\label{eq_ta_ca_lm}
\begin{align}
\partial_t \left(\partial_t^\ell \tilde \AAA^m \tilde u\right) + \tilde \AAA \left(\partial_t^\ell \tilde \AAA^m \tilde u\right) &= \left(\partial_t^\ell \tilde \AAA^m \tilde f\right) \quad \mbox{for} \quad t,x > 0,\\
\partial_t \left(\partial_t^\ell \check \AAA^m \check u\right) + \check \AAA \left(\partial_t^\ell \check \AAA^m \check u\right) &= \left(\partial_t^\ell \check \AAA^m \check f\right) \quad \mbox{for} \quad t,x > 0.
\end{align}
\end{subequations}
This implies that the tuples $\left(\partial_t^\ell \tilde \AAA^m \tilde u, \partial_t^\ell \tilde \AAA^m \tilde f\right)$ or $\left(\partial_t^\ell \check \AAA^m \check u, \partial_t^\ell \check \AAA^m \check f\right)$ fulfill exactly the same equations as $\left(\tilde u, \tilde f\right)$ or $\left(\check u, \check f\right)$, respectively. Furthermore, if $u: [0,\infty) \to \R$ is smooth with $u_{x = 0} = 0$, then $\tilde u_{x = 0} = \partial_x \tilde u_{|x = 0} = 0$ and thus $\tilde p(D) \tilde u \stackrel{\eqref{def_ta}}{=} O\left(x^3\right)$ as $x \searrow 0$ and $\tilde q(D) \tilde u \stackrel{\eqref{def_ta}}{=} O\left(x^4\right)$ as $x \searrow 0$, so that $\tilde \AAA \tilde u = O\left(x^2\right)$ as $x \searrow 0$. In particular, the boundary condition \eqref{bc_lin} remains valid for $u$ replaced by $\tilde u$ and by the same argumentation also \eqref{bc_rhs} remains satisfied for $f$ replaced by $\tilde f$. An analogous reasoning also applies for the tuple $\left(\check \AAA \check u, \check \AAA \check f\right)$ and inductively we infer that $\left(\partial_t^\ell \tilde \AAA^m \tilde u, \partial_t^\ell \tilde \AAA^m \tilde f\right)$ and $\left(\partial_t^\ell \check \AAA^m \check u, \partial_t^\ell \check \AAA^m \check f\right)$ meet the same boundary conditions as $(u,f)$. Therefore, maximal-regularity estimates in the form of
\begin{subequations}\label{maxreg_tc}
\begin{equation}\label{maxreg_ta_lm}
\begin{aligned}
& \sup_{t \ge 0} \verti{\partial_t^\ell \tilde \AAA^m \tilde u}_{k,\tilde \alpha}^2 + \int_0^\infty \left(\verti{\partial_t^{\ell+1} \tilde \AAA^m \widetilde{x \underline u}}_{k-2,\tilde \alpha - \frac 1 2}^2 +  \verti{\partial_t^{\ell+1} \tilde \AAA^m \widetilde{\underline u}}_{k-2,\tilde \alpha - 1}^2\right) \d t \\
& + \int_0^\infty \left(\verti{\partial_t^\ell \tilde \AAA^m \tilde u}_{k+2,\tilde \alpha + \frac 1 2}^2 +  \verti{\partial_t^\ell \tilde \AAA^m \tilde u}_{k+2,\tilde \alpha + 1}^2\right) \d t \\
& \quad \lesssim_{k,\tilde\alpha} \verti{\left(\partial_t^\ell \tilde \AAA^m \tilde u\right)_{|t = 0}}_{k, \tilde \alpha}^2 + \int_0^\infty \left(\verti{\partial_t^\ell \tilde \AAA^m \widetilde{x \underline f}}_{k-2,\tilde \alpha - \frac 1 2}^2 +  \verti{\partial_t^\ell \tilde \AAA^m \widetilde{\underline f}}_{k-2,\tilde \alpha -1}^2\right) \d t
\end{aligned}
\end{equation}
and
\begin{equation} \label{maxreg_ca_lm}
\begin{aligned}
& \sup_{t \ge 0} \verti{\partial_t^\ell \check \AAA^m \check u}_{k,\check \alpha}^2 + \int_0^\infty \left(\verti{\partial_t^{\ell+1} \check \AAA^m \widecheck{x \underline u}}_{k-2,\check \alpha - \frac 1 2}^2 +  \verti{\partial_t^{\ell+1} \check \AAA^m \widecheck{\underline u}}_{k-2,\check \alpha - 1}^2\right) \d t \\
& + \int_0^\infty \left(\verti{\partial_t^\ell \check \AAA^m \check u}_{k+2,\check \alpha + \frac 1 2}^2 +  \verti{\partial_t^\ell \check \AAA^m \check u}_{k+2,\check \alpha + 1}^2\right) \d t \\
& \quad \lesssim_{k,\check\alpha} \verti{\left(\partial_t^\ell \check \AAA^m \check u\right)_{|t = 0}}_{k, \check \alpha}^2 + \int_0^\infty \left(\verti{\partial_t^\ell \check \AAA^m \widecheck{x \underline f}}_{k-2,\check \alpha - \frac 1 2}^2 +  \verti{\partial_t^\ell \check \AAA^m \widecheck{\underline f}}_{k-2,\check \alpha -1}^2\right) \d t
\end{aligned}
\end{equation}
\end{subequations}
for all $k \ge 2$, $\ell \ge 0$, and $m \ge 0$ need to be satisfied. We additionally observe furthermore that inductively from \eqref{eq_ta_ca_lm} it follows
\begin{align*}
\partial_t^\ell \tilde \AAA^m \tilde u &= (-1)^\ell \tilde \AAA^{m+\ell} \tilde u + \sum_{\ell^\prime = 0}^{\ell-1} (-1)^{\ell-1-\ell^\prime} \partial_t^{\ell^\prime} \tilde \AAA^{m+\ell-1-\ell^\prime} \tilde f, \\
\partial_t^\ell \check \AAA^m \check u &= (-1)^\ell \check \AAA^{m+\ell} \check u + \sum_{\ell^\prime = 0}^{\ell-1} (-1)^{\ell-1-\ell^\prime} \partial_t^{\ell^\prime} \check \AAA^{m+\ell-1-\ell^\prime} \check f,
\end{align*}
so that for higher-regularity estimates by taking the boundary value at time $t = 0$ the following \emph{time-trace identities}
\begin{subequations}\label{comp_mr}
\begin{align}
\left(\partial_t^\ell \tilde \AAA^m \tilde u\right)_{|t = 0} &= (-1)^\ell \tilde \AAA^{m+\ell} \widetilde{u^{(0)}} + \sum_{\ell^\prime = 0}^{\ell-1} (-1)^{\ell^\prime} \left(\partial_t^{\ell-1-\ell^\prime} \tilde \AAA^{m+\ell^\prime} \tilde f\right)_{|t = 0}, \\
\left(\partial_t^\ell \check \AAA^m \check u\right)_{|t = 0} &= (-1)^\ell \check \AAA^{m+\ell} \widecheck{u^{(0)}} + \sum_{\ell^\prime = 0}^{\ell-1} (-1)^{\ell^\prime} \left(\partial_t^{\ell-1-\ell^\prime} \check \AAA^{m+\ell^\prime} \check f\right)_{|t = 0},
\end{align}
\end{subequations}
are fulfilled almost everywhere. A rigorous statement reads as follows:
\begin{proposition}\label{prop:maxreg1}
For locally integrable functions $f: \, (0,\infty)^2 \to \R$ and $u^{(0)}: \, (0,\infty) \to \R$ such that the right-hand sides of \eqref{maxreg_tc} are finite for all $k \ge 2$, $\ell \ge 0$, and $m \ge 0$, problem~\eqref{linear_u} subject to $u_{|t = 0} = u^{(0)}$ has exactly one locally integrable solution $u: (0,\infty)^2 \to \R$ with finite left-hand sides of \eqref{maxreg_tc} for all $k \ge 2$, $\ell \ge 0$, and $m \ge 0$. Furthermore, the maximal-regularity estimates \eqref{maxreg_tc} are satisfied for $k \ge 2$, $\ell \ge 0$, and $m \ge 0$, where $\tilde \alpha$ fulfills \eqref{coer_ta}, $\check \alpha$ meets \eqref{coer_ca}, and the compatibility conditions \eqref{comp_mr} are satisfied almost everywhere. Here, all estimates only depend on $k$, $\ell$, $m$, and $\tilde\alpha$ or $\check\alpha$, respectively. Furthermore, uniqueness holds under the weaker assumption that $u: (0,\infty)^2 \to \R$ is locally integrable with
\[
\int_0^\infty \left(\verti{\tilde u}_{4,\tilde \alpha + \frac 1 2}^2 +  \verti{\tilde u}_{4,\tilde \alpha + 1}^2\right) \d t < \infty \quad \mbox{or} \quad \int_0^\infty \left(\verti{\check u}_{4,\check \alpha + \frac 1 2}^2 +  \verti{\check u}_{4,\check \alpha + 1}^2\right) \d t < \infty.
\]
\end{proposition}
We will rigorously prove Proposition~\ref{prop:maxreg1} in \S\ref{sec:lin_res} and \S\ref{sec:lin_rig}.

\subsection{Elliptic maximal regularity\label{sec:elliptic}}

In this section, we discuss a main ingredient in how the maximal-regularity estimates of Proposition~\ref{prop:maxreg1} yield higher-regularity estimates for the solution $u$ to \eqref{linear_u} at the boundary $x = 0$. This is formulated in the following statement:
\begin{proposition}\label{prop:ell_reg}
For $\gamma \in (-2,\infty) \setminus \Z$ and $k,m \in \N_0$ we have
\begin{subequations}\label{ell_main}
\begin{align}
\sum_{r = 0}^m  \verti{w - \sum_{j = 1}^{\floor{\gamma}+m+r} w_j x^j}_{k+4m+1,\gamma+m+r} &\sim  \verti{\tilde \AAA^m \tilde w}_{k,\gamma} + \sum_{r = 1}^m  \verti{w_{\floor\gamma+m+r}}, \label{ell_main_1}\\
\sum_{r = 0}^m  \verti{w - \sum_{j = 1}^{\floor{\gamma}+m+r} w_j x^j}_{k+4m+2,\gamma+m+r} &\sim  \verti{\check \AAA^m \check w}_{k,\gamma} + \sum_{r = 1}^m  \verti{w_{\floor\gamma+m+r}}, \label{ell_main_2}
\end{align}
\end{subequations}
where $w : \, [0,\infty) \to \R$ is smooth with $w = 0$ at $x = 0$ and the constants in \eqref{ell_main} only depend on $k$, $m$, and $\gamma$.
\end{proposition}
Indeed, the terms appearing in Proposition~\ref{prop:maxreg1} are exactly of the form $\verti{\tilde \AAA^m \tilde w}_{k,\gamma}$ or $\verti{\check \AAA^m \check w}_{k,\gamma}$. The remaining coefficients on the right-hand sides of \eqref{ell_main} will make additional considerations necessary, which will be addressed in \S\ref{sec:maxreg2}.

\medskip

First, observe that by the commutation relations \eqref{commute_a_op} and by \eqref{def_tc} we have
\begin{equation}\label{commute_tc_au}
\tilde \AAA^m \tilde u = (D-1) \AAA^m u = \widetilde{\AAA^m u} \quad \mbox{and} \quad \check \AAA^m \check u = (D-2) (D-1) \AAA^m u = \widecheck{\AAA^m u}.
\end{equation}
and analogous expressions for $\tilde \AAA^k \tilde f$ and $\check \AAA^k \check f$. Furthermore, we observe that the operator $\AAA$ factorizes (by commuting $D$-derivatives with $x$, cf.~\eqref{defa} and \eqref{def_pq}):
\begin{align}\nonumber
\AAA &= x^{-2} x D^2 (D-1) (D-2) +  x^{-2} D (D-1)^2 (D-2) \\
&= x^{-2} (D-1)^2 (D-2) \left(x (D-2) +  D\right). \label{factor_a_0}
\end{align}
The last factor in \eqref{factor_a_0} vanishes on the function $(x+1)^2$, so that \eqref{factor_a_0} can be rewritten as
\begin{equation}\label{factor_a}
\AAA = x^{-2} (D-1)^2 (D-2) \BB, \quad \mbox{where} \quad \BB := x (D-2) +  D = (x+1)^3 D (x+1)^{-2}.
\end{equation}
Note that on smooth functions $f: (0,\infty) \to \R$ such that $\verti{f}_\gamma < \infty$ for some $\gamma > 0$, we can invert $\BB$ and have
\begin{equation}\label{rep_w}
\BB^{-1} f(x) = (x+1)^2 \int_0^x \left(x^\prime+1\right)^{-3} f\left(x^\prime\right) \frac{\d x^\prime}{x^\prime}.
\end{equation}
In view of \eqref{commute_tc_au} and \eqref{factor_a}, we need to study the elliptic regularity of
\begin{enumerate}[(a)]
\item a polynomial operator $P(D) = \prod_{j = 1}^N (D-\gamma_j)$ with $N \in \N$ and $\gamma_j \in \R$ for all $j = 1,\ldots,N$,
\item the operator $\BB \stackrel{\eqref{factor_a}}{=} (x+1)^3 D (x+1)^{-2}$.
\end{enumerate}
In fact, the elliptic regularity of a polynomial operator $P(D)$ follows by a straightforward application of Hardy's inequality (cf.~\cite[Lem.~7.2,~Lem.~7.4]{ggko.2014} and \cite[Lem.~A.1]{gko.2008} for similar statements):
\begin{lemma}\label{lem:elliptic}
Suppose that $w \in C^\infty([0,\infty))$, $k \in \N_0$, $P(D) = \prod_{j = 1}^N (D-\gamma_j)$, $\gamma_j,\varrho \in \R$, $\verti{w}_{N+k,\varrho} < \infty$, and
\[
D^{k^\prime} w(x) = o(x^\varrho) \quad \mbox{as } \, x \searrow 0 \, \mbox{ for } \, k^\prime = 0, \cdots, N+k-1.
\]
Then
\begin{equation}\label{elliptic_p}
\verti{w}_{k+N,\varrho} \sim_{\gamma_j,\varrho} \verti{P(D) w}_{k,\varrho}.
\end{equation}
\end{lemma}
The elliptic regularity of the operator $\BB$ requires special consideration:
\begin{lemma}\label{lem:estinvB}
The operator $\BB^{-1}$ defined in \eqref{rep_w} satisfies for any $k \in \N_0$ and  $\gamma>0$
\begin{equation}\label{estinvB}
 \verti{\BB^{-1} f}_{k+1,\gamma} + \verti{\BB^{-1} f}_{k+1,\gamma-1} \lesssim_{k,\gamma} \verti{f}_{k,\gamma},
\end{equation}
where $f : (0,\infty) \to \R$ is locally integrable and the constant in \eqref{estinvB} is independent of $w$.
\end{lemma}
\begin{proof}
We may assume without loss of generality $\verti{f}_{k,\gamma} < \infty$ and start by proving \eqref{estinvB} for $k = 0$. We set $g := \BB^{-1} f$ and notice that
\begin{equation}\label{b_k_0}
\begin{aligned}
\verti{f}_\gamma^2 &= \verti{\BB g}_\gamma^2 \stackrel{\eqref{factor_a}}{=} \left(x (D-2) g + D g, x (D-2) g + D g\right)_\gamma \\
&= \verti{(D-2) g}_{\gamma-1}^2 + \verti{D g}_\gamma^2 + 2 \left((D-2) g, D g\right)_{\gamma-\frac 1 2} \\
&= \verti{(D-2) g}_{\gamma-1}^2 + \verti{D g}_\gamma^2 + 2 \left((D-1) g - g, (D-1) g + g\right)_{\gamma-\frac 1 2} \\
&= \verti{(D-2) g}_{\gamma-1}^2 + 2 \verti{(D-1) g}_{\gamma-\frac 1 2}^2 + \verti{D g}_\gamma^2 - 2 \verti{g}_{\gamma-\frac 1 2}^2.
\end{aligned}
\end{equation}
Now observe that by Hardy's inequality (see for instance \cite[Lem.~A.1]{gko.2008}) we have because of $g(x) = \BB^{-1} f(x) \stackrel{\eqref{rep_w}}{=} o\left(x^\gamma\right)$ as $x \searrow 0$
\begin{subequations}\label{hardy_quant}
\begin{align}
\verti{D g}_\gamma^2 &= \int_0^\infty x^{2 - 2 \gamma} \left(\frac{\d}{\d x} g\right)^2 \frac{\d x}{x} \ge (\gamma-1)^2 \verti{g}_\gamma^2, \label{hardy_quant_1} \\
\verti{(D-1) g}_{\gamma-\frac 1 2}^2 &= \int_0^\infty x^{5 - 2 \gamma} \left(\frac{\d}{\d x} x^{-1} g\right)^2 \frac{\d x}{x} \ge \left(\gamma - \frac 5 2\right)^2 \verti{g}_{\gamma-\frac 1 2}^2, \label{hardy_quant_2} \\
\verti{(D-2) g}_{\gamma-1}^2 &= \int_0^\infty x^{8-2\gamma} \left(\frac{\d}{\d x} x^{-2} g\right)^2 \frac{\d x}{x} \ge \left(\gamma-4\right)^2 \verti{g}_{\gamma-1}^2. \label{hardy_quant_3}
\end{align}
\end{subequations}
Using \eqref{hardy_quant_2}, we recognize that the last term $- 2 \verti{g}_{\gamma-\frac 12}^2$ in the last line of \eqref{b_k_0} can be absorbed by the second term $2 \verti{(D-1) g}_{\gamma-\frac 1 2}^2$ of that line for $\gamma \in \left(0,\frac 3 2\right) \cup \left(\frac 7 2, \infty\right)$. For $\gamma \in \left[\frac 3 2, \frac 7 2\right]$, we may estimate by using Young's inequality
\begin{equation}\label{b_young}
2 \verti{g}_{\gamma-\frac 1 2}^2 \le c \verti{g}_{\gamma-1}^2 + c^{-1} \verti{g}_\gamma^2
\end{equation}
for some $c > 0$. Absorbing the right-hand side of \eqref{b_young} with the first and third term of the last line of \eqref{b_k_0} using \eqref{hardy_quant_1} and \eqref{hardy_quant_3}, we only need to fulfill the constraints
\[
(\gamma-1)^2 > c \quad \mbox{and} \quad (\gamma-4)^2 > c^{-1},
\]
which can be fulfilled if $L(\gamma) := (\gamma-1)^2 (\gamma-4)^2 > 1$. Note that this holds true since $L$ has its maximum at $\gamma = \frac 5 2$ and takes on its minima at $\gamma \in \left\{\frac 3 2, \frac 7 2\right\}$, where $L(\gamma) = (5/4)^2 > 1$.

\medskip

Because of $g = \BB^{-1} f$ we have proved \eqref{estinvB} for $k = 0$. The general case follows by an induction argument:

\medskip

Indeed, we may assume that \eqref{estinvB} holds for some $k \in \N_0$. Since we know from \eqref{estinvB} applied to $\BB D \BB^{-1} f$ that
\begin{equation}\label{ind_b_1}
 \verti{D \BB^{-1} f}_{k+1,\gamma} + \verti{D \BB^{-1} f}_{k+1,\gamma-1} \lesssim_{k,\gamma} \verti{\BB D \BB^{-1} f}_{k,\gamma},
\end{equation}
we need to understand the commutation properties between $\BB$ and the scaling-invariant derivative $D$. Observe that due to \eqref{factor_a} the operator identity
\[
\BB D = \left(x (D-2) +  D\right) D = (D-1) x (D-2) +  D^2 = D \BB - x (D-2)
\]
holds true. Thus we may conclude
\begin{equation}\label{ind_b_2}
\begin{aligned}
\verti{\BB D \BB^{-1} f}_{k,\gamma} &\lesssim_{k,\gamma} \verti{D f}_{k,\gamma} + \verti{(D-2) \BB^{-1} f}_{k,\gamma-1} \lesssim \verti{f}_{k+1,\gamma} + \verti{\BB^{-1} f}_{k+1,\gamma-1} \\
&\lesssim_{k,\gamma} \verti{f}_{k+1,\gamma},
\end{aligned}
\end{equation}
where the last estimate follows by the induction assumption. The combination of \eqref{ind_b_1}, \eqref{ind_b_2}, and the induction assumption finishes the proof.
\end{proof}

Now, we demonstrate how estimates~\eqref{elliptic_p} and \eqref{estinvB} of Lemmas~\ref{lem:elliptic} and \ref{lem:estinvB} can be used to prove Proposition~\ref{prop:ell_reg}: Suppose that $w$, $k$, $m$, and $\gamma$ are chosen as stated there. First observe that
\begin{equation}\label{appl_ell_t}
\verti{\tilde \AAA^m \tilde w}_{k,\gamma} \stackrel{\eqref{commute_a_op},\eqref{def_tc}}{=} \verti{(D-1) \AAA^m w}_{k,\gamma} \stackrel{\eqref{elliptic_p}}{\sim_{k,\gamma}} \begin{cases} \verti{\AAA^m w - \left(\AAA^m w\right)_1 x}_{k+1,\gamma} & \mbox{ for } \, \gamma > 1,\\ \verti{\AAA^m w}_{k+1,\gamma} & \mbox{ for } \, \gamma < 1,\end{cases}
\end{equation}
and
\begin{eqnarray}\nonumber
\verti{\check \AAA^m \check w}_{k,\gamma} &\stackrel{\eqref{commute_a_op},\eqref{def_tc}}{=}& \verti{(D-1) (D-2) \AAA^m w}_{k,\gamma} \\
&\stackrel{\eqref{elliptic_p}}{\sim_{k,\gamma}}& \begin{cases} \verti{\AAA^m w - \left(\AAA^m w\right)_1 x - \left(\AAA^m w\right)_2 x^2}_{k+2,\gamma} & \mbox{ for } \, \gamma > 2. \\ \verti{\AAA^m w - \left(\AAA^m w\right)_1 x}_{k+2,\gamma} & \mbox{ for } \, 2 > \gamma > 1. \\ \verti{\AAA^m w}_{k+2,\gamma} & \mbox{ for } \, \gamma < 1. \end{cases} \quad\qquad \label{appl_ell_c}
\end{eqnarray}
By an induction argument, we may without loss of generality study the following case:
\begin{lemma}\label{lem:elliptic_a}
For $k \in \N_0$ and $\gamma \in (-2,\infty) \setminus \Z$ we have
\begin{eqnarray}\nonumber
\lefteqn{ \verti{w - \sum_{j = 1}^{\floor{\gamma}+1} w_j x^j}_{k+4,\gamma+1} +  \verti{w - \sum_{j = 1}^{\floor{\gamma}+2} w_j x^j}_{k+4,\gamma+2}} \\
&\sim_{k,\gamma}&  \verti{\AAA w - \sum_{j = 1}^{\floor{\gamma}} (\AAA w)_j x^j}_{k,\gamma} +   \verti{w_{\floor\gamma+2}}, \label{elliptic_est_a}
\end{eqnarray}
where $w : [0,\infty) \to \R$ is a smooth function with $w = 0$ at $x = 0$.
\end{lemma}

\begin{proof}[Proof of Proposition~\ref{prop:ell_reg}]
Indeed, by noticing that
\[
\AAA x^j \stackrel{\eqref{defa}}{=} p(j) x^{j-1} +  q(j) x^{j-2}, \quad \mbox{that is,} \quad (\AAA w)_j = p(j+1) w_{j+1} +  q(j+2) w_{j+2},
\]
replacing $w$ with $\AAA w$ in \eqref{elliptic_est_a}, and using the original estimate \eqref{elliptic_est_a} twice afterwards, we get
\begin{eqnarray*}
\lefteqn{ \verti{w - \sum_{j = 1}^{\floor{\gamma}+2} w_j x^j}_{k+8,\gamma+2} +  \verti{w - \sum_{j = 1}^{\floor{\gamma}+3} w_j x^j}_{k+8,\gamma+3} +  \verti{w - \sum_{j = 1}^{\floor{\gamma}+4} w_j x^j}_{k+8,\gamma+4}} \\
&\sim_{k,\gamma}&  \verti{\AAA^2 w - \sum_{j = 1}^{\floor{\gamma}} (\AAA^2 w)_j x^j}_{k,\gamma} +  \verti{w_{\floor\gamma+3}} +  \verti{w_{\floor\gamma+4}}. \qquad\qquad
\end{eqnarray*}
Iterating this procedure yields
\begin{eqnarray}\nonumber
\lefteqn{\sum_{r = 0}^m  \verti{w - \sum_{j = 1}^{\floor{\gamma}+m+r} w_j x^j}_{k+4m,\gamma+m+r}} \\
&\sim_{k,m,\gamma}&  \verti{\AAA^m w - \sum_{j = 1}^{\floor{\gamma}} (\AAA^m w)_j x^j}_{k,\gamma} + \sum_{r = 1}^m  \verti{w_{\floor\gamma+m+r}}. \label{elliptic_est_a_m}
\end{eqnarray}
Now combining \eqref{appl_ell_t} or \eqref{appl_ell_c}, respectively, with \eqref{elliptic_est_a_m}, we arrive at estimates~\eqref{ell_main_1} and \eqref{ell_main_2}.
\end{proof}

\begin{proof}[Proof of Lemma~\ref{lem:elliptic_a}]
Observe
\begin{eqnarray}\nonumber
\verti{\AAA w - \sum_{j = 1}^{\floor{\gamma}} (\AAA w)_j x^j}_{k,\gamma} &\stackrel{\eqref{factor_a}}{=}& \verti{x^{-2} (D-1)^2 (D-2) \left(\BB w - \sum_{j = 3}^{\floor{\gamma}+2} (\BB w)_j x^j\right)}_{k,\gamma} \\
&\sim_k& \verti{(D-1)^2 (D-2) \left(\BB w - \sum_{j = 1}^{\floor{\gamma}+2} (\BB w)_j x^j\right)}_{k,\gamma+2} \nonumber\\
&\stackrel{\eqref{elliptic_p}}{\sim_\gamma}& \verti{\BB w - \sum_{j = 1}^{\floor{\gamma}+2} (\BB w)_j x^j}_{k+3,\gamma+2}, \label{est_a_b}
\end{eqnarray}
where we have used Lemma~\ref{lem:elliptic}. Now we may use
\begin{equation}\label{b_monom}
\BB x^j \stackrel{\eqref{factor_a_0},\eqref{factor_a}}{=} (j-2) x^{j+1} + j x^j, \quad \mbox{that is}, \quad (\BB w)_j = (j-3) w_{j-1} +  j w_j \quad \mbox{for} \quad j \in \N,
\end{equation}
where by assumption $w_0 = 0$. This leads to the identity
\begin{equation}\label{identity_b_0}
\BB w - \sum_{j = 1}^{\floor{\gamma}+2} (\BB w)_j x^j = \BB \left(w - \sum_{j = 1}^{\floor{\gamma}+2} w_j x^j\right) + w_{\floor\gamma+2} \floor\gamma x^{\floor\gamma+3},
\end{equation}
which can be rephrased as
\begin{equation}\label{identity_b}
w - \sum_{j = 1}^{\floor{\gamma}+2} w_j x^j = \BB^{-1} \left(\BB w - \sum_{j = 1}^{\floor{\gamma}+2} (\BB w)_j x^j\right) - w_{\floor\gamma+2} \floor\gamma \BB^{-1} x^{\floor\gamma+3}.
\end{equation}
We note that the term $\BB^{-1} x^{\floor\gamma+3}$ cannot be estimated using Lemma~\ref{lem:estinvB} as inserting any power of $x$ into the right-hand side of \eqref{estinvB} produces infinity. However, in this case we may use the explicit representation of $\BB^{-1}$ given by \eqref{rep_w}, so that we can infer
\begin{equation}\label{inv_b_mon}
\BB^{-1} x^{\floor\gamma+3} = (x+1)^2 \int_0^x \left(x^\prime+1\right)^{-3} (x^\prime)^{\floor\gamma+3} \, \frac{\d x^\prime}{x^\prime}
\end{equation}
and thus
\begin{eqnarray}\nonumber
\verti{\BB^{-1} x^{\floor\gamma+3}}_{\gamma+2}^2 &=& \int_0^\infty x^{-4-2\gamma} (x+1)^4 \left(\int_0^x \left(x^\prime+1\right)^{-3} (x^\prime)^{\floor\gamma+3} \, \frac{\d x^\prime}{x^\prime}\right)^2 \frac{\d x}{x} \\
&=&  \int_0^1 O\left(x^{2 - 2\gamma + 2 \floor\gamma}\right) \frac{\d x}{x} +  \int_1^\infty O\left(x^{- 2\gamma + 2 \floor\gamma}\right) \frac{\d x}{x} \nonumber \\
&\sim_\gamma& 1. \label{b_mon_est}
\end{eqnarray}
Furthermore,
\begin{equation}\label{b_mon_ind}
D \BB^{-1} x^{\floor\gamma+3} \stackrel{\eqref{inv_b_mon}}{=} (x+1)^{-1} \left(2 x \BB^{-1} + 1\right) x^{\floor\gamma+3}
\end{equation}
and since $\vertii{D^j x (x+1)^{-1}}_{BC^0\left([0,\infty)\right)} \lesssim_j 1$ and
\begin{eqnarray*}
\verti{(x+1)^{-1} x^{\floor\gamma+3}}_{\gamma+2}^2 &=& \int_0^\infty x^{2 - 2 \gamma + 2 \floor\gamma} (x+1)^{-2} \frac{\d x}{x} \\
&\sim& \int_0^1 x^{2-2\gamma+2\floor\gamma} \frac{\d x}{x} + \int_1^\infty x^{-2\gamma+2\floor\gamma} \frac{\d x}{x} \\
&\sim_\gamma& 1,
\end{eqnarray*}
an induction argument using \eqref{b_mon_est} and \eqref{b_mon_ind} shows
\begin{equation}\label{b_mon_est_2}
\verti{\BB^{-1} x^{\floor\gamma+3}}_{k+4,\gamma+2} \sim_{k,\gamma} 1.
\end{equation}
The combination of \eqref{est_a_b}, \eqref{identity_b}, and \eqref{b_mon_est_2} implies with help of Lemma~\ref{lem:estinvB}:
\begin{eqnarray*}
\lefteqn{ \verti{w - \sum_{j = 1}^{\floor{\gamma}+2} w_j x^j}_{k+4,\gamma+2}} \\
&\stackrel{\eqref{identity_b}}{\lesssim_\gamma}&  \verti{\BB^{-1} \left(\BB w - \sum_{j = 1}^{\floor{\gamma}+2} (\BB w)_j x^j\right)}_{k+4,\gamma+2} +  \verti{w_{\floor\gamma+2}} \verti{\BB^{-1} x^{\floor\gamma+3}}_{k+4,\gamma+2} \\
&\stackrel{\eqref{estinvB},\eqref{b_mon_est_2}}{\lesssim_{k,\gamma}}& \verti{\BB w - \sum_{j = 1}^{\floor{\gamma}+2} (\BB w)_j x^j}_{k+3,\gamma+2} +  \verti{w_{\floor\gamma+2}} \\
&\stackrel{\eqref{est_a_b}}{\lesssim_\gamma}& \verti{\AAA w - \sum_{j = 1}^{\floor{\gamma}} (\AAA w)_j x^j}_{k,\gamma} + \verti{w_{\floor\gamma+2}}.
\end{eqnarray*}

\medskip

Next, we repeat the same steps for the norm $\verti{w - \sum_{j = 1}^{\floor{\gamma}+1} w_j x^j}_{k+4,\gamma+1}$. Using \eqref{b_monom} and \eqref{identity_b_0}, we have
\begin{equation}\label{identity_b_2}
w - \sum_{j = 1}^{\floor{\gamma}+1} w_j x^j = \BB^{-1} \left(\BB w - \sum_{j = 1}^{\floor{\gamma}+2} (\BB w)_j x^j\right) +  \left(\floor\gamma+2\right) w_{\floor\gamma+2} \BB^{-1} x^{\floor\gamma+2}.
\end{equation}
Employing \eqref{b_mon_est_2} with $\gamma$ replaced by $\gamma-1$, we have
\begin{equation}\label{b_mon_est_3}
\verti{\BB^{-1} x^{\floor\gamma+2}}_{k+4,\gamma+1} \sim_{k,\gamma} 1.
\end{equation}
Hence, \eqref{est_a_b}, \eqref{identity_b_2}, and \eqref{b_mon_est_3} give
\begin{eqnarray*}
\lefteqn{\verti{w - \sum_{j = 1}^{\floor{\gamma}+1} w_j x^j}_{k+4,\gamma+1}} \\
&\stackrel{\eqref{identity_b_2}}{\lesssim_\gamma}& \verti{\BB^{-1} \left(\BB w - \sum_{j = 1}^{\floor{\gamma}+2} (\BB w)_j x^j\right)}_{k+4,\gamma+1} +  \verti{w_{\floor\gamma+2}} \verti{\BB^{-1} x^{\floor\gamma+2}}_{k+4,\gamma+1} \\
&\stackrel{\eqref{estinvB},\eqref{b_mon_est_3}}{\lesssim_{k,\gamma}}& \verti{\BB w - \sum_{j = 1}^{\floor{\gamma}+2} (\BB w)_j x^j}_{k+3,\gamma+2} +  \verti{w_{\floor\gamma+2}} \\
&\stackrel{\eqref{est_a_b}}{\lesssim_\gamma}& \verti{\AAA w - \sum_{j = 1}^{\floor{\gamma}} (\AAA w)_j x^j}_{k,\gamma} +  \verti{w_{\floor\gamma+2}},\end{eqnarray*}
thus finishing the proof of one direction of estimate~\eqref{elliptic_est_a}.

\medskip

For proving the other direction of estimate~\eqref{elliptic_est_a}, we first observe that \eqref{factor_a} and \eqref{b_monom} imply
\begin{align*}
\BB w - \sum_{j = 1}^{\floor\gamma+2} (\BB w)_j x^j &= x (D-2) w +  D w - \sum_{j = 2}^{\floor\gamma+2} \left((j-3) w_{j-1} +  j w_j\right) x^j \\
&= x (D-2) \left(w - \sum_{j = 1}^{\floor\gamma+1} w_j x^j\right) +  D \left(w - \sum_{j = 1}^{\floor\gamma+2} w_j x^j\right),
\end{align*}
so that \eqref{est_a_b} yields
\begin{align*}
\verti{\AAA w - \sum_{j = 1}^{\floor{\gamma}} (\AAA w)_j x^j}_{k,\gamma} &\lesssim_\gamma \verti{\BB w - \sum_{j = 1}^{\floor{\gamma}+2} (\BB w)_j x^j}_{k+3,\gamma+2} \\
&\lesssim_k \verti{w - \sum_{j = 1}^{\floor\gamma+1} w_j x^j}_{k+4,\gamma+1} +  \verti{w - \sum_{j = 1}^{\floor\gamma+2} w_j x^j}_{k+4,\gamma+2}.
\end{align*}
Hence, it remains to bound $ \verti{w_{\floor\gamma+2}}$: Observe that by quite elementary arguments we have
\begin{align*}
 \verti{w_{\floor\gamma+2}}^2 &\lesssim \int_1^2 \left(w(x) - \sum_{j = 1}^{\floor\gamma+1} w_j x^j\right)^2 \d x + \int_1^2 \left(w(x) - \sum_{j = 1}^{\floor\gamma+2} w_j x^j\right)^2 \d x \\
&\lesssim_\gamma \int_0^\infty x^{-2\gamma-2} \left(w(x) - \sum_{j = 1}^{\floor\gamma+1} w_j x^j\right)^2 \frac{\d x}{x} + \int_0^\infty x^{-2\gamma-4} \left(w(x) - \sum_{j = 1}^{\floor\gamma+2} w_j x^j\right)^2 \frac{\d x}{x} \\
&=  \verti{w - \sum_{j = 1}^{\floor\gamma+1} w_j x^j}_{\gamma+1}^2 +  \verti{w - \sum_{j = 1}^{\floor\gamma+2} w_j x^j}_{\gamma+2}^2,
\end{align*}
thus completing the proof of Lemma~\ref{lem:elliptic_a}.
\end{proof}

\subsection{Parabolic maximal regularity II\label{sec:maxreg2}}
\subsubsection{Definition of norms}
Our aim is to show maximal-regularity estimates for the solution $u$ to arbitrary orders of the expansion given in \eqref{formal_exp_u}. This can be achieved by combining the parabolic estimates \eqref{maxreg_tc} in conjunction with \eqref{comp_mr} (cf.~Proposition~\ref{prop:maxreg1}) and \eqref{ell_main} (cf.~Proposition~\ref{prop:ell_reg}), where $k$ is replaced by $k + 4 (N - m - \ell)$ in \eqref{maxreg_ta_lm} and $k + 4 (N - m - \ell) - 1$ in \eqref{maxreg_ca_lm}. The resulting parabolic estimate reads
\begin{equation}\label{maxreg_t_ell}
\vertiii{u}_{\mathrm{sol}} \lesssim_{k,N,\delta} \vertiii{u^{(0)}}_{\mathrm{init}} + \vertiii{f}_{\mathrm{rhs}} + R
\end{equation}
with norms $\vertiii{\cdot}_{\mathrm{sol}}$, $\vertiii{\cdot}_{\mathrm{init}}$, and $\vertiii{\cdot}_{\mathrm{rhs}}$ for the solution $u$, the initial data $u^{(0)}$ and the right-hand side $f$, respectively, and a remainder term $R$ due to the summed absolute values of the coefficients on the right-hand sides of \eqref{ell_main}. Therefore, we introduce the index sets
\begin{subequations}\label{index_sets}
\begin{equation}\label{index_sets_a}
\II_{N,\delta} := \left\{\left(\alpha, \ell, m\right) : \, \alpha \in \{\delta,1+\delta\} \, \mbox{ and } \, \ell, m \in \N_0 \, \mbox{ with } \, 0 \le \ell + m \le N - \floor\alpha\right\}
\end{equation}
and
\begin{equation}
\JJ_{N,\delta} := \II_{N,\delta} \cup \left\{\left(\alpha,\ell,m\right): \, \left(\alpha + \frac 1 2, \ell, m\right) \in \II_{N,\delta}\right\}.
\end{equation}
\end{subequations}
Then we can define our norms in compact form:

\medskip

The norm for the solution $u$ is defined through
\begin{eqnarray}\nonumber
\vertiii{u}^2_{\mathrm{sol}} &:=& \sum_{\left(\alpha,\ell,m\right) \in \II_{N,\delta}} \sum_{r = 0}^m  \sup_{t \ge 0} \verti{\partial_t^\ell u - \sum_{j = 1}^{\floor{\alpha}+m+r} \frac{\d^\ell u_j}{\d t^\ell} x^j}_{k+4 (N - \ell)+1,\alpha+m+r}^2 \\
&&+ \sum_{\left(\alpha,\ell,m\right) \in \JJ_{N,\delta}} \sum_{r = 0}^m  \int_0^\infty \verti{\partial_t^{\ell+1} \underline u - \sum_{j = 1}^{\floor{\alpha}+m+r-1} \frac{\d^{\ell+1} \underline u_j}{\d t^{\ell+1}} x^j}_{k+4 (N - \ell)-1,\alpha + m + r - 1}^2 \d t \nonumber\\
&&+ \sum_{\left(\alpha,\ell,m\right) \in \JJ_{N,\delta}} \sum_{r = 0}^m  \int_0^\infty \verti{\partial_t^\ell u - \sum_{j = 1}^{\floor{\alpha}+m+r+1} \frac{\d^\ell u_j}{\d t^\ell} x^j}_{k+4 (N - \ell)+3,\alpha + m + r + 1}^2 \d t. \nonumber\\
\label{norm_sol}
\end{eqnarray}
Here, $N \in \N_0$ and the choice of $k \ge 3$ will be addressed later. For the right-hand side $f$ we are first led to choose instead of $\vertiii{\cdot}_{\mathrm{rhs}}$ the squared norm
\begin{eqnarray}\nonumber
\vertiii{f}_{\mathrm{rhs},*}^2 &:=& \sum_{\left(\alpha,\ell,m\right) \in \II_{N,\delta}} \sum_{\ell^\prime=0}^{\ell-1} \sup_{t \ge 0} \sum_{r = 0}^{m+\ell^\prime} \Bigg|\partial_t^{\ell-1-\ell^\prime} f_{|t=0} \\
&& - \sum_{j = 1}^{\floor\alpha+m+\ell^\prime+r} \left(\frac{\d^{\ell-1-\ell^\prime} f_j}{\d t^{\ell-1-\ell^\prime}}\right)_{|t=0} x^j \Bigg|_{k+4 \left(N+\ell^\prime-\ell \right)+1,\alpha+m+\ell^\prime+r}^2 \nonumber \\
&& + \sum_{\left(\alpha,\ell,m\right) \in \JJ_{N,\delta}} \int_0^\infty \sum_{r = 0}^m  \verti{\partial_t^\ell \underline f - \sum_{j = 1}^{\floor{\alpha}+m+r-1} \frac{\d^\ell \underline f_j}{\d t^\ell} x^j}_{k+4 (N - \ell)-1,\alpha+m+r -1}^2 \d t. \nonumber\\
\label{norm_rhs_0}
\end{eqnarray}
Notice that the first two lines in \eqref{norm_rhs_0} originate from inserting the time-trace identities \eqref{comp_mr} into \eqref{maxreg_tc}. The expression in \eqref{norm_rhs_0} can be simplified by setting $\tilde \ell = \ell-1-\ell^\prime$, $\tilde m = m + \ell^\prime = m + \ell - 1 - \tilde \ell$, so that
\begin{align*}
&\sum_{\left(\alpha,\ell,m\right) \in \II_{N,\delta}} \sum_{\ell^\prime=0}^{\ell-1}  \sum_{r = 0}^{m+\ell^\prime} \verti{\partial_t^{\ell-1-\ell^\prime} f_{|t=0} - \sum_{j = 1}^{\floor\alpha+m+\ell^\prime+r} \left(\frac{\d^{\ell-1-\ell^\prime} f_j}{\d t^{\ell-1-\ell^\prime}}\right)_{|t=0} x^j}_{k+4\left(N+\ell^\prime-\ell\right)+1,\alpha+m+\ell^\prime+r}^2 \\
&\quad \sim_N \sum_{\left(\alpha,\tilde\ell,\tilde m\right) \in \II_{N-1,\delta}} \sum_{\tilde r = 0}^{\tilde m} \verti{\partial_t^{\tilde \ell} f_{|t=0} - \sum_{j = 1}^{\floor\alpha+\tilde m+\tilde r} \left(\frac{\d^{\tilde \ell} f_j}{\d t^{\tilde \ell}}\right)_{|t=0} x^j}_{k+ 4 \left(N - \tilde \ell\right) - 3,\alpha+ \tilde m + \tilde r}^2.
\end{align*}
Hence, we may choose instead of \eqref{norm_rhs_0} the slightly stronger norm
\begin{eqnarray}\nonumber
\vertiii{f}_{\mathrm{rhs}}^2 &:=& \sum_{\left(\alpha,\tilde\ell,\tilde m\right) \in \II_{N-1,\delta}} \sum_{\tilde r = 0}^{\tilde m} \sup_{t \ge 0} \verti{\partial_t^{\tilde \ell} f - \sum_{j = 1}^{\floor\alpha+\tilde m+\tilde r} \frac{\d^{\tilde \ell} f_j}{\d t^{\tilde \ell}} x^j}_{k+ 4 \left(N - \tilde \ell\right) -3,\alpha+ \tilde m + \tilde r}^2 \nonumber \\
&& + \sum_{\left(\alpha,\ell,m\right) \in \JJ_{N,\delta}} \sum_{r = 0}^m \int_0^\infty \verti{\partial_t^\ell \underline f - \sum_{j = 1}^{\floor{\alpha}+m+r-1} \frac{\d^\ell \underline f_j}{\d t^\ell} x^j}_{k+4 (N - \ell)-1,\alpha+m+r -1}^2 \d t. \nonumber\\
\label{norm_rhs}
\end{eqnarray}

\medskip

Because we have used the time-trace identities \eqref{comp_mr}, the initial data norm reads
\begin{equation}\label{norm_init}
\vertiii{u^{(0)}}_{\mathrm{init}}^2 := \sum_{\left(\alpha,\cdot,m\right) \in \II_{N,\delta}} \sum_{r = 0}^m \verti{u^{(0)} - \sum_{j = 1}^{\floor{\alpha}+m+r} u^{(0)}_j x^j}_{k+4 N+1,\alpha+m+r}^2.
\end{equation}
As for $m \ge 1$ in \eqref{ell_main} remnant coefficient terms appear, our considerations up to now lead   to a remainder of the form
\begin{eqnarray}\nonumber
R^2 &=& \sum_{\left(\alpha,\ell,m\right) \in \II_{N,\delta}} \sum_{r = 1}^m \sup_{t \ge 0} \verti{\frac{\d^\ell u_{\floor{\alpha}+m+r}}{\d t^\ell}}^2 \\
&&+ \sum_{\left(\alpha,\ell,m\right) \in \JJ_{N,\delta}} \sum_{r = 1}^m  \int_0^\infty \verti{\frac{\d^{\ell+1} \underline u_{\floor{\alpha}+m+r-1}}{\d t^{\ell+1}}}^2 \d t \nonumber \\
&&+ \sum_{\left(\alpha,\ell,m\right) \in \JJ_{N,\delta}} \sum_{r = 1}^m  \int_0^\infty \verti{\frac{\d^\ell u_{\floor{\alpha}+m+r+1}}{\d t^\ell}}^2 \d t . \nonumber
\end{eqnarray}
In view of the definitions of $\II_{N,\delta}$ and $\JJ_{N,\delta}$ in \eqref{index_sets}, we can deduce the bound
\begin{eqnarray}\nonumber
R^2 &\lesssim_N& \sum_{j = 1}^{2N} \sum_{\ell = 0}^{\floor{\frac{2 N-j}{2}}} \sup_{t \ge 0}  \verti{\frac{\d^\ell u_j}{\d t^\ell}}^2 + \sum_{j = 1}^{2N-1} \sum_{\ell = 0}^{\floor{\frac{2 N - 1 -j}{2}}} \int_0^\infty \verti{\frac{\d^{\ell+1} \underline u_j}{\d t^{\ell+1}}}^2 \d t \\
&& + \sum_{j = 1}^{2N+1} \sum_{\ell = 0}^{\floor{\frac{2 N + 1 -j}{2}}}\int_0^\infty \verti{\frac{\d^\ell u_j}{\d t^\ell}}^2 \d t \nonumber\\
&\stackrel{\eqref{decomp_w},\eqref{coeff_uw}}{\lesssim_N}& \sum_{j = 1}^{2N} \sum_{\ell = 0}^{\floor{\frac{2 N -j}{2}}}\sup_{t \ge 0}  \verti{\frac{\d^\ell u_j}{\d t^\ell}}^2 + \sum_{j = 1}^{2 N - 1} \sum_{\ell = 0}^{\floor{\frac{2 N - 1 -j}{2}}} \int_0^\infty \verti{\frac{\d^{\ell+1} u_j}{\d t^{\ell+1}}}^2 \d t \nonumber\\
&& + \sum_{j = 1}^{2N + 1} \sum_{\ell = 0}^{\floor{\frac{2 N + 1 -j}{2}}} \int_0^\infty \verti{\frac{\d^\ell u_j}{\d t^\ell}}^2 \d t \nonumber \\
&\lesssim& \sum_{j = 1}^{2N} \sum_{\ell = 0}^{\floor{\frac{2 N -j}{2}}}\sup_{t \ge 0}  \verti{\frac{\d^\ell u_j}{\d t^\ell}}^2 + \sum_{j = 1}^{2 N + 1} \sum_{\ell = 0}^{\floor{\frac{2 N + 1 -j}{2}}} \int_0^\infty \verti{\frac{\d^\ell u_j}{\d t^\ell}}^2 \d t. \label{bd_remnant}
\end{eqnarray}
The fact that coefficients of the solution $u$ appear on the right-hand side of \eqref{maxreg_t_ell} is inconvenient. Furthermore, highest-order terms like $\sup_{t \ge 0} \verti{u_{2 N}}^2$, appearing in the second but last sum in \eqref{bd_remnant}, cannot be controlled with a trace estimate using the $L^2$-parts in the last term only. In what follows, we first demonstrate how these remnant terms can be estimated and thus simplified using the fact that expansions of $u$ and $f$ have to meet equation~\eqref{lin_pde}, leading to an infinite-dimensional system of ODEs.

\subsubsection{The polynomial equation}
We start by noting that
\[
\AAA x^j \stackrel{\eqref{defa}}{=} p(j) x^{j-1} +  q(j) x^{j-2} \qquad \Rightarrow \qquad (\AAA u)_j = p(j+1) u_{j+1} +  q(j+2) u_{j+2}
\]
and therefore, if $u, f \in C^\infty([0,\infty))$ meet equation~\eqref{lin_pde}, then the polynomial equation
\begin{equation}\label{poly_eq}
\frac{\d u_j}{\d t} + p(j+1) u_{j+1} +  q(j+2) u_{j+2} = f_j \quad \mbox{for} \quad t > 0 \quad \mbox{and} \quad j \in \N
\end{equation}
is satisfied. Notice that \eqref{poly_eq} is an infinite-dimensional first-order ODE for $(u_1, u_2, u_3, \ldots)$ with right-hand side $(f_1, f_2, f_3, \ldots)$ and initial data $\left(u^{(0)}_1, u^{(0)}_2, u^{(0)}_3, \ldots\right)$. We can prove:
\begin{lemma}\label{lem:poly_eq}
Suppose that $u,f \in C^\infty([0,\infty))$ meet equation~\eqref{lin_pde}, then we have
\begin{equation}\label{poly_coeff}
\sum_{j = 1}^M \sum_{\ell = 0}^{\floor{\frac{M-j}{2}}}  \verti{\frac{\d^\ell u_j}{\d t^\ell}} \lesssim_M \sum_{j = 1}^{M-2} \sum_{\ell = 0}^{\floor{\frac{M-2-j}{2}}}  \verti{\frac{\d^\ell f_j}{\d t^\ell}} + \sum_{\ell = 0}^{\floor{\frac{M-1}{2}}}  \verti{\frac{\d^\ell u_1}{\d t^\ell}} + \sum_{\ell = 0}^{\floor{\frac{M-2}{2}}}  \verti{\frac{\d^\ell u_2}{\d t^\ell}}
\end{equation}
for any $M \in \N$ with $M \ge 1$, where the constant in \eqref{poly_coeff} only depends on $M$.
\end{lemma}
\begin{proof}
The proof is a simple induction argument employing \eqref{poly_eq}. For $M = 1$ and $M = 2$ there is nothing to show.

\medskip

Now suppose that \eqref{poly_coeff} holds up to some fixed $M \ge 2$.  Differentiating \eqref{poly_eq} in time, we have
\begin{equation}\label{ind_coeff}
\verti{\frac{\d^\ell u_{j+2}}{\d t^\ell}} \lesssim_j \verti{\frac{\d^\ell f_j}{\d t^\ell}} + \verti{\frac{\d^\ell u_{j+1}}{\d t^\ell}} + \verti{\frac{\d^{\ell+1} u_j}{\d t^{\ell+1}}} \quad \mbox{for} \quad j \in \N.
\end{equation}
Now consider the term on the left-hand side of \eqref{poly_coeff} with $M$ replaced by $M+1$ and observe
\[
\sum_{j = 1}^{M+1} \sum_{\ell = 0}^{\floor{\frac{M+1-j}{2}}} \verti{\frac{\d^\ell u_j}{\d t^\ell}} = \sum_{\ell = 0}^{\floor{\frac M 2}} \verti{\frac{\d^\ell u_1}{\d t^\ell}} + \sum_{\ell = 0}^{\floor{\frac{M-1}{2}}} \verti{\frac{\d^\ell u_2}{\d t^\ell}} + \sum_{j = 3}^{M+1} \sum_{\ell = 0}^{\floor{\frac{M+1-j}{2}}} \verti{\frac{\d^\ell u_j}{\d t^\ell}},
\]
Since the first two sums on the right-hand side appear in \eqref{poly_coeff} with $M$ replaced by $M+1$, it remains to estimate the last term. Using \eqref{ind_coeff}, we get
\begin{equation}\label{coeff_ind_1}
\sum_{j = 3}^{M+1} \sum_{\ell = 0}^{\floor{\frac{M+1-j}{2}}} \verti{\frac{\d^\ell u_j}{\d t^\ell}} \lesssim_M \sum_{j = 1}^{M-1} \sum_{\ell = 0}^{\floor{\frac{M-1-j}{2}}} \verti{\frac{\d^\ell f_j}{\d t^\ell}} + \sum_{j = 2}^M \sum_{\ell = 0}^{\floor{\frac{M-j}{2}}} \verti{\frac{\d^\ell u_j}{\d t^\ell}} + \sum_{j = 1}^{M-1} \sum_{\ell = 1}^{\floor{\frac{M+1-j}{2}}} \verti{\frac{\d^\ell u_j}{\d t^\ell}}.
\end{equation}
The first term on the right-hand side of this estimate is bounded by the first term on the right-hand side of \eqref{poly_coeff} with $M$ replaced by $M+1$, while the second is controlled by the left-hand side of the original estimate \eqref{poly_coeff}, that is, it is treated by the induction assumption. The addends with $j \in \{1,2\}$ in the last term of \eqref{coeff_ind_1} appear on the right-hand side of \eqref{poly_coeff} with $M$ replaced by $M+1$ as well, so that it remains to bound the sum $\sum_{j = 3}^{M-1} \sum_{\ell = 1}^{\floor{\frac{M+1-j}{2}}} \verti{\frac{\d^\ell u_j}{\d t^\ell}}$. Applying \eqref{ind_coeff}, we have:
\[
\sum_{j = 3}^{M-1} \sum_{\ell = 1}^{\floor{\frac{M+1-j}{2}}} \verti{\frac{\d^\ell u_j}{\d t^\ell}} \lesssim_M \sum_{j = 1}^{M-3} \sum_{\ell = 0}^{\floor{\frac{M-1-j}{2}}} \verti{\frac{\d^\ell f_j}{\d t^\ell}} + \sum_{j = 2}^{M-2} \sum_{\ell = 0}^{\floor{\frac{M-j}{2}}} \verti{\frac{\d^\ell u_j}{\d t^\ell}} + \sum_{j = 1}^{M-3} \sum_{\ell = 0}^{\floor{\frac{M+1-j}{2}}} \verti{\frac{\d^\ell u_j}{\d t^\ell}},
\]
where the first two terms on the left-hand side are treated as before. The last term is the same as in \eqref{coeff_ind_1} except that $j \le M-3$ instead of $j \le m-1$. Iterating this procedure reduces this last term to $j \in \{1,2\}$ only, which appear on the right-hand side of \eqref{poly_coeff} with $M$ replaced by $M+1$. This concludes the induction step.
\end{proof}
Lemma~\ref{lem:poly_eq} implies the bound
\begin{eqnarray}\nonumber
R^2 &\stackrel{\eqref{bd_remnant}}{\lesssim_N}& \sum_{j = 1}^{2 N - 2} \sum_{\ell = 0}^{\floor{\frac{2 N - 2 - j}{2}}} \sup_{t \ge 0} \verti{\frac{\d^\ell f_j}{\d t^\ell}}^2 + \sum_{j = 1}^{2 N - 1} \sum_{\ell = 0}^{\floor{\frac{2 N - 1 - j}{2}}} \int_0^\infty \verti{\frac{\d^\ell f_j}{\d t^\ell}}^2 \d t  \\
&& + \sum_{\ell = 0}^{N-1} \sup_{t \ge 0} \verti{\frac{\d^\ell u_1}{\d t^\ell}}^2 + \sum_{\ell = 0}^N \int_0^\infty \verti{\frac{\d^\ell u_1}{\d t^\ell}}^2 \d t + \sum_{\ell = 0}^{N-1} \sup_{t \ge 0} \verti{\frac{\d^\ell u_2}{\d t^\ell}}^2 + \sum_{\ell = 0}^{N-1} \int_0^\infty \verti{\frac{\d^\ell u_2}{\d t^\ell}}^2 \d t. \nonumber\\
\label{bound_r_poly}
\end{eqnarray}
For treating the terms in the first line of \eqref{bound_r_poly}, observe that by the first line of \eqref{norm_rhs} we have
\[
\vertiii{f}_{\mathrm{rhs}}^2 \gtrsim_N \sum_{\left(\alpha,\tilde\ell,\tilde m\right) \in \II_{N-1,\delta}} \sum_{\tilde r = 0}^{\tilde m} \sup_{t \ge 0} \verti{\frac{\d^{\tilde\ell} f_{\floor{\alpha}+\tilde m+\tilde r}}{\d t^{\tilde \ell}}}^2 \gtrsim_N \sum_{j = 1}^{2 N - 2} \sum_{\ell = 0}^{\floor{\frac{2 N-2-j}{2}}} \sup_{t \ge 0} \verti{\frac{\d^\ell f_j}{\d t^\ell}}^2,
\]
while the second line of \eqref{norm_rhs} yields
\begin{eqnarray*}
\vertiii{f}_{\mathrm{rhs}}^2 &\gtrsim_N& \sum_{\left(\alpha,\ell,m\right) \in \II_{N,\delta}} \sum_{r = 0}^m \int_0^\infty \verti{\frac{\d^\ell \underline f_{\floor{\alpha}+m+r-1}}{\d t^\ell}}^2 \d t \gtrsim_N \sum_{j = 1}^{2 N - 1} \sum_{\ell = 0}^{\floor{\frac{2 N - 1 - j}{2}}} \int_0^\infty \verti{\frac{\d^\ell \underline f_j}{\d t^\ell}}^2 \d t \\
&\stackrel{\eqref{decomp_w},\eqref{coeff_uw}}{\gtrsim_N}& \sum_{j = 1}^{2 N - 1} \sum_{\ell = 0}^{\floor{\frac{2 N - 1 - j}{2}}} \int_0^\infty \verti{\frac{\d^\ell f_j}{\d t^\ell}}^2 \d t.
\end{eqnarray*}
Together with \eqref{bound_r_poly} this implies
\begin{eqnarray}\nonumber
R^2 &\lesssim_N& \vertiii{f}_{\mathrm{rhs}}^2 \\
&& + \sum_{\ell = 0}^{N-1} \sup_{t \ge 0} \verti{\frac{\d^\ell u_1}{\d t^\ell}}^2 + \sum_{\ell = 0}^N \int_0^\infty \verti{\frac{\d^\ell u_1}{\d t^\ell}}^2 \d t + \sum_{\ell = 0}^{N-1} \sup_{t \ge 0} \verti{\frac{\d^\ell u_2}{\d t^\ell}}^2 + \sum_{\ell = 0}^{N-1} \int_0^\infty \verti{\frac{\d^\ell u_2}{\d t^\ell}}^2 \d t. \nonumber\\
\label{bound_r_poly_2}
\end{eqnarray}
For treating the remaining terms in $\frac{\d^\ell u_1}{\d t^\ell}$ and $\frac{\d^\ell u_2}{\d t^\ell}$ observe that using Lemma~\ref{lem:elliptic} in Proposition~\ref{prop:maxreg1} with $m = 0$, Proposition~\ref{prop:ell_reg} is not applied a single time and we obtain \eqref{maxreg_t_ell} with $R \equiv 0$ if we take $m = 0$ in all the sums $\sum_{\left(\alpha,\ell,m\right) \in \II_{N,\delta}} (\ldots)$ and $\sum_{\left(\alpha,\ell,m\right) \in \JJ_{N,\delta}} (\ldots)$, and  $\tilde m = 0$ in all the sums $\sum_{\left(\alpha,\tilde\ell,\tilde m\right) \in \II_{N-1,\delta}} (\ldots)$ in the definition of the norms \eqref{norm_sol}, \eqref{norm_rhs}, and \eqref{norm_init}. On the other hand, the resulting norm replacing $\vertiii{u}$ controls the remnant terms in the second line of \eqref{bound_r_poly_2}, because
\begin{equation}\label{est_coeff_u1}
\verti{\frac{\d^\ell u_1}{\d t^\ell}}^2 \lesssim \int_{\frac 1 2}^2 \left(\partial_t^\ell u - \frac{\d^\ell u_1}{\d t^\ell} x\right)^2 \d x + \int_{\frac 1 2}^2 \left(\partial_t^\ell u\right)^2 \d x \lesssim \verti{\partial_t^\ell u - \frac{\d^\ell u_1}{\d t^\ell} x}_{1+\delta}^2 + \verti{\partial_t^\ell u}_\delta^2
\end{equation}
and
\begin{eqnarray}\nonumber
\verti{\frac{\d^\ell u_2}{\d t^\ell}}^2 &\lesssim& \int_{\frac 1 2}^2 \left(\partial_t^\ell u - \frac{\d^\ell u_1}{\d t^\ell} x - \frac{\d^\ell u_2}{\d t^\ell} x^2\right)^2 \d x + \int_{\frac 1 2}^2 \left(\partial_t^\ell u - \frac{\d^\ell u_1}{\d t^\ell} x\right)^2 \d x \\
&\lesssim& \verti{\partial_t^\ell u - \frac{\d^\ell u_1}{\d t^\ell} x - \frac{\d^\ell u_2}{\d t^\ell} x^2}_{2+\delta}^2 + \verti{\partial_t^\ell u - \frac{\d^\ell u_1}{\d t^\ell} x}_{1+\delta}^2. \label{est_coeff_u2}
\end{eqnarray}
Thus we have concluded all arguments leading to
\begin{equation}\label{maxreg_full}
\vertiii{u}_{\mathrm{sol}} \lesssim_{k,N,\delta} \vertiii{u^{(0)}}_{\mathrm{init}} + \vertiii{f}_{\mathrm{rhs}}
\end{equation}
instead of \eqref{maxreg_t_ell} upon increasing the constant in the estimate. This is the desired maximal-regularity estimate with arbitrary regularity of the solution close to the free boundary $\{x = 0\}$. We summarize our findings in the following statement:
\begin{proposition}\label{prop:mr_full}
Take $N \in \N_0$, $k \ge 2$, and $0 < \delta < \frac 1 2$, and suppose that $f: (0,\infty)^2 \to \R$ and $u^{(0)} : (0,\infty) \to \R$ are locally integrable with $\vertiii{f}_{\mathrm{rhs}} < \infty$ and $\vertiii{u^{(0)}}_{\mathrm{init}} < \infty$. Then problem \eqref{linear_u} has exactly one locally integrable solution $u: (0,\infty)^2 \to \R$ with $\vertiii{u}_{\mathrm{sol}} < \infty$. This solution fulfills the maximal-regularity estimate \eqref{maxreg_full}.
\end{proposition}
A rigorous justification of Proposition~\ref{prop:mr_full} only requires to make the heuristic arguments in \S\ref{sec:maxreg} mathematically precise. The respective reasoning will be detailed in \S\ref{sec:lin_res} and \S\ref{sec:lin_rig}.

\subsection{Properties of the norms\label{sec:norms}}
The following estimates for the coefficients follow by an elementary reasoning as in \eqref{est_coeff_u1} or \eqref{est_coeff_u2} (see also \cite[Lem.~4.3]{ggko.2014} or \cite[Lem.~3.3]{g.2016}):
\begin{lemma}\label{lem:est_coeff}
Suppose that $k, N \in \N_0$ and $0 < \delta < \frac 1 2$. Further suppose that $u: (0,\infty)^2 \to \R$, $f: (0,\infty)^2 \to \R$ and $u^{(0)} : (0,\infty) \to \R$ are locally integrable with finite norms $\vertiii{u}_{\mathrm{sol}}$, $\vertiii{f}_{\mathrm{rhs}}$, and $\vertiii{u^{(0)}}_{\mathrm{init}}$. Then the following estimates with constants independent of $u$, $f$, and $u^{(0)}$ hold true:
\begin{subequations}\label{coeff_est}
\begin{align}
\sup_{t \ge 0} \verti{\frac{\d^\ell u_j}{\d t^\ell}} & \lesssim_{k,N,\delta} \vertiii{u}_{\mathrm{sol}} \quad \mbox{for} \quad j = 1,\ldots,2 (N-\ell), \quad \ell = 0,\ldots,N, \label{coeff_est_u_sup}\\
\int_0^\infty \verti{\frac{\d^\ell u_j}{\d t^\ell}}^2 \d t & \lesssim_{k,N,\delta} \vertiii{u}_{\mathrm{sol}}^2 \quad \mbox{for} \quad j = 1,\ldots,2(N-\ell)+1, \quad \ell = 0,\ldots,N, \label{coeff_est_u_l2} \\
\verti{u_j^{(0)}} &\lesssim_{k,N,\delta} \vertiii{u^{(0)}}_{\mathrm{init}} \quad \mbox{for} \quad j = 1,\ldots,2N, \label{coeff_est_u0}\\
\sup_{t \ge 0} \verti{\frac{\d^{\tilde\ell} f_j}{\d t^{\tilde\ell}}} & \lesssim_{k,N,\delta} \vertiii{f}_{\mathrm{rhs}} \quad \mbox{for} \quad j = 1,\ldots,2 \left(N-1-\tilde\ell\right), \quad \tilde\ell = 0,\ldots,N-1, \\
\int_0^\infty \verti{\frac{\d^\ell f_j}{\d t^\ell}}^2 \d t & \lesssim_{k,N,\delta} \vertiii{f}_{\mathrm{rhs}}^2 \quad \mbox{for} \quad j = 1,\ldots,2\left(N-\ell\right)-1, \quad \ell = 0,\ldots,N.
\end{align}
\end{subequations}
\end{lemma}
Lemma~\ref{lem:est_coeff} will be useful for the treatment of the nonlinear problem \eqref{nonlinear_u} in \S\ref{sec:nonlin}. Furthermore, it enables us to properly define suitable function spaces for our solution $u$, the initial data $u^{(0)}$, and the right-hand side $f$:
\begin{definition}\label{def:function_spaces}
For $k \ge 0$, $N \ge 1$, and $0 < \delta < \frac 1 2$ we define the spaces $U = U(k,N,\delta)$, $U_0 = U_0(k,N,\delta)$, and $F = F(k,N,\delta)$ as the closure of all smooth $u: \, [0,\infty)^2 \to \R$, $u^{(0)}: \, [0,\infty) \to \R$, or $f: \, [0,\infty)^2 \to \R$ with $\vertiii{u}_{\mathrm{sol}} < \infty$, $\vertiii{u^{(0)}}_{\mathrm{init}} < \infty$, or $\vertiii{f}_{\mathrm{rhs}} < \infty$, respectively.
\end{definition}
Note that in view of Lemma~\ref{lem:est_coeff} the coefficients of the expansion of $u \in U$, $u^{(0)} \in U_0$, or $f \in F$ at $x = 0$ is still defined, so that the representation of the norms in \eqref{norm_sol}, \eqref{norm_rhs}, and \eqref{norm_init} remain valid. Also note that a theorem in the sense of ``$H = W$'' holds, that is, we obtain the same space if we define our function spaces in the sense of Definition~\ref{def:function_spaces} or as the space of all locally integrable $u: (0,\infty)^2 \to \R$ for which \eqref{formal_exp_u} holds true locally almost everywhere to order $O(x^{2N+1})$ and the norm $\vertiii{u}_{\mathrm{sol}}$ is finite. This is also the case for the spaces $U_0$ and $F$. We refer to \cite[Lem.~B.3, Lem.~B.4]{ggko.2014} for details in an analogous situation.

\subsection{The resolvent equation\label{sec:lin_res}}
In this section, we start making some of the arguments in \S\ref{sec:maxreg}--\S\ref{sec:maxreg2} rigorous. At the core of our reasoning is a solid understanding of the resolvent problem to \eqref{lin_pde}. Suppose we are given initial data $u_{|t = 0} = u^{(0)}$ satisfying the boundary condition \eqref{bc_lin}. Using the method of lines, at time $\delta t > 0$ sufficiently small we obtain an approximate solution $u^{(\delta t)}$ to \eqref{linear_u} by solving the ordinary differential equation (ODE)
\begin{equation}\label{evo_u_discrete}
\frac{u^{(\delta t)} - u^{(0)}}{\delta t} + \AAA u^{(\delta t)} = f^{(\delta t)} \quad \mbox{for} \quad x > 0,
\end{equation}
where the right-hand side is averaged over the interval $(0,\delta t)$, i.e., $f^{(\delta t)} := \frac{1}{\delta t} \int_0^{\delta t} f(t) \, \d t$. Setting $g := f^{(\delta t)} + \frac{1}{\delta t} u^{(0)}$ and writing $u := u^{(\delta t)}$ and $\lambda := \frac{1}{\delta t}$, we arrive at the resolvent equation
\begin{equation}\label{resolvent_u}
\lambda u + \AAA u = g \quad \mbox{for} \quad x > 0,
\end{equation}
where $\AAA \stackrel{\eqref{defa}}{=} \partial_x \left(x^3 + x^2\right) \partial_x^3 = x^{-1} p(D) + x^{-2} q(D)$ with $D = x \partial_x$, and $p(\zeta)$ and $q(\zeta)$ are as in \eqref{def_pq}. Additionally observe that $g_{|x = 0} = 0$ by compatibility with \eqref{bc_lin}. We note that, since \eqref{resolvent_u} is not time-dependent anymore, we cannot find a rescaling as in the time-dependent setting \eqref{linear_u} to eliminate the small constant $\delta t$, or the large constant $\lambda$ in \eqref{resolvent_u}, respectively.

\medskip

Our aim is to construct solutions to the ODE \eqref{resolvent_u}, having suitable asymptotic properties as $x \searrow 0$ and $x \nearrow \infty$. The basic idea is to construct two-parameter solution families in a right-neighborhood of $x = 0$ and for $x \gg 1$ and to match these solution families using the coercivity of $\AAA$ in form of \eqref{coercivity_a}:
\begin{proposition}\label{prop:resolvent}
Suppose $\lambda > 0$ and $g \in C^\infty\left([0,\infty)\right)$ with $g_{|x = 0} = 0$ satisfies
\begin{equation}\label{decay_x_gg_1}
\limsup_{x \to \infty} \verti{e^{\nu 4 \sqrt[4]{\lambda x}} \frac{\d^j g}{\d x^j}(x)} < \infty \quad \mbox{for all} \quad j \in \N_0 \quad \mbox{and} \quad \nu \in \left[0,\frac{1}{\sqrt 2}\right).
\end{equation}
Then the resolvent equation \eqref{resolvent_u} has exactly one locally integrable solution $u : (0,\infty) \to \R$ with $\verti{u}_{2,\frac 1 2} < \infty$, $\verti{(D-1) u}_1 < \infty$, $\verti{(D-1)^2 u}_1 < \infty$, and $u_{|x = 0} = 0$. This solution obeys $u \in C^\infty\left([0,\infty)\right)$ and satisfies \eqref{decay_x_gg_1}.
\end{proposition}
Here, we give a summary of the arguments based on a similar reasoning contained in \cite[\S6]{ggko.2014}.
\begin{proof}[Proof of Proposition~\ref{prop:resolvent}]$\mbox{}$
\subsubsection*{Step 1: A solution family close to the contact line}
As $x \searrow 0$ and for $\lambda > 0$ fixed, the term $\sim x^{-2}$ on the left-hand side of \eqref{resolvent_u} is dominant. First, we may reformulate \eqref{resolvent_u} in form of a fixed-point problem
\begin{equation}\label{fixed_x_small}
u = \TT[u] := \SSS (g - \lambda u) + a_1 x + a_2 x^2 \quad \mbox{for} \quad x \ll_\lambda 1,
\end{equation}
where $a_1, a_2 \in \R$ are free parameters and $\SSS$ is the inverse of $\AAA$ such that $\frac{\d^j}{\d x^j} \SSS g = 0$ for $j = 0, 1, 2$. Note that an explicit representation of $\SSS$ in terms of (singular) integrals is straight-forward by using the factorization of $\AAA$ as in \eqref{factor_a}:
\[
\AAA = x^{-2} (D-1)^2 (D-2) \BB, \quad \mbox{where} \quad \BB:= x (D-2) + D = (x+1)^3 D (x+1)^{-2}.
\]
Further note that from \eqref{ak} and \eqref{def_pkqk}, i.e., $\frac{\d^k}{\d x^k} \AAA = \AAA_k \frac{\d^k}{\d x^k}$, where
\[
\AAA_k = x^{-1} p_k(D) + x^{-2} q_k(D)
\]
with
\begin{align*}
p_k(D) &= D (D+k) (D-(1-k)) (D-(2-k)), \\
q_k(D) &= D (D-1) (D-(1-k)) (D-(2-k)),
\end{align*}
we can derive
\begin{equation}\label{ak_bk}
\begin{aligned}
& \AAA_k = x^{-2} (D-1) (D+k-1) (D+k-2) \BB_k \\
& \mbox{with} \quad \BB_k := x (D+k-2) + D = (x+1)^{3-k} D (x+1)^{k-2},
\end{aligned}
\end{equation}
which generalizes \eqref{factor_a}. Since $D-\gamma = x^{1+\gamma} \frac{\d}{\d x} x^{-\gamma}$ and due to uniqueness of $\SSS$ if $\frac{\d \SSS g}{\d x}(0) = \frac{\d^2 \SSS g}{\d x^2} (0) = 0$, we have after four integrations from $x = 0$
\begin{equation}\label{rep_s}
\begin{aligned}
& \frac{\d^k \left(\SSS g\right)}{\d x^k}(x) \\
& \quad = (x+1)^{2-k} \int_0^x \int_0^{x_1} \int_0^{x_2} \int_0^{x_3} (x_1+1)^{k-3} x_1^{2-k} x_2^{-1} x_3^k x_4 \left(\frac{\d^k g}{\d x^k}\right)(x_4) \frac{\d x_4}{x_4} \frac{\d x_3}{x_3} \frac{\d x_2}{x_2} \frac{\d x_1}{x_1} \\
& \quad = x^2 (x+1)^{2-k} \iiiint_{[0,1]^4} (r_1 x + 1)^{k-3} r_1^2 r_2^k r_3^{k+1} r_4 \left(\frac{\d^k g}{\d x^k}\right)(r_1 r_2 r_3 r_4 x) \frac{\d r_4}{r_4} \frac{\d r_3}{r_3} \frac{\d r_2}{r_2} \frac{\d r_1}{r_1}.
\end{aligned}
\end{equation}
Note that $\SSS$ is defined through \eqref{rep_s} with $k = 0$, so that by uniqueness of $\SSS$ under the constraint $\frac{\d^j}{\d x^j} \SSS g_{|x = 0} = 0$ for $j = 0, 1, 2$, the general formula \eqref{rep_s} is valid for all $k \in \N_0$. Then, one may verify that a fixed-point argument can be carried out in a right-neighborhood of the origin $x = 0$ using $\vertii{\cdot}_\infty$-based norms:

\medskip

Take $\vertii{g} := \max_{k = 1,\ldots,K} \max_{0 \le x \le \eps} \verti{\frac{\d^k g}{\d x^k}(x)}$ with $K \in \N$ arbitrary. Then from \eqref{rep_s} for $0 \le x \le \eps \le 1$ and $k \ge 1$:
\begin{eqnarray*}
\verti{\frac{\d^k \SSS g}{\d x^k}(x)} &\le& x^2 \max\left\{(x+1)^{2-k},(x+1)^{-1}\right\} \iiiint_{[0,1]^4} r_1^2 r_2^k r_3^{k+1} r_4 \frac{\d r_4}{r_4} \frac{\d r_3}{r_3} \frac{\d r_2}{r_2} \frac{\d r_1}{r_1} \\
&& \times \sup_{0 \le x \le \eps} \verti{\frac{\d^k g}{\d x^k}(x)} \\
&\le& \frac{2 x^2}{k (k+1)} \verti{\frac{\d^k g}{\d x^k}(x)},
\end{eqnarray*}
whence
\begin{equation}\label{est_lin_fixed_0}
\vertii{\SSS g} \le \eps^2 \vertii{g}.
\end{equation}
We apply estimate~\eqref{est_lin_fixed_0} to \eqref{fixed_x_small} for fixed $a_1, a_2 \in \R$ and note that
\begin{itemize}
\item $\TT$ maps the space $\{u: \, u(0) = 0 \; \mbox{and} \; \vertii{u} < \infty\}$ into itself,
\item for $u^{(1)}$ and $u^{(2)}$ with $\vertii{u^{(1)}} < \infty$ and $\vertii{u^{(2)}} < \infty$ we have
\[
\vertii{\TT\left[u^{(1)}\right] - \TT\left[u^{(2)}\right]} \le \lambda \eps^2 \vertii{u^{(1)} - u^{(2)}}, 
\]
that is, $\TT$ is a self map for $\eps < \frac{1}{\sqrt\lambda} = \sqrt{\delta t}$.
\end{itemize}
Thus, the contraction-mapping theorem yields for every fixed $(a_1,a_2) \in \R^2$ a unique fixed-point $u$ in $\{u: \, \vertii{u} < \infty\}$ provided $\vertii{g} < \infty$. Due to the choice of the norm $\vertii{\cdot}$ and the fact that $K$ can be chosen arbitrarily large, we have $u \in C^\infty\left([0,1]\right)$ and inverting $\SSS$, we see that $u$ is a solution to the resolvent problem \eqref{resolvent_u}.

\medskip

Denote by $u^{(1)}$ and $u^{(2)}$ the solutions $u$ to \eqref{fixed_x_small} with $g \equiv 0$ and $(a_1,a_2) = (1,0)$ or $(a_1,a_2) = (0,1)$, respectively. We also write $u =: \TT_\ll g$ for the solution $u$ to \eqref{fixed_x_small} with $(a_1,a_2) = (0,0)$. These solutions can be extended to the whole interval $(0,\infty)$ by standard theory.

\subsubsection*{Step 2: A solution family as $x \to \infty$}
The above reasoning shows that we have a two-parameter family of solutions that are well-behaved as $x \searrow 0$. For constructing solutions with analogous features as $x \to \infty$, note that now the terms $\sim x^{-1}$ and $\sim x^{-2}$ are small compared to the addend $\lambda u$ in \eqref{resolvent_u}, so that we need to adapt our arguments.

\medskip

First take $r := 4 \sqrt[4]{x}$ with $r \ge 0$, so that equation~\eqref{resolvent_u} changes to
\begin{equation}\label{res_large_x}
\left(\lambda + \frac{\d^4}{\d r^4}\right) u + r^{-1} Q\left(r^{-1},\frac{\d}{\d r}\right) u = g \quad \mbox{for} \quad r > 0,
\end{equation}
where $Q\left(r^{-1},\frac{\d}{\d r}\right)$ is a fourth-order linear operator in $\partial_r$ with coefficients that are bounded for $r \ge 1$. The operator $\lambda + \frac{\d^4}{\d r^4}$ is simple and a fundamental solution $G$, fulfilling
\begin{equation}\label{fundamental}
\left(\lambda + \frac{\d^4}{\d r^4}\right) G = \delta_0 \quad \mbox{with} \quad \lim_{r \to \pm \infty} G(r) = 0,
\end{equation}
can be found using the fundamental system of $\lambda + \partial_r^4$, i.e., by looking for the fourth roots of $-\lambda$. Note that with $\breve r := \sqrt[4]{\lambda} r$ and $\breve G := \lambda G$, the dependence on $\lambda$ in \eqref{fundamental} disappears and we are in the same situation as in \cite[Eq.~(6.25)]{ggko.2014}, that is, we have
\begin{equation}\label{fundamental_sol}
G(r) =
\begin{cases}
\frac{1}{\lambda \sqrt 2} \sin \left(\frac{\sqrt[4]{\lambda} r}{\sqrt 2}\right) e^{\frac{\sqrt[4]{\lambda}  r}{\sqrt 2}} - \frac{1}{\lambda \sqrt 2} \cos \left(\frac{\sqrt[4]{\lambda} r}{\sqrt 2}\right) e^{\frac{\sqrt[4]{\lambda}  r}{\sqrt 2}} & \mbox{ for } \, r < 0, \\
- \frac{1}{\lambda \sqrt 2} \sin \left(\frac{\sqrt[4]{\lambda} r}{\sqrt 2}\right) e^{- \frac{\sqrt[4]{\lambda}  r}{\sqrt 2}} - \frac{1}{\lambda \sqrt 2} \cos \left(\frac{\sqrt[4]{\lambda} r}{\sqrt 2}\right) e^{- \frac{\sqrt[4]{\lambda}  r}{\sqrt 2}} & \mbox{ for } \, r < 0,
\end{cases}
\end{equation}
Cutting off with $\eta_R$, where $\eta_R(r) := \eta(r/R)$ and $\eta \in C^\infty(\R)$ fulfills $\eta_{|\left(-\infty,\frac 1 2\right]} \equiv 0$ and $\eta_{|\left[1,\infty\right)} \equiv 1$, we can then express $u$ as a fixed point
\begin{equation}\label{fixed_large_x}
u = G * \left(\eta_R g - r^{-1} Q\left(r,\frac{\d}{\d r}\right) u\right) \quad \mbox{for} \quad r \ge 2,
\end{equation}
which is the same problem as the one treated in \cite[\S6.3]{ggko.2014} giving a solution $\TT_\gg g := u \in C^\infty\left((R,\infty)\right)$, where $R \gg 1$, with decay \eqref{decay_x_gg_1} (coming from the decay of $g$, cf.~\eqref{fundamental_sol}, and the transformation $r = 4 \sqrt[4]{x}$) provided $g \in C^\infty\left((R,\infty)\right)$ fulfills \eqref{decay_x_gg_1} as well (cf.~\cite[Def.~6.2, Lem.~6.6]{ggko.2014} for details). Note, however, that this only gave us a particular solution to \eqref{resolvent_u}. The more subtle part is to find two linearly independent solutions to the homogeneous version of \eqref{resolvent_u}. Here, the problem is that setting $g \equiv 0$ in \eqref{fixed_large_x}, our fixed-point iteration will merely select the trivial solution $u \equiv 0$. This requires a more involved change of variables, also taking terms $\sim r^{-1}$ into account. This is done by substituting $u =: r^{\beta r} e^{\mu r} \tilde u$ and $g =: r^{\beta r} e^{\mu r} \tilde g$ in \eqref{res_large_x}, where $\mu = \frac{-1\pm i}{2}$ and $\beta \in \R$ is chosen suitably such that
\[
\left(\left(\mu+\frac{\d}{\d r}\right)^4 + 1\right) \tilde u + r^{-2} \tilde Q\left(r^{-1},\frac{\d}{\d r}\right) \tilde u = \tilde g \quad \mbox{for} \quad r > 0,
\]
where $\tilde Q\left(r^{-1},\frac{\d}{\d r}\right)$ is a fourth-order linear operator in $\frac{\d}{\d r}$ with coefficients that are bounded for $r \ge 1$. A solution can be found by a contraction principle through inverting the linear operator $\left(\mu+\frac{\d}{\d r}\right)^4 + 1$. We refer to \cite[\S6.3]{ggko.2014} for details in an analogous case. Take real and imaginary parts and denote these smooth and in the sense of \eqref{decay_x_gg_1} decaying real-valued solutions by $u^{(3)}$ and $u^{(4)}$. Again, we can extend $u^{(3)}$ and $u^{(4)}$ to the whole interval $(0,\infty)$ by standard theory.

\subsubsection*{Step 3: Global solutions to the resolvent problem}
In summary, we have obtained a two-dimensional solution manifold with good regularity properties as $x \searrow 0$ and a two-dimensional solution manifold with good decay properties as $x \to \infty$. Our aim is to construct a solution to \eqref{resolvent_u} with good (decay/regularity) properties as $x \searrow 0$ and $x \nearrow \infty$. This can be achieved by employing a uniqueness result for the resolvent equation \eqref{resolvent_u}. The latter can be obtained using the (partial) coercivity of the operator $\AAA$ given through \eqref{coercivity_a}: Test the homogeneous equation \eqref{resolvent_u} with $u$ in the inner product $\left(\cdot,\cdot\right)_0$ assuming that $\verti{u}_{2,\frac 1 2}$, $\verti{(D-1) u}_1$, and $\verti{(D-1)^2 u}_1$ are finite. This gives
\[
0 = \left(\lambda u + \AAA u, u\right)_0 \ge \lambda \verti{u}_0^2 + c \left(\verti{u}_{2,\frac 1 2}^2 + \verti{(D-1) u}_1^2 + \verti{(D-1)^2 u}_1^2\right)
\]
for some $c > 0$. Then we obtain in particular $\verti{u}_{2,\frac 1 2} = 0$ and therefore $u \equiv 0$.

\medskip

Let us first demonstrate that the above sketched uniqueness result for \eqref{resolvent_u} in particular implies that the functions $u^{(1)}, u^{(2)}, u^{(3)}, u^{(4)}$ are linearly independent and thus form a fundamental system of solutions to \eqref{resolvent_u}: Suppose that there are $a_1, a_2, a_3, a_4 \in \R$ with
\[
a_1 u^{(1)} + \ldots + a_4 u^{(4)} = 0 \quad \mbox{for all} \quad x > 0.
\]
Then in particular $u := a_1 u^{(1)} + a_2 u^{(2)} = - a_3 u^{(3)} - a_4 u^{(4)}$ is a solution to the homogeneous resolvent equation \eqref{resolvent_u} that (including derivatives) decays super-algebraically as $x \to \infty$ in the sense of \eqref{decay_x_gg_1}, and is smooth in $x = 0$ with $u_{|x = 0} = 0$. This implies in particular that the norms $\verti{u}_{2,\frac 1 2}$, $\verti{(D-1) u}_1$, and $\verti{(D-1)^2 u}_1$ are finite, and so by uniqueness necessarily $u \equiv 0$. Then linear independence of $(u^{(1)},u^{(2)})$ and $(u^{(3)},u^{(4)})$, respectively, also implies $a_1 = \ldots = a_4 = 0$.

\medskip

Now take a smooth $g$ with $g_{|x = 0}$ and sufficient decay as $x \to \infty$. Then we know that $\TT_\gg g - \TT_\ll g$ is a solution to the homogeneous resolvent equation \eqref{resolvent_u} and therefore can be written as a linear combination of $u^{(1)},\ldots,u^{(4)}$, i.e.,
\[
\TT_\gg g - \TT_\ll g = a_1 u^{(1)} + a_2 u^{(2)} - a_3 u^{(3)} - a_4 u^{(4)} \quad \mbox{for some} \quad a_1,\ldots,a_4 \in \R.
\]
Define $u := \TT_\ll g + a_1 u^{(1)} + a_2 u^{(2)} = \TT_\gg g + a_3 u^{(3)} + a_4 u^{(4)}$. From the first equality we learn that $u$ is smooth in $x = 0$ with $u_{|x = 0} = 0$ and from the second equality we know that $u$ decays in the sense of \eqref{decay_x_gg_1}. This implies finiteness of $\verti{u}_{2,\frac 1 2}$, $\verti{(D-1) u}_1$, and $\verti{(D-1)^2 u}_1$.

\medskip

Hence, we have proved existence and uniqueness of solutions to \eqref{resolvent_u} for $g$ smooth with $g_{|x = 0} = 0$ and super-algebraic decay as $x \to \infty$ in the sense of \eqref{decay_x_gg_1} under the assumption $\verti{u}_{2,\frac 1 2} < \infty$, $\verti{(D-1) u}_1 < \infty$, and $\verti{(D-1)^2 u}_1 < \infty$. This solution fulfills $u \in C^\infty\left([0,\infty)\right)$, $u_{|x = 0} = 0$, and meets the decay estimates \eqref{decay_x_gg_1}.
\end{proof}

\subsection{Rigorous treatment of the linear equation\label{sec:lin_rig}}
\subsubsection{Statement of results}
The goal of this section is to prove Propositions~\ref{prop:maxreg1}, from which Proposition~\ref{prop:mr_full} follows immediately, since all arguments in \S\ref{sec:elliptic} and \S\ref{sec:maxreg2} have been rigorous. To this end, consider again the time-discrete problem \eqref{evo_u_discrete} for which we can prove the analogue of the differential estimate \eqref{mr_ta_0}:

\begin{lemma}\label{lem:discrete_step}
Suppose that $k \ge 2$, $\delta t > 0$, and $f, u^{(0)} \in C^\infty\left([0,\infty)\right)$, with $f_{|x = 0} = 0$ and $u^{(0)}_{|x = 0} = 0$, satisfy the decay properties \eqref{decay_x_gg_1}. Then the solution $u \in C^\infty\left([0,\infty)\right)$, with $u_{|x = 0}$ meeting \eqref{decay_x_gg_1}, satisfies
\begin{subequations}\label{discrete_step}
\begin{align}
\begin{split}
&\frac{1}{\delta t} \left( \verti{D^k \widetilde{u^{(\delta t)}}}_{\tilde\alpha}^2 + \tilde C \verti{\widetilde{u^{(\delta t)}}}_{\tilde\alpha}^2 - \verti{D^k \widetilde{u^{(0)}}}_{\tilde\alpha}^2 - \tilde C \verti{\widetilde{u^{(0)}}}_{\tilde\alpha}^2\right) + \verti{\widetilde{u^{(\delta t)}}}_{k+2,\tilde\alpha+\frac 1 2}^2 + \verti{\widetilde{u^{(\delta t)}}}_{k+2,\tilde\alpha+1}^2 \\
&\quad \lesssim_{k,\tilde\alpha} \verti{\widetilde{x \underline f^{(\delta t)}}}_{k-2,\tilde\alpha-\frac 1 2}^2 + \verti{\widetilde{\underline f^{(\delta t)}}}_{k-2,\tilde\alpha-1}^2
\end{split}
\label{discrete_step_t}\\
\begin{split}
& \frac{1}{\delta t} \left(\verti{D^k \widecheck{u^{(\delta t)}}}_{\check\alpha}^2 + \check C \verti{\widecheck{u^{(\delta t)}}}_{\check\alpha}^2 - \verti{D^k \widecheck{u^{(0)}}}_{\check\alpha}^2 - \check C \verti{\widecheck{u^{(0)}}}_{\check\alpha}^2\right) + \verti{\widecheck{u^{(\delta t)}}}_{k+2,\check\alpha+\frac 1 2}^2 + \verti{\widecheck{u^{(\delta t)}}}_{k+2,\check\alpha+1}^2 \\
&\quad \lesssim_{k,\check\alpha} \verti{\widecheck{x \underline f^{(\delta t)}}}_{k-2,\check\alpha-\frac 1 2}^2 + \verti{\widecheck{\underline f^{(\delta t)}}}_{k-2,\check\alpha-1}^2
\end{split}
\label{discrete_step_c}
\end{align}
\end{subequations}
where $\tilde\alpha$ and $\check\alpha$ are in the coercivity ranges of $\tilde\AAA$ and $\check\AAA$, respectively (cf.~Lemma~\ref{lem:coerc_ta} and Lemma~\ref{lem:coerc_ca}), and $\tilde C = \tilde C\left(k,\tilde\alpha\right) > 0$ and $\check C = \check C\left(k,\check\alpha\right) > 0$ ($j = 1,2$) only depend on $k$ and $\tilde \alpha$ or $\check \alpha$, respectively.
\end{lemma}

The following interpolation estimate is valid, which is crucial for being able to give the initial value at $\{t = 0\}$ (the trace in time) a precise notion:
\begin{lemma}\label{lem:trace}
For $T \in (0,\infty]$ and any locally integrable $w, w^{(i)}, w^{(ii)}: (0,T) \times (0,\infty) \to \R$ with $w = w^{(i)} + w^{(ii)}$ we have
\begin{eqnarray}\nonumber
\sup_{t \ge 0} \verti{w}_{k,\gamma}^2 &\lesssim_{k,\gamma}& \verti{w_{|t = 0}}_{k,\gamma}^2 + \int_0^T \left(\verti{\partial_t w^{(i)}}_{k-2,\gamma-\frac 1 2}^2 +  \verti{\partial_t w^{(ii)}}_{k-2,\gamma-1}^2\right) \d t \\
&& + \int_0^T \left(\verti{w}_{k+2,\gamma+\frac 1 2}^2 +  \verti{w}_{k+2,\gamma+1}^2\right) \d t, \label{trace_est}
\end{eqnarray}
where $\gamma \in \R$ and $k \ge 2$.
\end{lemma}

Lemma~\ref{lem:discrete_step} can be used to construct solutions to the linear problem \eqref{linear_u} and to rigorously derive the parabolic maximal-regularity estimates \eqref{maxreg_ta} and \eqref{maxreg_ca}:
\begin{proposition}\label{prop:linear_existence}
Suppose $k \ge 2$ and $\tilde \alpha$ and $\check\alpha$ are in the coercivity ranges of the operators $\tilde\AAA$ and $\check\AAA$ ($\tilde \alpha \in (0,1)$ and $\check\alpha \in \left(1-\sqrt{\frac 5 6}, \frac 3 2\right)$ are sufficient, cf.~\eqref{coer_ta} and \eqref{coer_ca} of Lemma~\ref{lem:coerc_ta} and Lemma~\ref{lem:coerc_ca}). Assume $T \in (0,\infty]$ and suppose that $u^{(0)} : \, (0,\infty) \to \R$ and $f: \, (0,T) \times (0,\infty) \to \R$ are locally integrable with
\[
\verti{\widetilde{u^{(0)}}}_{k,\tilde\alpha} < \infty, \quad \verti{\widetilde{u^{(0)}}}_{k,\check\alpha} < \infty
\]
as well as
\[
 \int_0^T \left(\verti{\widetilde{x \underline f}}_{k-2,\tilde \alpha - \frac 1 2}^2 +  \verti{\widetilde{\underline f}}_{k-2,\tilde \alpha -1}^2\right) \d t < \infty, \quad  \int_0^T \left(\verti{\widecheck{x \underline f}}_{k-2,\check\alpha - \frac 1 2}^2 +  \verti{\widecheck{\underline f}}_{k-2,\check \alpha -1}^2\right) \d t < \infty.
\]
Then there exists a locally integrable solution $u : \, (0,T) \times (0,\infty) \to \R$ to \eqref{linear_u} fulfilling the parabolic maximal-regularity estimates
\begin{subequations}\label{maxreg_taca_t}
\begin{equation}\label{maxreg_ta_t}
\begin{aligned}
&\sup_{t \in [0,T)} \verti{\tilde u}_{k,\tilde \alpha}^2 + \int_0^T \left(\verti{\partial_t \widetilde{x \underline u}}_{k-2,\tilde\alpha-\frac 1 2}^2 +  \verti{\partial_t \widetilde{\underline u}}_{k-2,\tilde\alpha-1}^2 + \verti{\tilde u}_{k+2,\tilde \alpha + \frac 1 2}^2 +  \verti{\tilde u}_{k+2,\tilde \alpha + 1}^2\right) \d t  \\
& \quad\lesssim_{k,\tilde\alpha} \verti{\tilde u_{|t = 0}}_{k, \tilde \alpha}^2 + \int_0^T \left(\verti{\widetilde{x \underline f}}_{k-2,\tilde \alpha - \frac 1 2}^2 +  \verti{\widetilde{\underline f}}_{k-2,\tilde \alpha -1}^2\right) \d t
\end{aligned}
\end{equation}
and
\begin{equation}\label{maxreg_ca_t}
\begin{aligned}
&\sup_{t \in [0,T)} \verti{\check u}_{k,\check \alpha}^2 + \int_0^T \left(\verti{\partial_t \widecheck{x \underline u}}_{k-2,\check\alpha-\frac 1 2}^2 +  \verti{\partial_t \widecheck{\underline u}}_{k-2,\check\alpha-1}^2 + \verti{\check u}_{k+2,\check \alpha + \frac 1 2}^2 +  \verti{\check u}_{k+2,\check \alpha + 1}^2\right) \d t \\
&\quad \lesssim_{k,\check\alpha} \verti{\check u_{|t = 0}}_{k, \check \alpha}^2 + \int_0^T \left(\verti{\widecheck{x \underline f}}_{k-2,\check \alpha - \frac 1 2}^2 +  \verti{\widecheck{\underline f}}_{k-2,\check \alpha -1}^2\right) \d t.
\end{aligned}
\end{equation}
\end{subequations}
\end{proposition}

A uniqueness result holds under weaker conditions
\begin{proposition}\label{prop:linear_uniqueness}
For $T \in (0,\infty]$ suppose that $u: \, (0,T) \times (0,\infty) \to \R$ is locally integrable and solves \eqref{linear_u} distributionally on the time interval $(0,T)$ with homogeneous initial data $u^{(0)} \equiv 0$ and right-hand side $f \equiv 0$ such that one of the following conditions holds true
\begin{enumerate}[(a)]
\item\label{num:case_t} We have $\int_0^T \left(\verti{\tilde u}_{4,\tilde\alpha+\frac 1 2}^2 + \verti{\tilde u}_{4,\tilde\alpha+1}^2\right) < \infty$, where $\tilde \alpha$ is in the coercivity range of $\tilde \AAA$ ($\tilde \alpha \in (0,1)$ is sufficient, cf.~\eqref{coer_ta} of Lemma~\ref{lem:coerc_ta}).
\item\label{num:case_c} We have $\int_0^T \left(\verti{\check u}_{4,\check\alpha+\frac 1 2}^2 + \verti{\check u}_{4,\check\alpha+1}^2\right) < \infty$, where $\check \alpha$ is in the coercivity range of $\check \AAA$ ($\check \alpha \in \left(1-\sqrt{\frac 5 6}, \frac 3 2\right)$ is sufficient, cf.~\eqref{coer_ca} of Lemma~\ref{lem:coerc_ca}). 
\end{enumerate}
Then $u \equiv 0$. In particular the solution $u$ constructed in Proposition~\ref{prop:linear_existence} is unique.
\end{proposition}
As a corollary of Propositions~\ref{prop:linear_existence} and \ref{prop:linear_uniqueness} we can infer that Proposition~\ref{prop:maxreg1} and Propostion~\ref{prop:mr_full} have to hold true as well.

\subsubsection{Proofs}

\begin{proof}[Proof of Lemma~\ref{lem:discrete_step}]
Since the proof in essence only uses coercivity of $\tilde\AAA$ or $\check\AAA$ and the fact that the operators are fourth order, we restrict ourselves to proving estimate~\eqref{discrete_step_t}.

\medskip

We first apply $D-1$ to equation~\eqref{evo_u_discrete} and get (cf.~\eqref{def_tc})
\begin{equation}\label{res_t}
\frac{1}{\delta t} \left(\widetilde{u^{(\delta t)}} - \widetilde{u^{(0)}}\right) + \tilde\AAA \widetilde{u^{(\delta t)}} = \widetilde{f^{(\delta t)}}
\end{equation}
Next, we test \eqref{res_t} with $\widetilde{u^{(\delta t)}}$ in the inner product $\left(\cdot,\cdot\right)_{\tilde\alpha}$, so that
\begin{equation}\label{test_discrete_1}
\frac{1}{\delta t} \left(\verti{\widetilde{u^{(\delta t)}}}_{\tilde\alpha}^2 - \left(\widetilde{u^{(0)}}, \widetilde{u^{(\delta t)}}\right)_{\tilde\alpha}\right) + \left(\tilde\AAA \widetilde{u^{(\delta t)}}, \widetilde{u^{(\delta t)}}\right)_{\tilde\alpha} = \left(\widetilde{f^{(\delta t)}}, \widetilde{u^{(\delta t)}}\right)_{\tilde\alpha}.
\end{equation}
Using the elementary estimate
\[
\left(\widetilde{u^{(0)}}, \widetilde{u^{(\delta t)}}\right)_{\tilde\alpha} \le \frac 1 2 \verti{\widetilde{u^{(0)}}}_{\tilde\alpha}^2 + \frac 1 2 \verti{\widetilde{u^{(\delta t)}}}_{\tilde\alpha}^2,
\]
coercivity of $\tilde\AAA$ (cf.~Lemma~\ref{lem:coerc_ta}) and Young's inequality for the right-hand side as for the arguments leading from \eqref{test_ta} to \eqref{basic_ta}, we obtain
\begin{equation}\label{discrete_basic}
\frac{1}{\delta t} \left(\verti{\widetilde{u^{(\delta t)}}}_{\tilde\alpha}^2 - \verti{\widetilde{u^{(0)}}}_{\tilde\alpha}^2\right) + \left(\verti{\widetilde{u^{(\delta t)}}}_{2,\tilde\alpha+\frac 1 2}^2 + \verti{\widetilde{u^{(\delta t)}}}_{2,\tilde\alpha+1}^2\right) \lesssim_{\tilde\alpha} \left(\verti{\widetilde{x \underline f^{(\delta t)}}}_{\tilde\alpha-\frac 1 2}^2 + \verti{\widetilde{\underline f^{(\delta t)}}}_{\tilde\alpha-1}^2\right).
\end{equation}
Estimate~\eqref{discrete_basic} is the discrete analogue of the differential estimate \eqref{basic_ta}. For upgrading it to a strong estimate, apply $D^k$ to \eqref{res_t} and test the resulting equation with $D^k \widetilde{u^{(\delta t)}}$ in the inner product $\left(\cdot,\cdot\right)_{\tilde\alpha}$. This gives
\begin{equation}\label{discrete_test_higher}
\frac{1}{\delta t} \left(\verti{D^k \widetilde{u^{(\delta t)}}}_{\tilde\alpha}^2 - \left(D^k \widetilde{u^{(0)}}, D^k \widetilde{u^{(\delta t)}}\right)_{\tilde\alpha}\right) + \left(D^k \widetilde{u^{(\delta t)}}, D^k \tilde\AAA \widetilde{u^{(\delta t)}}\right)_{\tilde\alpha} = \left(D^k \widetilde{f^{(\delta t)}}, D^k \widetilde{u^{(\delta t)}}\right)_{\tilde\alpha}.
\end{equation}
Again we have
\[
\left(D^k \widetilde{u^{(\delta t)}}, D^k \widetilde{u^{(0)}}\right)_{\tilde\alpha} \le \frac 1 2 \verti{D^k \widetilde{u^{(\delta t)}}}_{\tilde\alpha}^2 + \frac 1 2 \verti{D^k \widetilde{u^{(0)}}}_{\tilde\alpha}^2
\]
and we may treat the other terms in \eqref{discrete_test_higher} as done in the context of going from \eqref{test_lin2} to \eqref{higher_ta}, so that
\begin{equation}\label{discrete_strong_-}
\begin{aligned}
&\frac{1}{\delta t} \left(\verti{D^k \widetilde{u^{(\delta t)}}}_{\tilde\alpha}^2 - \verti{D^k \widetilde{u^{(0)}}}_{\tilde\alpha}^2\right) + \verti{\widetilde{u^{(\delta t)}}}_{k+2,\tilde\alpha+\frac 1 2}^2 - \tilde C \verti{\widetilde{u^{(\delta t)}}}_{\tilde\alpha+\frac 1 2}^2 + \verti{\widetilde{u^{(\delta t)}}}_{k+2,\tilde\alpha+1}^2 - \tilde C \verti{\widetilde{u^{(\delta t)}}}_{\tilde\alpha+1}^2 \\
&\quad \lesssim_{k,\tilde\alpha} \verti{\widetilde{x \underline f^{(\delta t)}}}_{k-2,\tilde\alpha-\frac 1 2}^2 + \verti{\widetilde{\underline f^{(\delta t)}}}_{k-2,\tilde\alpha-1}^2.
\end{aligned}
\end{equation}
Adding a multiple of \eqref{discrete_basic} to \eqref{discrete_strong_-}, we arrive at \eqref{discrete_step_t}.
\end{proof}

\begin{proof}[Proof of Lemma~\ref{lem:trace}]
By approximation, we may assume without loss of generality that $w$ is smooth and continuous up to the boundary. Then we can compute, using a standard interpolation estimate (cf.~\cite[Lem.~B.1]{ggko.2014}),
\begin{eqnarray*}
\frac{\d}{\d t} \verti{\tilde w}_{k,\gamma}^2 &=& 2 \left(\partial_t w^{(i)},w\right)_{k,\gamma} + 2 \left(\partial_t w^{(ii)},w\right)_{k,\gamma} \\
&\lesssim_{k,\gamma}& \verti{\partial_t w^{(i)}}_{k-2,\gamma-\frac 1 2}^2 + \verti{w}_{k+2,\gamma+\frac 1 2}^2 +  \verti{\partial_t w^{(ii)}}_{k-2,\gamma-1}^2 +  \verti{w}_{k+2,\gamma+1}^2.
\end{eqnarray*}
Integrating this inequality, we immediately obtain \eqref{trace_est}.
\end{proof}

\begin{proof}[Proof of Proposition~\ref{prop:linear_existence}]
The arguments for proving \eqref{maxreg_ca_t} are analogous to those for proving \eqref{maxreg_ta_t}, so we concentrate on proving the latter. For simplicity we may further assume $T < \infty$. Take $J \in \N$, $\delta t := \frac{T}{J}$, and take piece-wise time averages of the right-hand side $f$ according to $f^{(j \delta t)} := \frac{1}{\delta t} \int_{(j-1) \delta t}^{j \delta t} f\left(t^\prime\right) \d t^\prime$ where $j \in \N$. Note that by approximation (cf.~Definition~\ref{def:function_spaces}) we can assume $f$ to be smooth and rapidly decaying in the sense of \eqref{decay_x_gg_1}. We discretize the linear problem \eqref{linear_u} as in \eqref{evo_u_discrete}, i.e.,
\begin{equation}\label{linear_u_j}
\frac{u^{(j \delta t)} - u^{((j-1) \delta t)}}{\delta t} + \AAA u^{(j \delta t)} = f^{(j \delta t)} \quad \mbox{for} \quad j \ge 1.
\end{equation}
Proposition~\ref{prop:resolvent} yields a smooth solution $u^{(j\delta t)} : \, [0,\infty) \to \R$ with decay as in \eqref{decay_x_gg_1}. In particular, estimates~\eqref{discrete_step} of Lemma~\ref{lem:discrete_step} are satisfied, so we start out with \eqref{discrete_step_t} at time step $j \delta t$, i.e.,
\begin{eqnarray*}
\lefteqn{\verti{D^k \widetilde{u^{(j \delta t)}}}_{\tilde\alpha}^2 + \tilde C \verti{\widetilde{u^{(j \delta t)}}}_{\tilde\alpha}^2 - \verti{D^k \widetilde{u^{((j-1) \delta t)}}}_{\tilde\alpha}^2 - \tilde C \verti{\widetilde{u^{((j-1) \delta t)}}}_{\tilde\alpha}^2} \\
&& + \delta t \left(\verti{\widetilde{u^{(j \delta t)}}}_{k+2,\tilde\alpha+\frac 1 2}^2 + \verti{\widetilde{u^{(j \delta t)}}}_{k+2,\tilde\alpha+1}^2\right) \lesssim_{k,\tilde\alpha} \delta t \left(\verti{\widetilde{x \underline f^{(j \delta t)}}}_{k-2,\tilde\alpha-\frac 1 2}^2 + \verti{\widetilde{\underline f^{(j \delta t)}}}_{k-2,\tilde\alpha-1}^2\right).
\end{eqnarray*}
Summation over $j = 1,\ldots,J^\prime$, where $J^\prime \in \{1,\ldots,J\}$, gives
\begin{eqnarray}\nonumber
\lefteqn{\verti{D^k \widetilde{u^{(J^\prime \delta t)}}}_{\tilde\alpha}^2 + \tilde C \verti{\widetilde{u^{(J^\prime \delta t)}}}_{\tilde\alpha}^2  + \sum_{j = 1}^{J^\prime} \delta t \left(\verti{\widetilde{u^{(j \delta t)}}}_{k+2,\tilde\alpha+\frac 1 2}^2 + \verti{\widetilde{u^{(j \delta t)}}}_{k+2,\tilde\alpha+1}^2\right)} \\
&\lesssim_{k,\tilde\alpha}& \verti{D^k \widetilde{u^{(0)}}}_{\tilde\alpha}^2 + \tilde C \verti{\widetilde{u^{(0)}}}_{\tilde\alpha}^2 + \sum_{j = 1}^{J^\prime} \delta t \left(\verti{\widetilde{x \underline f^{(j \delta t)}}}_{k-2,\tilde\alpha-\frac 1 2}^2 + \verti{\widetilde{\underline f^{(j \delta t)}}}_{k-2,\tilde\alpha-1}^2\right). \qquad \label{discrete_summed}
\end{eqnarray}
By interpolation and noting that the second line of \eqref{discrete_summed} is increasing in $J^\prime$, we have
\begin{eqnarray}\nonumber
\lefteqn{\max_{j = 1,\ldots,J}\verti{\widetilde{u^{(j \delta t)}}}_{k,\tilde\alpha}^2 + \sum_{j = 1}^J \delta t \left(\verti{\widetilde{u^{(j \delta t)}}}_{k+2,\tilde\alpha+\frac 1 2}^2 + \verti{\widetilde{u^{(j \delta t)}}}_{k+2,\tilde\alpha+1}^2\right)} \\
&\lesssim_{k,\tilde\alpha}& \verti{\widetilde{u^{(0)}}}_{k,\tilde\alpha}^2 + \sum_{j = 1}^J \delta t \left(\verti{\widetilde{f^{(j \delta t,1)}}}_{k-2,\tilde\alpha-\frac 1 2}^2 + \verti{\widetilde{f^{(j \delta t,2)}}}_{k-2,\tilde\alpha-1}^2\right). \label{discrete_mr}
\end{eqnarray}
Now, we define a piece-wise in time constant right-hand side $\phi_J$ and interpolate linearly in time for our approximate solution $\psi_J$, that is,
\begin{subequations}\label{cont_approx}
\begin{align}
\phi_J(t,x) &:= \sum_{j = 1}^J f^{(j \delta t)}(x) \ind_{[(j-1) \delta t, j \delta t)}(t) = \sum_{j = 1}^J \frac{1}{\delta t} \int_{(j-1) \delta t}^{j \delta t} f\left(t^\prime,x\right) \d t^\prime \ind_{[(j-1) \delta t, j \delta t)}(t), \label{cont_approx_phi} \\
\psi_J(t,x) &:= \sum_{j = 1}^J \left(\frac{t - (j-1) \delta t}{\delta t} u^{(j \delta t)}(x) + \frac{j \delta t - t}{\delta t} u^{((j-1) \delta t)}\right) \ind_{[(j-1) \delta t, j \delta t)}(t). \label{cont_approx_psi}
\end{align}
\end{subequations}
Then we note that
\begin{subequations}\label{simplify_discrete_sum}
\begin{equation}
\sup_{t \in [0,T)} \verti{\psi_J}_{k,\tilde\alpha}^2 \le 2 \max_{j = 1,\ldots,J}\verti{\widetilde{u^{(j \delta t)}}}_{k,\tilde\alpha}^2
\end{equation}
and
\begin{equation}
\int_0^T \left(\verti{\widetilde{\psi_J}}_{k+2,\tilde\alpha+\frac 1 2}^2 + \verti{\widetilde{\psi_J}}_{k+2,\tilde\alpha+1}^2\right) \d t \le 2 \sum_{j = 1}^J \delta t \left(\verti{\widetilde{u^{(j \delta t)}}}_{k+2,\tilde\alpha+\frac 1 2}^2 + \verti{\widetilde{u^{(j \delta t)}}}_{k+2,\tilde\alpha+1}^2\right)
\end{equation}
as well as
\begin{equation}
\int_0^T \left(\verti{\widetilde{\phi_J}}_{k-2,\tilde\alpha-\frac 1 2}^2 + \verti{\widetilde{\phi_J}}_{k-2,\tilde\alpha-1}^2\right) \d t = \sum_{j = 1}^J \delta t \left(\verti{\widetilde{x \underline f^{(j \delta t)}}}_{k-2,\tilde\alpha-\frac 1 2}^2 + \verti{\widetilde{\underline f^{(j \delta t)}}}_{k-2,\tilde\alpha-1}^2\right).
\end{equation}
\end{subequations}
Putting \eqref{simplify_discrete_sum} in \eqref{discrete_mr}, we get
\begin{eqnarray}\nonumber
\lefteqn{\sup_{t \in [0,T)} \verti{\widetilde{\psi_J}}_{k,\tilde\alpha}^2 + \int_0^T \left(\verti{\widetilde{\psi_J}}_{k+2,\tilde\alpha+\frac 1 2}^2 + \verti{\widetilde{\psi_J}}_{k+2,\tilde\alpha+1}^2\right) \d t} \\
&\lesssim_{k,\tilde\alpha}& \verti{\widetilde{u^{(0)}}}_{k,\tilde\alpha}^2 + \int_0^T \left(\verti{\widetilde{x \underline \phi_J}}_{k-2,\tilde\alpha-\frac 1 2}^2 + \verti{\widetilde{\underline \phi_J}}_{k-2,\tilde\alpha-1}^2\right) \d t. \label{discrete_mr_2}
\end{eqnarray}
For getting control of the time derivative $\partial_t \psi_J$ (defined almost everywhere in $[0,T)$ by virtue of \eqref{cont_approx_psi}), observe that for $j \in \{1,\ldots,J\}$ and $t \in \left( (j-1) \delta t, j \delta t\right)$ we have
\[
\partial_t \psi_J = \frac{u^{(j \delta t)} - u^{((j-1) \delta t)}}{\delta t} \stackrel{\eqref{linear_u_j}}{=} - \AAA u^{(j \delta t)} + f^{(j \delta t)},
\]
so that with \eqref{cont_approx} we obtain
\begin{equation}\label{discrete_eq}
\partial_t \psi_J + \AAA \psi_J = \phi_J + \AAA \left(\psi_J - u^{(j\delta t)}\right) \quad \mbox{where} \quad j \in \{1,\ldots,J\}, \quad t \in \left((j-1) \delta t, j \delta t\right).
\end{equation}
In exactly the same manner as in the time-continuous case (cf.~\eqref{time_tu1}, \eqref{time_tu2}), we obtain control on $\verti{\partial_t \widetilde{x \underline \psi_J}}_{k-2,\tilde\alpha-\frac 1 2}$ and $\verti{\partial_t \widetilde{\underline \psi_J}}_{k-2,\tilde\alpha-1}$, that is, \eqref{discrete_mr_2} upgrades to
\begin{eqnarray*}
\lefteqn{\sup_{t \in [0,T)} \verti{\widetilde{\psi_J}}_{k,\tilde\alpha}^2 + \int_0^T \left(\verti{\partial_t \widetilde{x \underline\psi_J}}_{k-2,\tilde\alpha-\frac 1 2}^2 +  \verti{\partial_t \widetilde{\underline\psi_J}}_{k-2,\tilde\alpha-1}^2\right) \d t} \\
&& + \int_0^T \left(\verti{\widetilde{\psi_J}}_{k+2,\tilde\alpha+\frac 1 2}^2 + \verti{\widetilde{\psi_J}}_{k+2,\tilde\alpha+1}^2\right) \d t \\
&\lesssim_{k,\tilde\alpha}& \verti{\widetilde{u^{(0)}}}_{k,\tilde\alpha}^2 + \int_0^T \left(\verti{\widetilde{x \underline\phi_J}}_{k-2,\tilde\alpha-\frac 1 2}^2 + \verti{\widetilde{\underline \phi_J}}_{k-2,\tilde\alpha-1}^2\right) \d t.
\end{eqnarray*}
Furthermore, observe that by Jensen's and H\"older's inequality
\begin{eqnarray*}
\lefteqn{\int_0^T \left(\verti{\widetilde{x \underline\phi_J}}_{k-2,\tilde\alpha-\frac 1 2}^2 + \verti{\widetilde{\underline\phi_J}}_{k-2,\tilde\alpha-1}^2\right) \d t} \\
&\stackrel{\eqref{cont_approx_phi}}{=}& \sum_{j = 1}^J \int_{(j-1) \delta t}^{j \delta t} \frac{1}{(\delta t)^2} \left(\verti{\int_{(j-1) \delta t}^{j \delta t} \widetilde{x \underline f}\left(t^\prime\right) \d t^\prime}_{k-2,\tilde\alpha - \frac 1 2}^2 + \verti{\int_{(j-1) \delta t}^{j \delta t} \widetilde{\underline f}\left(t^\prime\right) \d t^\prime}_{k-2,\tilde\alpha - 1}^2\right) \d t \\
&\le& \sum_{j = 1}^J \left(\frac{1}{\delta t} \left(\int_{(j-1) \delta t}^{j \delta t} \verti{\widetilde{x \underline f}\left(t\right)}_{k-2,\tilde\alpha - \frac 1 2} \d t\right)^2 + \frac{1}{\delta t} \left(\int_{(j-1) \delta t}^{j \delta t} \verti{\widetilde{\underline f}\left(t\right)}_{k-2,\tilde\alpha - 1} \d t\right)^2 \right) \\
&\le& \sum_{j = 1}^J \left(\int_{(j-1) \delta t}^{j \delta t} \verti{\widetilde{x \underline f}\left(t\right)}_{k-2,\tilde\alpha - \frac 1 2}^2 \d t + \int_{(j-1) \delta t}^{j \delta t} \verti{\widetilde{\underline f}\left(t\right)}_{k-2,\tilde\alpha - 1}^2 \d t\right) \\
&=& \int_0^T \left(\verti{\widetilde{x \underline f}\left(t\right)}_{k-2,\tilde\alpha - \frac 1 2}^2 + \verti{\widetilde{\underline f}\left(t\right)}_{k-2,\tilde\alpha - 1}^2\right) \d t
\end{eqnarray*}
and therefore
\begin{eqnarray}\nonumber
\lefteqn{\sup_{t \in [0,T)} \verti{\widetilde{\psi_J}}_{k,\tilde\alpha}^2 + \int_0^T \left(\verti{\partial_t \widetilde{x \underline \psi_J}}_{k-2,\tilde\alpha-\frac 1 2}^2 +  \verti{\partial_t \widetilde{\underline \psi_J}}_{k-2,\tilde\alpha-1}^2\right) \d t} \\
&& + \int_0^T \left(\verti{\widetilde{\psi_J}}_{k+2,\tilde\alpha+\frac 1 2}^2 + \verti{\widetilde{\psi_J}}_{k+2,\tilde\alpha+1}^2\right) \d t \nonumber \\
&\lesssim_{k,\tilde\alpha}& \verti{\widetilde{u^{(0)}}}_{k,\tilde\alpha}^2 + \int_0^T \left(\verti{\widetilde{x \underline f}\left(t\right)}_{k-2,\tilde\alpha - \frac 1 2}^2 + \verti{\widetilde{\underline f}\left(t\right)}_{k-2,\tilde\alpha - \frac 1 2}^2\right) \d t. \label{discrete_mr_3}
\end{eqnarray}
Note that the root of the first two lines of \eqref{discrete_mr_3} defines a norm for $\psi_J$ up to adding $a x$, where $a$ is constant in time, which is fixed by the initial datum $u^{(0)}$. Taking the limit $J \to \infty$ we infer that a subsequence of $\psi_J$ weak-$*$-converges to a locally integrable function $u : \, (0,T) \times (0,\infty) \to \R$. By weak lower semi-continuity, estimate~\eqref{discrete_mr_3} turns into \eqref{maxreg_ta_t} in the limit $J \to \infty$. Because of Lemma~\ref{lem:trace}, necessarily $u_{|t = 0} = u^{(0)}$ holds true. Furthermore, due to \eqref{discrete_eq}, equation~\eqref{lin_pde}, i.e., $\partial_t u + \AAA u = f$, is satisfied distributionally.
\end{proof}

\begin{proof}[Proof of Proposition~\ref{prop:linear_uniqueness}]
We concentrate on proving Proposition~\ref{prop:linear_uniqueness} under the assumption \eqref{num:case_t}. We take a test function $\eta : \R \to [0,1]$ with $\eta_{|[-1,1]} = 1$ and $\supp \eta \subset [-2,2]$, and define $\eta_n := \eta(s/n)$, where $s := \ln x$ and $n \in \N$. Next, we apply $\eta_n (D-1)$ to equation~\eqref{lin_pde} and test the resulting equation with $\eta_n \tilde u$ in the inner product $\left(\cdot,\cdot\right)_{\tilde\alpha}$ to the result
\begin{equation}\label{test_unique}
\left(\eta_n \partial_t \tilde u, \eta_n \tilde u\right)_{\tilde\alpha} + \left(\eta_n \tilde\AAA \tilde u, \eta_n \tilde u\right)_{\tilde\alpha} = 0.
\end{equation}
Now employing coercivity of $\tilde\AAA$ (cf.~Lemma~\ref{lem:coerc_ta}) is not directly possible, since we need to commute $\eta_n$ with the operator $\tilde\AAA$. Note that every term in the commutator $\left[\eta_n,\tilde\AAA\right]$ must contain at least one derivative $D=\partial_s$ acting on $\eta_n$, giving a factor $n^{-1}$. Hence, we can conclude that there exists a constant $\tilde c = \tilde c\left(\tilde\alpha\right)$ such that
\begin{equation}\label{unique_simplify_1}
\left(\tilde\AAA \tilde u, \tilde u\right)_{\tilde\alpha} \ge \tilde c \left(\verti{\tilde u}_{2,\tilde\alpha + \frac 1 2}^2 + \verti{\tilde u}_{2,\tilde\alpha+1}^2\right) - R_n,
\end{equation}
where $R_n = R_n(t)$ is a uniformly integrable remainder (one may recognize that through integration by parts it is up to a multiplicative constant dominated by $\verti{\tilde u}_{2,\tilde\alpha + \frac 1 2}^2 + \verti{\tilde u}_{2,\tilde\alpha+1}^2$) with $R_n(t) \to 0$ as $n \to \infty$ almost everywhere in $t \in (0,T)$. By arguments analogous to those in the context of \eqref{time_tu1} and \eqref{time_tu2}, we have
\[
\int_0^T \left(\verti{\partial_t\widetilde{x \underline u}}_{k-2,\tilde\alpha-\frac 1 2}^2 + \verti{\partial_t\widetilde{\underline u}}_{\tilde\alpha-1}^2\right) \d t \lesssim \int_0^\infty \left(\verti{\tilde u}_{4,\tilde\alpha+\frac 1 2}^2 + \verti{\tilde u}_{4,\tilde\alpha+\frac 1 2}^2\right) \d t < \infty,
\]
that is, we have $\tilde u \in W^{1,2}\left((0,T); L^2_\mathrm{loc}((0,\infty))\right)$. In particular, $u_{|t = 0}$ is well-defined and we arrive at
\begin{equation}\label{unique_simplify_2}
\left(\eta_n \partial_t \tilde u, \eta_n \partial_t \tilde u\right)_{\tilde\alpha} = \frac 1 2 \frac{\d}{\d t} \verti{\eta_n \tilde u}_{\tilde\alpha}^2.
\end{equation}
Putting \eqref{unique_simplify_1} and \eqref{unique_simplify_2} in \eqref{test_unique} gives after integration in time
\begin{equation}\label{unique_weak_est}
\sup_{t \in (0,T)} \verti{\eta_n \tilde u}_{\tilde\alpha}^2 + \tilde c \int_0^T \left(\verti{\eta_n \tilde u}_{2,\tilde\alpha-\frac 1 2}^2 + \verti{\eta_n \tilde u}_{2,\tilde\alpha-1}^2\right) \d t \le \int_0^T R_n \, \d t.
\end{equation}
Taking $n \to \infty$ in \eqref{unique_weak_est}, we note that by dominated convergence the right-hand side vanishes and therefore $\sup_{t \in (0,T)} \verti{\tilde u}_{\tilde\alpha} = 0$, i.e., $(D-1) u = \tilde u = 0$. Hence, $u \in \ker\{D-1\}$, which implies $u(t,x) = u_1(t) x$. Using this in \eqref{lin_pde} gives
\[
0 = \partial_t u_1 x + \AAA u_1 x \stackrel{\eqref{defa}}{=} \frac{\d u_1}{\d t} x \quad \mbox{for} \quad (t,x) \in (0,T) \times (0,\infty)
\]
whence
\[
\frac{\d u_1}{\d t} = 0 \quad \mbox{for} \quad t \in (0,T),
\]
which together with the initial condition $\left.u_1\right|_{t = 0} = 0$ amounts to $u_1 \equiv 0$.
\end{proof}

\section{Nonlinear theory\label{sec:nonlin}}
\subsection{Main results\label{sec:nonlin_main}}
Our aim is to prove that the nonlinearity $\NN(u)$ (defined in \eqref{def_nu} and with a structure detailed in \eqref{term_cond_non}) fulfills the following nonlinear estimate:
\begin{proposition}\label{prop:non_est}
Suppose $N \ge 1$, $k \ge 3$, and $0 < \delta \ll_{k, N} 1$. Then for $\vertiii{\cdot}_{\mathrm{sol}}$ and $\vertiii{\cdot}_{\mathrm{rhs}}$ defined as in \eqref{norm_sol} and \eqref{norm_rhs}, we have
\begin{equation}\label{estimate_nonlinear_main}
\vertiii{\NN(u) - \NN\left(\tilde u\right)}_{\mathrm{rhs}} \lesssim_{k, N, \delta} \left(\vertiii{u}_{\mathrm{sol}} + \vertiii{\tilde u}_{\mathrm{sol}}\right) \vertiii{u - \tilde u}_{\mathrm{sol}}
\end{equation}
for all smooth $u, \tilde u \in U$ with $\vertiii{u}_{\mathrm{sol}} \ll_{k,N,\delta} 1$, $\vertiii{\tilde u}_{\mathrm{sol}} \ll_{k,N,\delta} 1$.
\end{proposition}
The combination of Proposition~\ref{prop:mr_full}, Proposition~\ref{prop:non_est}, and a standard fixed-point argument yields existence, uniqueness, and stability for the nonlinear problem \eqref{nonlinear_u}, i.e.,
\begin{align*}
\partial_t u + \AAA u &= \NN(u) \quad \mbox{for} \quad t,x > 0, \\
u &= 0 \quad \mbox{for} \quad t > 0, x = 0,
\end{align*}
subject to the initial condition $u_{|t = 0} = u^{(0)}$, where $\AAA$ and $\NN(u)$ are given through \eqref{defa}, \eqref{def_nu}, and \eqref{term_cond_non}. The main theorem reads as follows:
\begin{theorem}\label{th:main}
Suppose that $N \ge 1$, $k \ge 3$, and $0 < \delta < \frac 1 2$. Then there exists $\eps > 0$ such that for initial data $u^{(0)} \in U_0$ with $\vertiii{u^{(0)}}_{\mathrm{init}} \le \eps$, the nonlinear problem \eqref{nonlinear_u} with initial condition $u_{|t = 0} = u^{(0)}$ has exactly one locally integrable solution $u \in U$. This solution fulfills the a-priori estimate $\vertiii{u}_{\mathrm{sol}} \lesssim_{k, N, \delta} \vertiii{u^{(0)}}_{\mathrm{init}}$. Furthermore, we have $\vertiii{u(t)}_{\mathrm{init}} \to 0$ as $t \to \infty$, that is, the traveling wave \eqref{tw} is asymptotically stable.
\end{theorem}
Note that by estimates~\eqref{coeff_est_u_l2} and \eqref{coeff_est_u_sup} of Lemma~\ref{lem:est_coeff} also regularity in time and a-priori estimates for the coefficients $u_j$ ($j = 1,\ldots,2N+2$) follow, that is, we have:
\begin{corollary}\label{cor:main}
In the situation of Theorem~\ref{th:main} it holds
\begin{equation}\label{exp_u_rigorous}
u(t,x) = u_1(t) x + u_2(t) x^2 + \ldots + u_{2N}(t) x^{2N} + R_N(t,x) x^{2N+1} \quad \mbox{as} \quad x \searrow 0,
\end{equation}
where the remainder $R_N(t,x)$ and the coefficients $u_j$ fulfill
\begin{enumerate}[(a)]
\item\label{num:rn} $R_N \in L^2\left((0,\infty);BC^0\left([0,\infty)\right)\right)$ with
\[
\int_0^\infty \sup_{x \ge 0} \verti{R_N(t,x)}^2 \d t \lesssim_{k,N,\delta} \vertiii{u^{(0)}}_{\mathrm{init}}^2,
\]
\item\label{num:c0} $\frac{\d^\ell u_j}{\d t^\ell} \in BC^\ell\left(\left[0,\infty\right)\right)$ for $j = 1,\ldots,2(N-\ell)$ and $\ell=0,\ldots,N$ such that
\[
\sup_{t \ge 0} \verti{\frac{\d^\ell u_j}{\d t^\ell}} \lesssim_{k,N,\delta} \vertiii{u^{(0)}}_{\mathrm{init}},
\]
\item\label{num:l2} $\frac{\d^\ell u_j}{\d t^\ell} \in L^2\left(\left[0,\infty\right)\right)$ for $j = 1,\ldots,2(N-\ell)+2$ and $\ell=0,\ldots,N$ such that
\[
\int_0^\infty \verti{\frac{\d^\ell u_j}{\d t^\ell}}^2 \d t \lesssim_{k,N,\delta} \vertiii{u^{(0)}}_{\mathrm{init}}^2.
\]
\end{enumerate}
\end{corollary}
\begin{proof}[Proof of Theorem~\ref{th:main}]
Denote by $\SSS: \, U_0 \times F \to U$ the linear solution operator constructed in Proposition~\ref{prop:mr_full}. Then the nonlinear problem \eqref{nonlinear_u} turns into the fixed-point equation
\begin{equation}\label{fixed_non_u}
u = \TT[u] := \SSS \left(u^{(0)}, \NN(u)\right),
\end{equation}
where $u^{(0)} \in U_0$ with $\vertiii{u^{(0)}}_{\mathrm{init}} \le \eps$ and $\eps > 0$ will be determined in what follows. For $0 < \omega \ll_{k,N,\delta} 1$ and $u, u^{(1)}, u^{(2)} \in U$ with $u^{(1)}_{|t = 0} = u^{(0)}$, $u^{(2)}_{|t = 0} = u^{(0)}$, $\vertiii{u}_{\mathrm{sol}} \le \omega$, $\vertiii{u^{(1)}}_{\mathrm{sol}} \le \omega$, and $\vertiii{u^{(2)}}_{\mathrm{sol}} \le \omega$, we have by Proposition~\ref{lem:est_coeff} and Proposition~\ref{prop:non_est}
\begin{equation}\label{main_self}
\vertiii{\TT[u]}_{\mathrm{sol}} \stackrel{\eqref{maxreg_full},\eqref{fixed_non_u}}{\lesssim_{k,N,\delta}} \vertiii{u^{(0)}}_{\mathrm{init}} + \vertiii{\NN(u)}_{\mathrm{rhs}} \stackrel{\eqref{estimate_nonlinear_main}}{\lesssim_{k,N,\delta}} \eps + \omega^2
\end{equation}
and
\begin{eqnarray}\nonumber
\vertiii{\TT\left[u^{(1)}\right] - \TT\left[u^{(2)}\right]}_{\mathrm{sol}} &\stackrel{\eqref{fixed_non_u}}{=}& \vertiii{\SSS \left(0,\NN\left(u^{(1)}\right) - \NN\left(u^{(2)}\right)\right)}_{\mathrm{sol}} \\
&\stackrel{\eqref{maxreg_full}}{\lesssim_{k,N,\delta}}& \vertiii{\NN\left(u^{(1)}\right) - \NN\left(u^{(2)}\right)}_{\mathrm{rhs}} \nonumber \\
&\stackrel{\eqref{estimate_nonlinear_main}}{\lesssim_{k,N,\delta}}& \omega \vertiii{u^{(1)} - u^{(2)}}_{\mathrm{sol}}. \label{main_contract}
\end{eqnarray}
Under the assumptions $0 < \eps \ll_{k,N,\delta} \omega \ll_{k,N,\delta} 1$, estimates \eqref{main_self} and \eqref{main_contract} imply existence of a unique fixed-point to \eqref{fixed_non_u}.

\medskip

Note that this does not imply the uniqueness assertion of Theorem~\ref{th:main} up to now as uniqueness of $u$ only holds in the ball $\{u: \, \vertiii{u}_{\mathrm{sol}} \le \omega\}$. Therefore observe that by approximation with smooth functions (cf.~Definition~\ref{def:function_spaces}) we need to have $\vertiii{u}_{\mathrm{sol},[0,T]} \to \vertiii{u^{(0)}}_{\mathrm{init}}$, where $\vertiii{u}_{\mathrm{sol},[0,T]}$ is the analogue of \eqref{norm_sol} on a fixed time interval $[0,T]$ instead of $[0,\infty)$. This implies in particular $\vertiii{u}_{\mathrm{sol},[0,T]} \le \omega$ if we choose $T > 0$ sufficiently small. As the above fixed-point argument also holds on a fixed time interval $[0,T]$, we infer that $u$ is unique on the interval $[0,T]$. A contradiction argument then yields global uniqueness.

\medskip

Finally, the trace estimate, Lemma~\ref{lem:trace}, allows to approximate $u$ by smooth and compactly supported functions in the norm $\vertiii{\cdot}_{\mathrm{sol}}$, which by \eqref{norm_sol} and \eqref{norm_init} controls also $\vertiii{\cdot}_{\mathrm{init}}$. Therefore, we necessarily have $\vertiii{u(t)}_{\mathrm{init}} \to 0$ as $t \to \infty$.
\end{proof}
\begin{proof}[Proof of Corollary~\ref{cor:main}]
As already noted, parts~\eqref{num:c0} and \eqref{num:l2} follow by combining estimates~\eqref{coeff_est_u_sup} and \eqref{coeff_est_u_l2} of Lemma~\ref{lem:est_coeff} with the a-priori estimate
\[
\vertiii{u}_{\mathrm{sol}} \lesssim_{k,N,\delta} \vertiii{u^{(0)}}_{\mathrm{init}}
\]
of Theorem~\ref{th:main}. 

\medskip

For proving the bound on $R_N$, observe that by a standard embedding
\begin{subequations}\label{est_rn}
\begin{equation}
\vertiii{u}_{\mathrm{sol}}^2 \stackrel{\eqref{norm_sol}}{\ge} \int_0^\infty \verti{u - \sum_{j = 1}^{2 N} u_j x^j}_{1,2 N + \delta}^2 \d t \stackrel{\eqref{exp_u_rigorous}}{\gtrsim_{k,N,\delta}} \int_0^\infty \verti{R_N}_{1,-1+ \delta}^2 \d t \gtrsim \int_0^\infty \sup_{x \ge 0} x^{2 - 2 \delta} \verti{R_N}^2 \d t
\end{equation}
and
\begin{eqnarray}\nonumber
\vertiii{u}_{\mathrm{sol}}^2 &\stackrel{\eqref{norm_sol}}{\ge}& \int_0^\infty \verti{u - \sum_{j = 1}^{2 N+1} u_j x^j}_{1,2 N + 1 + \delta}^2 \d t \stackrel{\eqref{exp_u_rigorous}}{\gtrsim_{k,N,\delta}} \int_0^\infty \verti{R_N - u_{2N+1}}_{1,\delta}^2 \d t \\
&\gtrsim& \int_0^\infty \sup_{x \ge 0} x^{-2\delta} \verti{R_N - u_{2N+1}}^2 \d t.
\end{eqnarray}
\end{subequations}
Take $\eta \in C^\infty\left([0,\infty)\right)$ with $\eta_{|[0,1]} \equiv 1$ and $\eta_{|[2,\infty)} \equiv 0$. Then almost everywhere in $\{t > 0\}$ we have
\begin{eqnarray}\nonumber
\sup_{x \ge 0} \verti{R_N} &\le& \sup_{x \ge 0} \verti{\eta R_N} + \sup_{x \ge 0} \verti{(1-\eta) R_N} \\
&\lesssim& \sup_{x \ge 0} \verti{\eta \left(R_N - u_{2N+1} x\right)} + \sup_{x \ge 0} \verti{(1-\eta) R_N} + \verti{u_{2N+1}} \nonumber\\
&\lesssim_\delta& \sup_{x \ge 0} x^{-\delta} \verti{\eta \left(R_N - u_{2N+1} x\right)} + \sup_{x \ge 0} x^{1-\delta} \verti{(1-\eta) R_N} + \verti{u_{2N+1}} \qquad \label{est_rn_2}
\end{eqnarray}
Combining \eqref{est_rn} and \eqref{est_rn_2} with \eqref{coeff_est_u_l2} of Lemma~\ref{lem:est_coeff}, we have proved part~\eqref{num:rn}.
\end{proof}

The goal of the remaining sections (cf.~\S\ref{sec:nonlin_struct}--\S\ref{sec:nonlin_main_est}) is to prove Proposition~\ref{prop:non_est}.

\subsection{Observations on the structure of the nonlinearity\label{sec:nonlin_struct}}
In $D$-notation we derive from \eqref{term_cond_non} and the identity $\partial_x^s = x^{-s} D (D-1) \ldots (D-s+1)$ that $\NN(u)$ is a linear combination (with constant coefficients) of terms of the form\footnote{Note that here and in what follows, derivatives only act on the factors separated by $\times$.}
\begin{equation}\label{term_non_D}
\begin{aligned}
x^{- 4} \times (1+v_x)^{-3-s^\prime} &\times \left(\prod_{\sigma = 0}^{s_0^\prime} (D-\sigma)\right) \left(x^3 + x^2\right) \times \left(\prod_{\sigma = 0}^{s_0-1} (D-\sigma)\right) \left(x^3 + x^2\right) \\
&\times \bigtimes_{j = 1}^n \left(\prod_{\sigma = 0}^{s_j-1} (D-\sigma)\right) v_x,
\end{aligned}
\end{equation}
where $s_0^\prime, s_0, s_1, \ldots, s_n, n$ meet \eqref{cond_coeff_non}, i.e.,
\[
s_0^\prime + s_0 + s_1 + \cdots + s_n = 3, \quad s_0 \le 1, \quad n \in \{2,\cdots,6\}, \quad s^\prime := \# \{s_j : \, j \ge 1 \mbox{ and } s_j \ge 1\},
\]
and $u \stackrel{\eqref{def_u}}{=} (3 x^2 + 2 x) v$. These nonlinear terms need to be estimated in the norm $\vertiii{\cdot}_{\mathrm{rhs}}$ (cf.~\eqref{norm_rhs}), that is, terms $f$ of the above form \eqref{term_non_D} need to be estimated in
\begin{eqnarray*}
\vertiii{f}_{\mathrm{rhs}}^2 &:=& \sum_{\left(\alpha,\tilde\ell,\tilde m\right) \in \II_{N-1,\delta}} \sum_{\tilde r = 0}^{\tilde m} \sup_{t \ge 0} \verti{\partial_t^{\tilde \ell} f - \sum_{j = 1}^{\floor\alpha+\tilde m+\tilde r} \frac{\d^{\tilde \ell} f_j}{\d t^{\tilde \ell}} x^j}_{k+ 4 \left(N - \tilde \ell\right) -3,\alpha+ \tilde m + \tilde r}^2 \nonumber \\
&& + \sum_{\left(\alpha,\ell,m\right) \in \JJ_{N,\delta}} \sum_{r = 0}^m \int_0^\infty \verti{\partial_t^\ell \underline f - \sum_{j = 1}^{\floor{\alpha}+m+r-1} \frac{\d^\ell \underline f_j}{\d t^\ell} x^j}_{k+4 (N - \ell)-1,\alpha+m+r -1}^2 \d t.
\end{eqnarray*}
Notice that the notation $\underline f = \frac{1}{x + 1} f$ has been introduced in \eqref{decomp_w}. Since we need to subtract the power series expansion of $f$ to arbitrary orders in the norm $\vertiii{\cdot}_{\mathrm{rhs}}$, it appears more favorable to expand $(1+v_x)^{-3-s^\prime}$ in a power series in $v_x$, so that $\NN(u)$ can be written as a convergent series of terms of the form
\begin{subequations}\label{term_non_power}
\begin{equation}\label{term_non_power_a}
c \, x^{-4} \times \left(\prod_{\sigma = 0}^{s_0^\prime} (D-\sigma)\right) \left(x^3 + x^2\right) \times \left(\prod_{\sigma = 0}^{s_0-1} (D-\sigma)\right) \left(x^3 + x^2\right) \times \bigtimes_{j = 1}^n \left(\prod_{\sigma = 0}^{s_j-1} (D-\sigma)\right) v_x,
\end{equation}
where $c = c(s_0^\prime, s_0, s_1, \ldots, s_n)$ is a real constant and $s_0^\prime, s_0, s_1, \ldots, s_n, n$ fulfill
\begin{equation}
s_0^\prime + s_0 + s_1 + \cdots + s_n = 3, \quad s_0 \le 1, \quad \mbox{and} \quad n \in \N \quad \mbox{with} \quad n \ge 2.
\end{equation}
\end{subequations}
%

\subsubsection{The structure of the $\sup$-parts\label{sec:sup_struct}}
What follows is tailored to the first line of \eqref{norm_rhs}. Without loss of generality assume $s_1 \ge \ldots \ge s_n$ (so that in particular $s_j = 0$ for $j \ge 4$). Then we notice that the factors
\[
x^{-1} \left(\prod_{\sigma = 0}^{s_0^\prime} (D-\sigma)\right) \left(x^3 + x^2\right) \quad \mbox{and} \quad x^{-1} \left(\prod_{\sigma = 0}^{s_0-1} (D-\sigma)\right) \left(x^3 + x^2\right)
\]
can be combined with
\[
\left(\prod_{\sigma = 0}^{s_1-1} (D-\sigma)\right) v_x \quad \mbox{and} \quad \left(\prod_{\sigma = 0}^{s_2-1} (D-\sigma)\right) v_x,
\]
respectively, using $u \stackrel{\eqref{def_u}}{=} (3 x^2 + 2 x) v$. To this end, observe
\begin{equation}\label{moeb1}
\begin{aligned}
& x^{-1} \left(\prod_{\sigma = 0}^{s_0^\prime} (D-\sigma)\right) (x^3 + x^2) \times \left(\prod_{\sigma = 0}^{s_1-1} (D-\sigma)\right) v_x \\
& \quad = \frac{\left(\prod_{\sigma = 0}^{s_0^\prime} (3-\sigma)\right) x + \left(\prod_{\sigma = 0}^{s_0^\prime} (2-\sigma)\right)}{3 x + 2} (3 x^2 + 2 x) \left(\prod_{\sigma = 0}^{s_1-1} (D-\sigma)\right) v_x,
\end{aligned}
\end{equation}
where $\frac{\left(\prod_{\sigma = 0}^{s_0^\prime} (3-\sigma)\right) x + \left(\prod_{\sigma = 0}^{s_0^\prime} (2-\sigma)\right)}{3 x + 2}$ is smooth function and any number of $D$-derivatives of it is bounded. Next we observe that the following operator identity holds true:
\begin{align*}
\left(D-\sigma\right) (3 x^2 + 2 x)^{-1} &= (3 x^2 + 2 x)^{-1} \left(D - \sigma - \frac{6 x^2 + 2 x}{3 x^2 + 2 x}\right) \\
&= (3 x^2 + 2 x)^{-1} \left(D - \sigma - 1 - \frac{3 x}{3 x + 2}\right)
\end{align*}
Combining this with \eqref{def_u} and \eqref{moeb1} and iterating the procedure, we note that
\begin{equation}\label{moeb2}
\begin{aligned}
(3 x^2 + 2 x) \left(\prod_{\sigma = 0}^{s_1-1} (D-\sigma)\right) v_x &= x^{-1} (3 x^2 + 2 x) \left(\prod_{\sigma = 0}^{s_1} (D-\sigma)\right) v \\
&= x^{-1} \left(\prod_{\sigma = 1}^{s_1+1} \left(D-\sigma - \frac{3 x}{3 x+2}\right)\right) u.
\end{aligned}
\end{equation}
Now notice that
\begin{equation}\label{moeb3}
D \frac{3 x}{3 x + 2} = \frac{3 x}{3 x + 2} \left(D+1\right) - \left(\frac{3 x}{3 x + 2}\right)^2,
\end{equation}
so that \eqref{term_non_power}--\eqref{moeb3} and an analogous reasoning for the second pair of factors $x^{-1} D (D-1) \ldots (D-s_0+1) \left(x^3 + x^2\right)$ and $D (D-1) \ldots (D - s_2+1) v_x$ yield that the nonlinearity $\NN(u)$ can be written as a convergent series
\begin{subequations}\label{term_non_power_2}
\begin{equation}\label{term_non_power_2a}
c(x) \, x^{- 4} \times \left(\prod_{\sigma = 1}^{s_1} (D-\sigma)\right) u \times \left(\prod_{\sigma = 1}^{s_2} (D-\sigma)\right)  u \times \bigtimes_{j = 3}^n \left(\prod_{\sigma = 0}^{s_j-1} (D-\sigma)\right) v_x,
\end{equation}
where $c(x) = c(x, s_1, \ldots, s_n)$ is a smooth real-valued function in $x \in (0,\infty)$ such that $\verti{D^j c}$ is bounded for every $j \in \N_0$, and $s_1, \ldots, s_n, n$ fulfill
\begin{equation}\label{term_non_power_2b}
s_1 + \ldots + s_n \le 5 \quad \mbox{and} \quad n \in \N \quad \mbox{with} \quad n \ge 2.
\end{equation}
Here we have renumbered the two factors with $u$ appearing, as in view of \eqref{moeb2} and \eqref{moeb3} we need to allow for cases in which no derivatives are acting on them. Note that in view of the operations in \eqref{moeb1} and \eqref{moeb2} not necessarily $s_1 \ge s_2 \ge \ldots \ge s_n$ but we still need to have
\begin{equation}\label{term_non_power_2c}
s_1 \le 4, \quad s_2 \le 2, \quad s_3 \le 1, \quad \mbox{and} \quad s_j = 0 \quad \mbox{if} \quad j \ge 4.
\end{equation}
\end{subequations}
Note that $\verti{D^j \left(x^{-\tau}c(x)\right)}$ in \eqref{term_non_power_2} is bounded, where we have introduced the notation
\begin{equation}\label{def_tau}
\tau := \max\left\{0,2-s_1\right\} + \max\left\{0,1-s_2\right\}.
\end{equation}
The fact that $x^\tau$ can be factored out from $c(x)$ has two reasons:
\begin{enumerate}[(a)]
\item\label{num:factor_1} If $s_1 + \ldots + s_n = 0$ in \eqref{term_non_power}, then the term in \eqref{term_non_power_a} is only non-vanishing if $\left(s_0^\prime,s_0\right) = (2,1)$. In this case, we have
\[
\frac{\left(\prod_{\sigma = 0}^{s_0^\prime} (3-\sigma)\right) x + \left(\prod_{\sigma = 0}^{s_0^\prime} (2-\sigma)\right)}{3 x + 2} = \frac{6 x}{3 x + 2}
\]
in \eqref{moeb1} and we can factor $x$ from $c(x)$.
\item\label{num:factor_2} Commuting $\frac{3 x}{3 x + 2}$ as done in \eqref{moeb3} leads to an additional factor $x$ in $c(x)$ as well.
\end{enumerate}
One may easily verify that the combination of \eqref{num:factor_1} and \eqref{num:factor_2} leads to the definition of $\tau$ as in \eqref{def_tau}. Since from \eqref{bc_lin} and sufficient regularity of $u$ at $x = 0$ we have that the nonlinear term in \eqref{term_non_power_2} behaves as $c(x) O\left(x^\nu\right)$ as $x \searrow 0$ where $\nu := \sum_{j = 1}^n s_j - 2$, this shows that indeed the boundary condition \eqref{bc_rhs} is satisfied individually for each term in \eqref{term_non_power_2a}. 

\subsubsection{A general argument for estimating products involving Taylor polynomials\label{sec:power_product}}
In the norm $\vertiii{\NN(u)}_{\mathrm{rhs}}^2$ (cf.~\eqref{norm_rhs}) we are dealing with the subtraction of the Taylor polynomial of a product to a certain order that we will denote by $J$ in this section. Suppose that $f, g: \, [0,\infty) \to \R$ are smooth and consider the term
\[
f g - \sum_{j = 0}^J (f g)_j x^j, \quad \mbox{where} \quad (f g)_j = \frac{1}{j !} \frac{\d^j}{\d x^j} (f g)_{|x = 0}.
\]
Observe
\begin{eqnarray}\nonumber
f g - \sum_{j = 0}^J (f g)_j x^j &= \left(f - \sum_{j^\prime = 0}^J f_{j^\prime} x^{j^\prime}\right) g + \left(\sum_{j^\prime = 0}^J f_{j^\prime} x^{j^\prime} g - \sum_{j = 0}^J (f g)_j x^j\right) \\
&= \left(f - \sum_{j^\prime = 0}^J f_{j^\prime} x^{j^\prime}\right) g + \sum_{j^\prime = 0}^J f_{j^\prime} x^{j^\prime} \left(g - \sum_{j^{\prime\prime} = 0}^{J-j^\prime} g_{j^{\prime\prime}} x^{j^{\prime\prime}}\right), \label{power_product}
\end{eqnarray}
where $\left(f - \sum_{j^\prime = 0}^J f_{j^\prime} x^{j^\prime}\right) g = O\left(x^{J+1}\right)$ and $f_{j^\prime} x^{j^\prime} \left(g - \sum_{j^{\prime\prime} = 0}^{J-j^\prime} g_{j^{\prime\prime}} x^{j^{\prime\prime}}\right) = O\left(x^{J+1}\right)$. This simple mechanism will be used in what follows to reduce the subtraction of the Taylor polynomial to the subtraction of the Taylor polynomial from the individual factors.

\subsubsection{The structure of the $\sup$-parts (continued)\label{sec:sup_struct_2}}
We continue with \S\ref{sec:sup_struct} and concentrate on the structure of the terms in the first line of the squared norm $\vertiii{\NN(u)}_1^2$ (cf.~\eqref{norm_rhs}). We first consider the expression
\[
\sup_{t \ge 0} \verti{\partial_t^{\tilde \ell} f - \sum_{j = 1}^{\floor\alpha+\tilde m+\tilde r} \frac{\d^{\tilde \ell} f_j}{\d t^{\tilde \ell}} x^j}_{k+4 \left(N - \tilde \ell\right) -3,\alpha+ \tilde m + \tilde r}^2,
\]
where $\alpha \in \left\{\delta, 1 + \delta\right\}$, $0 \le \tilde\ell + \tilde m \le N - 1 - \floor{\alpha}$, $0 \le \tilde r \le \tilde m$, and $f$ is to be replaced by \eqref{term_non_power_2a}. Distributing $\tilde \ell$ time derivatives onto the factors in \eqref{term_non_power_2a}, we need to estimate a term of the form
\begin{equation}\label{term_non_sup_0}
f := c(x) \, x^{- 4} \times \left(\prod_{\sigma = 1}^{s_1} (D-\sigma)\right) \partial_t^{\tilde \ell_1} u \times \left(\prod_{\sigma = 1}^{s_2} (D-\sigma)\right) \partial_t^{\tilde \ell_2} u \times \bigtimes_{j = 3}^n \left(\prod_{\sigma = 0}^{s_j-1} (D-\sigma)\right) \partial_t^{\tilde \ell_j} v_x
\end{equation}
in
\begin{equation}\label{term_sup_rhs}
\sup_{t \ge 0} \verti{f - \sum_{j = 1}^{\floor\alpha+\tilde m+\tilde r} f_j x^j}_{k+4 \left(N - \tilde \ell\right) -3,\alpha+ \tilde m + \tilde r}^2,
\end{equation}
with $\tilde\ell_1+\ldots+\tilde\ell_n = \tilde\ell$, and where $s_1,\ldots,s_n$ and $n$ meet \eqref{term_non_power_2b} and \eqref{term_non_power_2c}. Applying the product decomposition of \S\ref{sec:power_product} (cf.~\eqref{power_product}), we note that we can factor out the finite terms
\[
\vertii{c - \sum_{j = 0}^{\floor\alpha+\tilde m+\tilde r+4} c_j x^j}_{k+4 \left(N - \tilde \ell\right) - 3,\floor\alpha+ \tilde m + \tilde r + 4}^2 \quad \mbox{and} \quad \max_{j = \tau,\ldots,\floor\alpha+\tilde m+\tilde r+4} \verti{c_j}^2,
\]
where
\begin{equation}\label{def_sup_norm}
\vertii{w}_{\kappa,\gamma} := \max_{j = 0,\ldots,\kappa} \vertii{D^j w}_\gamma \quad \mbox{and} \quad \vertii{w}_\gamma := \sup_{x \in (0,\infty)} \verti{x^{-\gamma} w(x)}.
\end{equation}
Hence, in this case we are left with estimating a term of the form \eqref{term_non_sup_0} with $c(x) = x^\iota$ and $\iota = \tau,\ldots,\floor\alpha+\tilde m + \tilde r$ (where $\tau$ is defined as in \eqref{def_tau}) in \eqref{term_sup_rhs}, that is, we need to estimate terms of the form
\begin{subequations}\label{term_non_sup}
\begin{equation}\label{term_non_sup_a}
f := \left(\prod_{\sigma = 1}^{s_1} (D-\sigma)\right) \partial_t^{\tilde \ell_1} u \times \left(\prod_{\sigma = 1}^{s_2} (D-\sigma)\right) \partial_t^{\tilde \ell_2} u \times \bigtimes_{j = 3}^n \left(\prod_{\sigma = 0}^{s_j-1} (D-\sigma)\right) \partial_t^{\tilde \ell_j} v_x
\end{equation}
in
\begin{equation}\label{term_non_sup_b}
\sup_{t \ge 0} \verti{f - \sum_{j = 1}^{\floor\alpha+\tilde m+\tilde r+4-\iota} f_j x^j}_{k+ 4 \left(N - \tilde \ell\right) -3,\alpha+ \tilde m + \tilde r + 4 - \iota}^2,
\end{equation}
where $s_1 \le 4$, $s_2 \le 2$, $s_3 \le 1$, $s_j = 0$ for $j \ge 4$, and
\begin{equation}\label{term_non_sup_c}
s_1 + \cdots + s_n \le 3, \quad n \in \N \quad \mbox{with} \quad n \ge 2, \quad \mbox{and} \quad \iota = \tau,\ldots,\floor\alpha+\tilde m+\tilde r+4,
\end{equation}
\end{subequations}
where $\tau$ is defined in \eqref{def_tau}. In order to obtain a factorization of the norm in terms of $\vertiii{u}_{\mathrm{sol}}$, we will derive control of factors $\vertii{\partial_t^{\tilde \ell_j} v_x - \sum_{j = 0}^{\mu} \frac{\d^{\tilde\ell} (v_x)_j}{\d t^{\tilde\ell}} x^j}_{s_j,\mu}$ and $\verti{\frac{\d^{\tilde\ell} (v_x)_j}{\d t^{\tilde\ell}}}$ by $\vertiii{u}_{\mathrm{sol}}$. Note that for $\tilde\ell_j = 0$, $s_j = 0$, and $\mu = 0$ we have that $\vertii{\partial_t^{\tilde \ell_j} v_x}_{s_j,0} = \vertii{v_x}_{BC^0([0,\infty))}$ controls the Lipschitz constant of $v$. Smallness of the Lipschitz constant is crucial for invertibility of the von Mises transform \eqref{vonmises}. Lipschitz control will be the objective of \S\ref{sec:lip_control}. As a result, we can essentially reduce the situation to the bilinear case of two factors $u$.

\subsubsection{The structure of the $L^2$-parts\label{sec:l2_struct}}
We turn our attention to the second line of \eqref{norm_rhs}, i.e.,
\begin{equation}\label{typical_norm_l2}
\int_0^\infty \verti{\partial_t^\ell \underline f - \sum_{j = 0}^{\floor{\alpha}+m+r-1} \frac{\d^\ell \underline f_j}{\d t^\ell} x^j}_{k+4 (N - \ell)-1,\alpha+m+r-1}^2 \d t,
\end{equation}
where $\alpha \in \left\{-\frac 1 2 + \delta, \delta, \frac 1 2 + \delta, 1 + \delta\right\}$, $0 \le \ell+m \le N - \floor{\frac 12 + \alpha}$, and $0 \le r \le m$. Here, $\underline f$ needs to be replaced by the nonlinearity $\underline{\NN(u)}$, so that the terms \eqref{term_non_power} obtain an additional factor $(x+1)^{-1}$. In comparison to the reasoning in \S\ref{sec:sup_struct}, we may already distribute the $D$-derivatives of \eqref{typical_norm_l2} at this stage. Then we apply the transformation $u \stackrel{\eqref{def_u}}{=} (3 x^2 + 2 x) v$ to a factor $D v_x$ on which the maximal number of $D$-derivatives acts and otherwise take $v_x$ itself. By a reasoning analogous to the one in \S\ref{sec:sup_struct} leading to \eqref{term_non_power_2}, we conclude that the nonlinearity can be written as a convergent series of terms of the form
\begin{subequations}\label{typical_power_l2}
\begin{equation}\label{typical_power_l2_a}
c(x) x^{-2} \times D^{r_1} \left(\prod_{\sigma = 1}^{s_1} (D-\sigma)\right) u \times \bigtimes_{j = 2}^n D^{r_j} \left(\prod_{\sigma = 0}^{s_j-1} (D-\sigma) \right) v_x,
\end{equation}
where $c = c\left(x,s_1,\ldots,s_n, r_1, \ldots, r_n\right)$ is smooth in $x \in [0,\infty)$, $\verti{D^j c}$ is bounded on $[0,\infty)$, and
\begin{equation}\label{typical_power_l2_b}
s_1 + \ldots + s_n \le 4, \quad r_1 + \ldots + r_n \le k + 4 (N - \ell) - 1, \quad \mbox{and} \quad n \in \N \quad \mbox{with} \quad n \ge 2.
\end{equation}
By the above choice of $u$, we additionally have
\begin{equation}\label{typical_power_l2_d}
s_j + r_j \le \max\left\{3+\floor{\frac{k+4(N-\ell)-1}{2}},k+4(N-\ell)-1\right\} \quad \mbox{for} \quad j \ge 2.
\end{equation}
\end{subequations}
As done in the context of \eqref{term_non_power_2} (cf.~items \eqref{num:factor_1} and \eqref{num:factor_2} there), we infer that $\verti{D^j \left(x^{-\tau^\prime}c(x)\right)}$ is bounded, where
\begin{equation}\label{def_taup}
\tau^\prime := \max\left\{0,1-s_1\right\} + \max\left\{0,1-\max\{0,s_1-1\}-\sum_{j=2}^n s_j\right\}.
\end{equation}
Thus it can be verified that each term in \eqref{typical_power_l2_a} satisfies the boundary condition \eqref{bc_rhs} individually.

\medskip

Now factoring out $c$ as done in \S\ref{sec:sup_struct_2} and distributing time derivatives onto the individual factors in \eqref{typical_power_l2_a}, we are left with estimating a term of the form
\begin{subequations}\label{term_non_l2}
\begin{equation}\label{term_non_l2_a}
f := D^{r_1} \left(\prod_{\sigma = 1}^{s_1} (D-\sigma)\right) \partial_t^{\ell_1} u \times \bigtimes_{j = 2}^n D^{r_j} \left(\prod_{\sigma = 0}^{s_j-1} (D-\sigma)\right) \partial_t^{\ell_j} v_x
\end{equation}
in
\begin{equation}\label{term_non_l2_b}
\int_0^\infty \verti{f - \sum_{j = 0}^{\floor{\alpha}+m+r+1-\iota} f_j x^j}_{\alpha +m+r+1 - \iota}^2 \d t,
\end{equation}
where \eqref{typical_power_l2_b}, \eqref{typical_power_l2_d}, and
\begin{equation}\label{term_non_l2_c}
\iota = \tau^\prime,\ldots,\floor\alpha+m+r+1
\end{equation}
\end{subequations}
with $\tau^\prime$ as in \eqref{def_taup}, hold true.

\subsection{Lipschitz control and supremum bounds\label{sec:lip_control}}
The following two estimates are essential in order to factorize the nonlinearity.

\begin{lemma}\label{lem:lipschitz}
Throughout the proof, estimates depend on $k$, $N$, and $\delta$. Suppose $N \ge 1$, $k \ge 3$, and $0 < \delta < 1$. Then for $\tilde\ell \in \{0,\ldots,N\}$, $\kappa \in \left\{0,\ldots,k+4\left(N - \tilde\ell\right) - 1\right\}$, and $\mu \in \left\{0, \ldots, 2 \left(N - \tilde \ell -1\right)\right\}$ we have the estimate
\begin{equation}\label{est_lip}
\max\left\{ \sup_{t \ge 0} \vertii{\partial_t^{\tilde\ell} v_x - \sum_{j = 0}^{\mu} \frac{\d^{\tilde\ell} (v_x)_j}{\d t^{\tilde\ell}} x^j}_{\kappa,\mu}, \max_{j = 0, \ldots, \mu} \sup_{t \ge 0} \verti{\frac{\d^{\tilde\ell} (v_x)_j}{\d t^{\tilde\ell}}} \right\} \lesssim_{k,N,\delta} \vertiii{u}_{\mathrm{sol}}
\end{equation}
for every locally integrable $u: (0,\infty)^2 \to \R$, where $v \stackrel{\eqref{def_u}}{=} (3 x^2 + 2 x)^{-1} u$ and $\vertii{\cdot}_{k,\mu}$ was defined in \eqref{def_sup_norm}.
\end{lemma}
\begin{proof}
We first observe that for a smooth cut off $\eta : [0,\infty) \to \R$ with $\eta_{|[0,1]} = 1$ and $\eta_{|[2,\infty)} = 0$ we have by a standard embedding
\begin{eqnarray}\nonumber
\lefteqn{\vertii{\partial_t^{\tilde\ell} v_x - \sum_{j = 0}^{\mu} \frac{\d^{\tilde\ell} (v_x)_j}{\d t^{\tilde\ell}} x^j}_{\kappa,\mu}} \\
&\lesssim& \vertii{\eta \left(\partial_t^{\tilde\ell} v_x - \sum_{j = 0}^{\mu} \frac{\d^{\tilde\ell} (v_x)_j}{\d t^{\tilde\ell}} x^j\right)}_{\kappa,\mu} + \verti{\frac{\d^{\tilde\ell} (v_x)_\mu}{\d t^{\tilde\ell}}} + \vertii{(1-\eta) \left(\partial_t^{\tilde\ell} v_x - \sum_{j = 0}^{\mu-1} \frac{\d^{\tilde\ell} (v_x)_j}{\d t^{\tilde\ell}} x^j\right)}_{\kappa,\mu} \nonumber \\
&\lesssim& \verti{\partial_t^{\tilde\ell} v_x - \sum_{j = 0}^{\mu} \frac{\d^{\tilde\ell} (v_x)_j}{\d t^{\tilde\ell}} x^j}_{\kappa+1,\mu+\delta} + \verti{\frac{\d^{\tilde\ell} (v_x)_\mu}{\d t^{\tilde\ell}}} + \verti{\partial_t^{\tilde\ell} v_x - \sum_{j = 0}^{\mu-1} \frac{\d^{\tilde\ell} (v_x)_j}{\d t^{\tilde\ell}} x^j}_{\kappa+1,\mu-1+\delta} \nonumber \\
&\lesssim& \verti{\partial_t^{\tilde\ell} v - \sum_{j = 0}^{\mu+1} \frac{\d^{\tilde\ell} v_j}{\d t^{\tilde\ell}} x^j}_{\kappa+2,\mu+1+\delta} + \verti{\partial_t^{\tilde\ell} v - \sum_{j = 0}^\mu \frac{\d^{\tilde\ell} v_j}{\d t^{\tilde\ell}} x^j}_{\kappa+2,\mu+\delta}, \label{est_linf1}
\end{eqnarray}
where we have used $(v_x)_j = (j+1) v_{j+1}$ and
\begin{eqnarray}\nonumber
\verti{\frac{\d^{\tilde\ell} (v_x)_{\mu+1}}{\d t^{\tilde\ell}}}^2 &\lesssim& \int_{\frac 1 2}^2 \left(\frac{\d^{\tilde\ell} (v_x)_{\mu+1}}{\d t^{\tilde\ell}}\right)^2 \frac{\d x}{x} \\
&\lesssim& \int_{\frac 1 2}^2 x^{-2 \mu-2 \delta} \left(\partial_t^{\tilde\ell} v_x - \sum_{j = 0}^{\mu} \frac{\d^{\tilde\ell} (v_x)_j}{\d t^{\tilde\ell}} x^j\right)^2 \frac{\d x}{x} \nonumber \\
&& + \int_{\frac 1 2}^2 x^{2 -2 \mu + 2 \delta} \left(\partial_t^{\tilde\ell} v_x - \sum_{j = 0}^{\mu-1} \frac{\d^{\tilde\ell} (v_x)_j}{\d t^{\tilde\ell}} x^j\right)^2 \frac{\d x}{x} \nonumber\\
&\lesssim& \verti{\partial_t^{\tilde\ell} v_x - \sum_{j = 0}^{\mu} \frac{\d^{\tilde\ell} (v_x)_j}{\d t^{\tilde\ell}} x^j}_{\kappa+1,\mu+\delta}^2 + \verti{\partial_t^{\tilde\ell} v_x - \sum_{j = 0}^{\mu-1} \frac{\d^{\tilde\ell} (v_x)_j}{\d t^{\tilde\ell}} x^j}_{\kappa+1,\mu-1+\delta}^2. \qquad \label{bound_der_vx}
\end{eqnarray}
We note that $v \stackrel{\eqref{def_u}}{=} (3 x^2 + 2 x)^{-1} u$, so that we have due to \eqref{est_linf1}
\begin{equation}\label{est_linf1_1}
\vertii{\partial_t^{\tilde\ell} v_x - \sum_{j = 0}^{\mu} \frac{\d^{\tilde\ell} (v_x)_j}{\d t^{\tilde\ell}} x^j}_{\kappa,\mu} \lesssim \verti{\partial_t^{\tilde\ell} w - \sum_{j = 0}^{\mu+2} \frac{\d^{\tilde\ell} w_j}{\d t^{\tilde\ell}} x^j}_{\kappa+2,\mu+2+\delta} + \verti{\partial_t^{\tilde\ell} w - \sum_{j = 0}^{\mu+1} \frac{\d^{\tilde\ell} w_j}{\d t^{\tilde\ell}} x^j}_{\kappa+2,\mu+1+\delta}
\end{equation}
where $w := (3 x + 2)^{-1} u = x v$ (in particular $v_j = w_{j+1}$). Applying the decomposition principle in \S\ref{sec:power_product} to the product $w = (3 x + 2)^{-1} u$, we have
\begin{equation}\label{est_linf2}
\verti{\partial_t^{\tilde\ell} w - \sum_{j = 0}^{\mu+2} \frac{\d^{\tilde\ell} w_j}{\d t^{\tilde\ell}} x^j}_{\kappa+2,\mu+2+\delta} \lesssim \sum_{\varrho = 0}^{\mu+2} \verti{\partial_t^{\tilde\ell} u - \sum_{j = 1}^\varrho \frac{\d^{\tilde\ell} u_j}{\d t^{\tilde\ell}} x^j}_{\kappa+2,\varrho+\delta}.
\end{equation}
Similarly, we obtain
\begin{equation}\label{est_linf3}
\verti{\partial_t^{\tilde\ell} w - \sum_{j = 0}^{\mu+1} \frac{\d^{\tilde\ell} w_j}{\d t^{\tilde\ell}} x^j}_{\kappa+2,\mu+2-\delta} \lesssim \sum_{\varrho = 0}^{\mu+1} \verti{\partial_t^{\tilde\ell} u - \sum_{j = 1}^\varrho \frac{\d^{\tilde\ell} u_j}{\d t^{\tilde\ell}} x^j}_{\kappa+2,\varrho+\delta}.
\end{equation}
Estimates~\eqref{est_linf2} and \eqref{est_linf3} in \eqref{est_linf1_1} yield
\begin{equation}\label{est_linf_main}
\vertii{\partial_t^{\tilde\ell} v_x - \sum_{j = 0}^{\mu} \frac{\d^{\tilde\ell} (v_x)_j}{\d t^{\tilde\ell}} x^j}_{\kappa,\mu} \lesssim \sum_{\varrho = 0}^{\mu+2} \verti{\partial_t^{\tilde\ell} u - \sum_{j = 1}^\varrho \frac{\d^{\tilde\ell} u_j}{\d t^{\tilde\ell}} x^j}_{\kappa+2,\varrho+\delta}.
\end{equation}
Notice that the right-hand side of \eqref{est_linf_main} also bounds the sum
\[
\sum_{j = 0}^{\mu} \verti{\frac{\d^{\tilde\ell} (v_x)_j}{\d t^{\tilde\ell}}} \stackrel{\eqref{def_u}}{\lesssim} \sum_{j = 1}^{\mu+2} \verti{\frac{\d^{\tilde\ell} u_j}{\d t^{\tilde\ell}}}
\]
by \eqref{bound_der_vx}. Hence it remains to prove that the right-hand side of \eqref{est_linf_main} can be controlled by the norm $\vertiii{u}_{\mathrm{sol}}$. This is an immediate consequence of
\begin{equation}\label{norm_sol_sup}
\vertiii{u}_{\mathrm{sol}}^2 \stackrel{\eqref{norm_sol}}{\ge} \sum_{\left(\alpha,\ell,m\right) \in \II_{N,\delta}} \sum_{r = 0}^m  \sup_{t \ge 0} \verti{\partial_t^\ell u - \sum_{j = 1}^{\floor{\alpha}+m+r} \frac{\d^\ell u_j}{\d t^\ell} x^j}_{k+4 (2 N - \ell)+1,\alpha+m+r}^2
\end{equation}
and the restrictions on $\kappa$ and $\mu$.
\end{proof}

In a similar way, we can also obtain $\sup$-control for $u$ instead of $v$:
\begin{lemma}\label{lem:sup_control}
Suppose $N \ge 1$, $k \ge 3$, and $0 < \delta < \frac 1 2$. Then for $\ell \in \{0,\ldots,N\}$, $\kappa \in \left\{0,\ldots,k+4\left(N - \ell\right)\right\}$, and $\mu \in \left\{1, \ldots, 2 \left(N - \ell\right)\right\}$ we have the estimate
\begin{equation}\label{est_sup_control}
\max\left\{ \sup_{t \ge 0} \vertii{\partial_t^\ell u - \sum_{j = 0}^{\mu} \frac{\d^\ell u_j}{\d t^\ell} x^j}_{\kappa,\mu}, \max_{j = 0, \ldots, \mu} \sup_{t \ge 0} \verti{\frac{\d^\ell u_j}{\d t^\ell}}\right\} \lesssim_{k,N,\delta} \vertiii{u}_{\mathrm{sol}}
\end{equation}
for every locally integrable $u: (0,\infty)^2 \to \R$.
\end{lemma}
\begin{proof}
As in the proof of Lemma~\ref{lem:lipschitz} we can show
\begin{eqnarray*}
\lefteqn{\max\left\{\vertii{\partial_t^\ell u - \sum_{j = 1}^\mu \frac{\d^\ell u_j}{\d t^\ell} x^j}_{\kappa,\mu}, \verti{\frac{\d^\ell u_\mu}{\d t^\ell}}\right\}} \\
&\lesssim_\delta& \verti{\partial_t^\ell u - \sum_{j = 1}^\mu \frac{\d^\ell u_j}{\d t^\ell} x^j}_{\kappa+1,\mu+\delta} + \verti{\partial_t^\ell u - \sum_{j = 1}^{\mu-1} \frac{\d^\ell u_j}{\d t^\ell} x^j}_{\kappa+1,\mu-1+\delta}.
\end{eqnarray*}
The terms on the right-hand side of this estimate can be bounded by $\vertiii{u}_{\mathrm{sol}}$ due to the restrictions on $\kappa$ and $\mu$ (cf.~\eqref{norm_sol}).
\end{proof}

\subsection{Proof of the main nonlinear estimate\label{sec:nonlin_main_est}}
We treat lines 1 and 2 of the squared norm $\vertiii{f}_{\mathrm{rhs}}^2$ (cf.~\eqref{norm_rhs}) separately:
\subsubsection{$\sup$-control\label{sec:sup_control}}
We continue with the term \eqref{term_non_sup} and apply the decomposition principle of \S\ref{sec:power_product} (cf.~\eqref{power_product}) and estimate~\eqref{est_lip} of Lemma~\ref{lem:lipschitz} with $\mu = 0$ to the product
\[
\bigtimes_{j = 3}^n \left(\prod_{\sigma = 0}^{s_j-1} (D-\sigma)\right) \partial_t^{\tilde \ell_j} v_x
\]
of \eqref{term_non_sup_a}. Observe that $j \ge 3$ and therefore $s_j \le 1$ which implies that in the norm \eqref{term_non_sup_b} at most $k + 4 \left(N - \tilde\ell\right) - 2$ $D$-derivatives act on $\partial_t^{\tilde\ell} v_x$ so that in particular $\kappa \le k + 4 \left(N - \tilde\ell\right) - 1$ is satisfied in Lemma~\ref{lem:lipschitz}.

\medskip

This implies that there exists a constant $K > 0$ only depending on $N$, $k$, and $\delta$ such that for $f$ as in \eqref{term_non_sup_a} we have the bound
\begin{equation}\label{bound_sup_lip}
\begin{aligned}
& \max_{\iota = \tau,\ldots,\floor\alpha+\tilde m+\tilde r} \sup_{t \ge 0} \verti{f - \sum_{j = 1}^{\floor\alpha+\tilde m+\tilde r + 4 -\iota} f_j x^j}_{k+ 4 \left(N - \tilde \ell\right) -3,\alpha+ \tilde m + \tilde r + 4 - \iota}^2 \\
& \quad \le \left(K \vertiii{u}_{\mathrm{sol}}\right)^{2(n-2)}  \times \max_{\iota = \tau,\ldots,\floor\alpha+\tilde m+\tilde r+4}\sup_{t \ge 0} \verti{g - \sum_{j = 1}^{\floor\alpha+\tilde m+\tilde r + 4 -\iota} g_j x^j}_{k+ 4 \left(N - \tilde \ell\right) -3,\alpha+ \tilde m + \tilde r + 4 - \iota}^2,
\end{aligned}
\end{equation}
where
\begin{equation}\label{def_non_g}
g := \left(\prod_{\sigma = 1}^{s_1} (D-\sigma)\right) \partial_t^{\tilde \ell_1} u \times \left(\prod_{\sigma = 1}^{s_2} (D-\sigma)\right) \partial_t^{\tilde \ell_2} u
\end{equation}
and all other quantities are as in \eqref{term_non_sup} (in particular, we have $s_1 \le 4$ and $s_2 \le 2$).

\medskip

Next we apply the decomposition principle of \S\ref{sec:power_product} and estimate~\eqref{est_sup_control} of Lemma~\ref{lem:sup_control} with $\mu = 2$ to the second factor in \eqref{def_non_g}. Note that
\begin{enumerate}[(a)]
\item\label{num:apply_lem_1} because of $s_2 \le 2$ at most $k + 4 \left(N - \tilde \ell\right) - 1$ $D$-derivatives act on $\partial_t^{\tilde\ell_2} u$,
\item\label{num:apply_lem_2} we have $\tilde\ell_2 \le N-1$, i.e., $\mu \in \{0,\ldots,2\}$ is always allowed.
\end{enumerate}
Items \eqref{num:apply_lem_1} and \eqref{num:apply_lem_2} imply that Lemma~\ref{lem:sup_control} can be applied, so that we arrive at
\begin{equation}\label{est_g_sup}
\begin{aligned}
& \max_{\iota = \tau,\ldots,\floor\alpha+\tilde m+\tilde r} \sup_{t \ge 0} \verti{g - \sum_{j = 1}^{\floor\alpha+\tilde m+\tilde r + 4 -\iota} g_j x^j}_{k+ 4 \left(N - \tilde \ell\right) -3,\alpha+ \tilde m + \tilde r + 4 - \iota}^2 \\
& \quad \le \left(K \vertiii{u}_{\mathrm{sol}}\right)^2 \\
& \qquad \times \left(K \max_{\iota = \tau+s_2+1,\ldots,\floor\alpha+\tilde m+\tilde r+4} \sup_{t \ge 0} \verti{\partial_t^{\tilde\ell_1} u - \sum_{j = 1}^{\floor\alpha+\tilde m+\tilde r + 4 -\iota} \frac{\d^{\tilde\ell_1} u_j}{\d t^{\tilde\ell_1}} x^j}_{k+ 4 \left(N - \tilde \ell\right) + 1,\alpha+ \tilde m + \tilde r + 4 - \iota} \right)^2
\end{aligned}
\end{equation}
upon enlarging $K = K(k,N,\delta) > 0$. Now recalling that $\left(\alpha,\tilde\ell,\tilde m\right) \in \II_{N-1}$ (cf.~\eqref{index_sets}), $\tau+s_2+1 \stackrel{\eqref{def_tau}}{\ge} 2$, and
\[
\vertiii{u}_{\mathrm{sol}}^2 \stackrel{\eqref{norm_sol}}{\ge} \sum_{\left(\alpha,\ell,m\right) \in \II_{N,\delta}} \sum_{r = 0}^m  \sup_{t \ge 0} \verti{\partial_t^\ell u - \sum_{j = 1}^{\floor{\alpha}+m+r} \frac{\d^\ell u_j}{\d t^\ell} x^j}_{k+4 (N - \ell)+1,\alpha+m+r}^2,
\]
we notice that the last factor
\[
\max_{\iota = \tau+s_2+1,\ldots,\floor\alpha+\tilde m+\tilde r+4} \sup_{t \ge 0} \verti{\partial_t^{\tilde\ell_1} u - \sum_{j = 1}^{\floor\alpha+\tilde m+\tilde r + 4 -\iota} \frac{\d^{\tilde\ell_1} u_j}{\d t^{\tilde\ell_1}} x^j}_{k+ 4 \left(N - \tilde \ell\right) + 1,\alpha+ \tilde m + \tilde r + 4 - \iota}^2
\]
in \eqref{est_g_sup} can be bounded by $\vertiii{u}_{\mathrm{sol}}^2$ as well, so that \eqref{est_g_sup} upgrades to
\begin{equation}\label{bound_g_sup}
\max_{\iota = \tau,\ldots,\floor\alpha+\tilde m+\tilde r} \sup_{t \ge 0} \verti{g - \sum_{j = 1}^{\floor\alpha+\tilde m+\tilde r + 4 -\iota} g_j x^j}_{k+ 4 \left(N - \tilde \ell\right) -3,\alpha+ \tilde m + \tilde r + 4 - \iota}^2 \le  \left(K \vertiii{u}_{\mathrm{sol}}\right)^4.
\end{equation}
The combination of \eqref{bound_sup_lip} and \eqref{bound_g_sup} yields
\begin{equation}\label{bound_sup_normsol}
\max_{\iota = \tau,\ldots,\floor\alpha+\tilde m+\tilde r} \sup_{t \ge 0} \verti{f - \sum_{j = 1}^{\floor\alpha+\tilde m+\tilde r + 4 -\iota} f_j x^j}_{k+ 4 \left(N - \tilde \ell\right) -3,\alpha+ \tilde m + \tilde r + 4 - \iota}^2 \le \left(K \vertiii{u}_{\mathrm{sol}}\right)^{2 n},
\end{equation}
where $f$ is as in \eqref{term_non_sup_a}. It is crucial to note that $K$ is independent of $n$, so that the series expansion of \S\ref{sec:nonlin_struct} is indeed convergent. Thus we obtain
\begin{equation}\label{main_est_sup}
\begin{aligned}
& \sum_{\substack{\alpha \in \{\frac 12 \pm \delta, 1 \pm \delta\}\\0 \le \tilde \ell + \tilde m \le N-1}} \sup_{t \ge 0} \sum_{\tilde r = 0}^{\tilde m} \verti{\partial_t^{\tilde \ell} f - \sum_{j = 1}^{\floor\alpha+\tilde m+\tilde r} \frac{\d^{\tilde \ell} f_j}{\d t^{\tilde \ell}} x^j}_{k+4 \left(N - \tilde \ell\right) -3,\alpha+ \tilde m + \tilde r}^2 \\
& \quad \lesssim_{k,N,\delta} \left(\vertiii{u}_{\mathrm{sol}}^{2 (n-1)} + \vertii{\tilde u}_{\mathrm{sol}}^{2 (n-1)}\right) \vertiii{u - \tilde u}_{\mathrm{sol}}^2,
\end{aligned}
\end{equation}
where $f = \NN(u) - \NN\left(\tilde u\right)$, $\vertiii{u}_{\mathrm{sol}} \ll_{k,N,\delta} 1$, and $\vertiii{\tilde u}_{\mathrm{sol}} \ll_{k,N,\delta} 1$, in the particular case $\tilde u \equiv 0$. In order to establish \eqref{main_est_sup} for the case in which $\tilde u \ne 0$, we may follow the reasoning of \S\ref{sec:nonlin_struct} up to the multilinear expression \eqref{term_non_sup_a}. For an $n$-linear form $\MM$ observe
\[
\MM(u,\ldots,u) - \MM\left(\tilde u,\ldots, \tilde u\right) = \MM\left(u - \tilde u, u, \ldots, u\right) + \ldots + \MM\left(\tilde u, \ldots, \tilde u, u - \tilde u\right),
\]
where each of the $n$ factors on the right-hand side can be treated in the same way as before, leading to a single $\vertiii{u - \tilde u}_{\mathrm{sol}}^2$ instead of $\vertiii{u}_{\mathrm{sol}}^2$ or $\vertiii{\tilde u}_{\mathrm{sol}}^2$ in \eqref{bound_sup_normsol} and $K$ needs to be replaced by $K n^{\frac 1 n} = K e^{\frac{\log n}{n}}$. Since the latter is bounded in $n$, we arrive at estimate~\eqref{main_est_sup} for $\tilde u \ne 0$ as well.

\subsubsection{$L^2$-control\label{sec:l2_control}}
We continue at the end of \S\ref{sec:l2_struct} and first apply the decomposition principle of \S\ref{sec:power_product} to $u$ in \eqref{term_non_l2}, noting that $D^{r_1} \left(\prod_{\sigma = 1}^{s_1} (D-\sigma)\right) \partial_t^{\ell_1} u = O\left(x^{s_1+1}\right)$ as $x \searrow 0$. Applying the $L^2$-bound in time on the factors containing $u$ (taking $s_1 + 1 + r_1 \le k + 4 (N - \ell) + 3$ into account and employing estimate~\eqref{coeff_est_u_l2} of Lemma~\ref{lem:est_coeff} for the coefficients) and the $\sup$-bound in time on those containing $v$, we can estimate as follows
\begin{eqnarray} \nonumber
\lefteqn{\max_{\iota = \tau^\prime,\ldots,\floor\alpha+m+r+1} \int_0^\infty \verti{f - \sum_{j = 0}^{\floor{\alpha}+m+r-1} f_j x^j}_{\alpha+m+r+1 - \iota}^2 \d t} \\
 &\le& \left(K \vertiii{u}_{\mathrm{sol}}\right)^2 \times \max_{\iota = \tau^\prime+s_1+1,\ldots,\floor\alpha+m+r+1} \sup_{t \ge 0} \vertii{g - \sum_{j = 0}^{\floor\alpha+m+r+1-\iota} g_j x^j}_{\floor\alpha+m+r+1-\iota}^2, \qquad \label{main_non_est_l2_1}
\end{eqnarray}
where $g := \prod_{j = 2}^n D^{r_j} \left(\prod_{\sigma = 0}^{s_j-1} (D-\sigma)\right) \partial_t^{\ell_j} v_x$ and $K = K(k,N,\delta) > 0$. For the last factor of \eqref{main_non_est_l2_1} we may use estimate~\eqref{est_lip} of Lemma~\ref{lem:lipschitz} and obtain
\begin{equation}\label{main_non_est_l2_2}
\max_{\iota = \tau^\prime+s_1+1,\ldots,\floor\alpha+m+r+1} \sup_{t \ge 0} \vertii{g - \sum_{j = 0}^{\floor\alpha+m+r+1-\iota} g_j x^j}_{\floor\alpha+m+r+1-\iota}^2 \le \left(K \vertiii{u}_{\mathrm{sol}}\right)^{2 (N-1)}
\end{equation}
upon enlarging $K = K(k,N,\delta)$. Note that indeed applicability of Lemma~\ref{lem:lipschitz} is guaranteed, because $\floor\alpha+m+r+1-\iota \le 2 (N-\ell-1)$ (cf.~\eqref{index_sets} and \eqref{def_taup}) and due to \eqref{typical_power_l2_d}, $\kappa$ in Lemma~\ref{lem:lipschitz} fulfills the bound
\[
\kappa \le  \max\left\{3+\floor{\frac{k+4(N-\ell)-1}{2}},k+4(N-\ell)-1\right\} \le k + 4 (N - \ell) -1,
\]
where we need to assume $k \ge 3$. Gathering \eqref{main_non_est_l2_1} and \eqref{main_non_est_l2_2}, we have
\begin{equation}\label{main_non_est_l2}
\max_{\iota = \tau^\prime,\ldots,\floor\alpha+m+r+1} \int_0^\infty \verti{f - \sum_{j = 0}^{\floor{\alpha}+m+r+1-\iota} f_j x^j}_{\alpha+m+r+1-\iota}^2 \d t \le (K \vertiii{u}_{\mathrm{sol}})^{2 n},
\end{equation}
Since the constant $K$ in estimate \eqref{main_non_est_l2} is independent of $n$, the series expansion of \S\ref{sec:nonlin_struct} is convergent and we have
\begin{equation}\label{main_est_l2}
\begin{aligned}
& \sum_{\substack{\alpha \in \{\pm \delta, \frac 12 \pm \delta, 1 \pm \delta\}\\0 \le \ell+m \le N}} \int_0^\infty \sum_{r = 0}^m  \verti{\partial_t^\ell \underline f - \sum_{j = 0}^{\floor{\alpha}+m+r-1} \frac{\d^\ell \underline f_j}{\d t^\ell} x^j}_{k+4 (2 N - \ell)-1,\alpha+m+r-1}^2 \d t \\
& \quad \lesssim_{k,N,\delta} \left(\vertiii{u}_{\mathrm{sol}}^{2 (n-1)} + \vertii{\tilde u}_{\mathrm{sol}}^{2 (n-1)}\right) \vertiii{u - \tilde u}_{\mathrm{sol}}^2,
\end{aligned}
\end{equation}
where $f = \NN(u) - \NN\left(\tilde u\right)$, $\vertiii{u}_{\mathrm{sol}} \ll_{k,N,\delta} 1$, and $\vertiii{\tilde u}_{\mathrm{sol}} \ll_{k,N,\delta} 1$, in the particular case $\tilde u \equiv 0$. Passing from this particular situation to the case of general $\tilde u$  follows the lines of the respective reasoning in \S\ref{sec:sup_control}, so that we do not need to discuss this once more.

\subsubsection{Conclusion}
The combination of \eqref{main_est_sup} and \eqref{main_est_l2} together with the definition of the norm $\vertiii{\cdot}_{\mathrm{rhs}}$ in \eqref{norm_rhs} proves estimate~\eqref{estimate_nonlinear_main} of Proposition~\ref{prop:non_est} and concludes the proof.

\bibliography{Gnann_Ibrahim_Masmoudi_2018}
\bibliographystyle{plain}

\end{document}